\documentclass[reqno,a4paper,11pt]{amsart}
\usepackage[T1]{fontenc}
\usepackage{mathrsfs}
\usepackage{libertine}
\usepackage[libertine]{newtxmath}
\usepackage[utf8]{inputenc}
\usepackage[english]{babel}
\usepackage{url}
\usepackage[colorlinks=true]{hyperref}
\hypersetup{linkcolor=blue, citecolor=blue, filecolor=black, urlcolor=blue}%MidnightBlue}
\usepackage{comment}

\usepackage{cancel}

\newcommand{\crosses}[1]{%
	\ifcase#1\relax
	\or
	\rslash\or
	\rslash\mskip-5.5mu\rslash\or
	\rslash\mskip-5.5mu\rslash\mskip-5.5mu\rslash%
	\fi
}
\newcommand{\rslash}{\raisebox{.15ex}{/}}
\usepackage[usenames,dvipsnames]{xcolor}
\usepackage{color}
\usepackage{geometry}
\usepackage[all]{xy}
\usepackage{enumitem}

\usepackage{slashed}
\usepackage{bbm}
\usepackage{cancel}

\usepackage{cleveref}
\usepackage{tikz}
\usepackage{tikz-cd}

\usepackage{xargs}  % Use more than one optional parameter in a new commands
\usepackage{soul}

\usepackage{amsmath}
\usepackage{amssymb,graphicx}
\usepackage{amsthm}
\usepackage{latexsym}
\usepackage{amsfonts}
\usepackage{bbm}
\usepackage{mathrsfs}
\usepackage{cases}

\usepackage{etex}
\calclayout

\numberwithin{equation}{section}

\theoremstyle{plain}
\newtheorem{lemma}{Lemma}[section]
\newtheorem{proposition}[lemma]{Proposition}
\newtheorem{proposition/definition}[lemma]{Proposition/Definition}
\newtheorem{theorem}[lemma]{Theorem}
\newtheorem{corollary}[lemma]{Corollary}

\theoremstyle{definition}
\newtheorem{definition}[lemma]{Definition}
\newtheorem{remark}[lemma]{Remark}
\newtheorem{example}[lemma]{Example}

\DeclareRobustCommand{\SkipTocEntry}[5]{}

\DeclareMathOperator{\id}{id}
\DeclareMathOperator{\im}{im}
\DeclareMathOperator{\rank}{rank}

\DeclareMathOperator{\Ad}{Ad}

\newcommand{\J}{\mathcal{J}}
\newcommand{\moment}{\bbJ}

\newcommand{\Mod}{\ \text{mod}\ }

\DeclareRobustCommand{\bigO}{%
	\text{\usefont{OMS}{cmsy}{m}{n}O}%
}

\newcommand{\calC}{\mathcal{C}}
\newcommand{\calD}{\mathcal{D}}

\newcommand{\calG}{\mathcal{G}}
\newcommand{\calH}{\mathcal{H}}

\newcommand{\calJ}{\mathcal{J}}
\newcommand{\calK}{\mathcal{K}}
\newcommand{\calL}{\mathcal{L}}

\newcommand{\calP}{\mathcal{P}}

\newcommand{\calR}{\mathcal{R}}
\newcommand{\calS}{\mathcal{S}}

\newcommand{\calX}{\mathcal{X}}

\newcommand{\bbC}{\mathbb{C}}
\newcommand{\bbD}{\mathbb{D}}

\newcommand{\bbJ}{\mathbb{J}}

\newcommand{\bbN}{\mathbb{N}}

\newcommand{\bbP}{\mathbb{P}}

\newcommand{\bbR}{\mathbb{R}}
\newcommand{\bbS}{\mathbb{S}}

\newcommand{\frakX}{\mathfrak{X}}

\newcommand{\frakg}{\mathfrak{g}}

\newcommand{\rmR}{\mathrm{R}}

\newcommand{\rmc}{\mathrm{c}}
\newcommand{\rmd}{\mathrm{d}}

\newcommand{\rmh}{\mathrm{h}}

\newcommand{\ldsb}{[\![}
\newcommand{\rdsb}{]\!]}
\newcommand{\ldab}{\langle\!\langle}
\newcommand{\rdab}{\rangle\!\rangle}
\renewcommand{\theta}{\vartheta}

\allowdisplaybreaks

\newcounter{proof:lem:centralizers}

\title{Contact Dual Pairs}

\author[A. M. Blaga]{Adara Monica Blaga}
\address{Department of Mathematics, West University of Timi\c soara, Bld.~V.~P\^arvan 4, 300223, Romania.}
\email{\href{mailto:adarablaga@yahoo.com}{adarablaga@yahoo.com}}

\author[M. A. Salazar]{Maria Amelia Salazar}
\address{Departamento de Matem\'atica Aplicada, Universidade Federal Fluminense, Niter\'oi - RJ, Brazil.}
\email{\href{mailto:mariasalazar@id.uff.br}{mariasalazar@id.uff.br}}

\author[A. G. Tortorella]{Alfonso Giuseppe Tortorella}
\address{Department of Mathematics, KU Leuven, Celestijnenlaan 200B - 3001 Leuven, Belgium.}
\email{\href{mailto:alfonsogiuseppe.tortorella@kuleuven.be}{alfonsogiuseppe.tortorella@kuleuven.be}}

\author[C. Vizman]{Cornelia Vizman}
\address{Department of Mathematics, West University of Timi\c soara, Bld.~V.~P\^arvan 4, 300223, Romania.}
\email{\href{mailto:cornelia.vizman@e-uvt.ro}{cornelia.vizman@e-uvt.ro}}



\begin{document}

\begin{abstract}
	We introduce and study the notion of contact dual pair adopting a line bundle approach to contact and Jacobi geometry.
	A contact dual pair is a pair of Jacobi morphisms defined on the same contact manifold and satisfying a certain orthogonality condition.
	Contact groupoids and contact reduction are the main sources of examples.
	Among other properties, we prove the Characteristic Leaf Correspondence Theorem for contact dual pairs which parallels the analogous result of Weinstein for symplectic dual pairs.
\end{abstract}

\maketitle

\tableofcontents

\section{Introduction}
\label{sec:intro}

Tracing back to S.~Lie~\cite{lie1890theorie}, the notion of dual pair of Poisson maps (symplectic dual pairs) has its modern origin in the works of Weinstein~\cite{We83}, on the local structure of Poisson manifolds, and Howe~\cite{howe1989remarks}, on representation theory in connection with quantum mechanics.
Symplectic dual pairs are  important in Poisson geometry and geometric mechanics.
For instance, they naturally pop up in relation with: Morita equivalence of Poisson manifolds, bifoliations and superintegrable Hamiltonian systems, as well as moment maps and symplectic reduction.

Jacobi structures, introduced independently by Kirillov~\cite{kirillov} and Lichnerowicz~\cite{lichnerowicz1978jacobi}, encompass, generalizing and unifying, several geometric structures, like
Poisson,  (locally conformal) symplectic, and (generically non-coorientable) contact.
Following~\cite{marle} a Jacobi bundle is a line bundle $L\to M$ equipped with a Jacobi structure $\{-,-\}$, i.e.~a Lie bracket on $\Gamma(L)$ which additionally is a differential operator (DO) in each entry.
Then a Jacobi manifold is a manifold with a Jacobi bundle over it.
In this paper, inspired by~\cite{crainic2015jacobi}, we adopt the line bundle approach to contact and Jacobi geometry.
The conceptual backgrounds of this approach are represented by the gauge algebroid $DL$ of a line bundle $L\to M$ and the associated der-complex of $L$-valued Atiyah forms and graded Lie algebra of multi-differential operators (cf.~Appendix~\ref{sec:gauge_algebroid}).

There exists a close relation between Poisson/symplectic and Jacobi/contact geometry.
On one hand, contact structures are viewed as the odd-dimensional analogue of symplectic structures and, on the other hand, Poisson structures can be seen as a contravariant generalization of symplectic structures, with the latter being the non-degenerate case of the former.
In addition, following Lichnerowicz's philosophy, one can also view Jacobi structures as a ``contravariant'' generalization of contact structures, with the latter being the non-degenerate case of the former.
This close relation between Poisson/symplectic and Jacobi/contact geometry makes it pretty natural to wonder what, if any, is the contact analogue of symplectic dual pairs.
This paper aims at introducing, on conceptually well-grounded basis, and systematically investigating the concept of contact dual pairs.

%Symplectic dual pairs
{\bf Dual pairs in symplectic and Poisson geometry.}
The symplectic dual pairs are the source of inspiration in handling contact dual pairs.
Here, following~\cite{ortega2013momentum}, we recollect their main properties.
Let $(M_i,\{-,-\}_i)$ be Poisson manifolds, for $i=1,2$, and $(M,\omega)$ be a symplectic manifold, with $\{-,-\}$ the corresponding non-degenerate Poisson structure on $M$.
A pair of Poisson maps
\begin{equation}
\label{eq:SDP}
\begin{tikzcd}
(M_1,\{-,-\}_1)&(M,\omega)\arrow[l, swap, "\varphi_1"]\arrow[r, "\varphi_2"]&(M_2,\{-,-\}_2)
\end{tikzcd}
\end{equation}
is called a \emph{(Lie--Weinstein) symplectic dual pair}, or simply \emph{symplectic dual pairs}, if the distributions $\ker T\varphi_1$ and $\ker T\varphi_2$ are the orthogonal complement of each other w.r.t.~$\omega$.
Diagram~\eqref{eq:SDP} forms a \emph{Howe symplectic dual pair} if the Poisson subalgebras $\varphi_1^\ast C^\infty(M_1)$ and $\varphi_2^\ast C^\infty(M_2)$ of $C^\infty(M)$ are the centralizer of each other w.r.t.~$\{-,-\}$.
Unlike the Lie--Weinstein definition, which is a local condition, the Howe definition has a global character.
The relation between these two notions has been investigated in~\cite{MOR}.

If two Poisson manifolds fit into a symplectic dual pair, their local structures are very closely related.
Indeed, let us assume that, in diagram~\eqref{eq:SDP}, the Poisson maps $\varphi_1$ and $\varphi_2$ are surjective submersions (i.e.~the dual pair is \emph{full}) with connected fibers.
Then the relation $\calS_2=\varphi_2(\varphi_1^{-1}(\calS_1))$ establishes a 1-1 correspondence between the symplectic leaves $\calS_1$ of $M_1$ and $\calS_2$ of $M_2$.  
Further, if $\omega_1$ and $\omega_2$ are the symplectic structures inherited by $\calS_1$ and $\calS_2$ respectively, then they are related as follows
\begin{equation*}
\iota_\calK^\ast\omega=\varphi_1|_\calK^\ast\omega_1+\varphi_2|_\calK^\ast\omega_2,
\end{equation*}
where $\calK:=\varphi_1^{-1}(\calS_1)=\varphi_2^{-1}(\calS_2)$, and $\iota_\calK:\calK\to M$ is the inclusion.
Finally, the transverse Poisson structures to $\calS_1$ and $\calS_2$ turn out to be anti-isomorphic.
These properties of symplectic dual pairs lead to introduce and investigate the \emph{Morita equivalence} of Poisson manifolds~\cite{Xu}.

One source of examples are the symplectic groupoids.
%Indeed, the source and the target map of any symplectic groupoid form a symplectic dual pair.
Additionally, symplectic dual pairs also naturally emerge from reduction of symplectic manifolds with symmetries.
Let us consider a Lie group $G$ acting freely, properly, and by symplectomorphisms on $(M,\omega)$.
Assume further that the action is Hamiltonian with $\Ad^\ast$ equivariant moment map $\moment:M\to\frakg^\ast$.
Then one gets the symplectic dual pair
\begin{equation*}
\begin{tikzcd}
(M/G,\{-,-\}_{M/G})&(M,\omega)\arrow[l, swap, "q"]\arrow[r, "\bbJ"]&(\frakg^\ast,\{-,-\}_+),
\end{tikzcd}
\end{equation*}
where $\{-,-\}_+$ is the $+$-Lie-Poisson bracket on $\frakg^\ast$ and $\{-,-\}_{M/G}$ is the quotient Poisson structure on $M/G$.

Finally, we point out that the notion of symplectic dual pairs generalizes to the setting of Dirac structures.
In particular, this generalization, introduced in~\cite{BR}, has allowed to obtain alternative proofs of the normal form theorem around Dirac transversals and the existence of symplectic realizations~\cite{FM}.

{\bf Dual pairs in contact and Jacobi geometry}.
%Definition of Lie--Weinstein contact dual pairs
Let $(M_i,L_i,\{-,-\}_i)$ be Jacobi manifolds, with $i=1,2$, and $(M,\calH)$ be a contact manifold.
Set $L:=TM/\calH$ and denote by $\theta$ the corresponding $L$-valued contact form, by $\rmc_\calH$ the associated curvature form, and by $\{-,-\}$ the corresponding non-degenerate Jacobi structure on $L\to M$ (see Section~\ref{sec:review} for a brief review of contact and Jacobi geometry).
Then a \emph{(Lie--Weinstein) contact dual pair} (or simply \emph{contact dual pair}) is a pair of Jacobi morphisms
\begin{equation}
\label{eq:intro:CDP}
\begin{tikzcd}
(M_1,L_1,\{-,-\}_1)&(M,\calH)\arrow[l, swap, "\Phi_1"]\arrow[r, "\Phi_2"]&(M_2,L_2,\{-,-\}_2),
\end{tikzcd}
\end{equation}
with underlying maps $\!\!\begin{tikzcd}M_1&M\arrow[l, swap, "\varphi_1"]\arrow[r, "\varphi_2"]&M_2\end{tikzcd}\!\!\!$, such that the following three conditions hold:
\begin{enumerate}[label=(\roman*)]
	\item the contact distribution $\calH$ is transverse to both $\ker T\varphi_1$ and $\ker T\varphi_2$,
	\item the pull-back sections $\Phi_1^\ast\lambda_1$ and $\Phi_2^\ast\lambda_2$ Jacobi commute, for all $\lambda_1\in\Gamma(L_1)$ and $\lambda_2\in\Gamma(L_2)$,
	\item $\calH\cap\ker T\varphi_1$ and $\calH\cap\ker T\varphi_2$
	are the orthogonal complement of each other w.r.t.~$\rmc_\calH$.
\end{enumerate}
The main motivating examples are contact groupoids.
Indeed, the very definition is modelled so that the source and the target map of any contact groupoid form a contact dual pair.
Nevertheless, a more compact and geometrically insightful description of contact dual pairs is obtained by means of the interpretation of the $L$-valued contact form $\theta$ as an $L$-valued symplectic Atiyah form $\varpi$.
Indeed, Proposition~\ref{prop:LW_CDP_Atiyah_forms} shows that diagram~\eqref{eq:intro:CDP} is a contact dual pair iff the kernels of the induced gauge algebroid morphisms $D\Phi_1:DL\to DL_1$ and $D\Phi_2:DL\to DL_2$ are the orthogonal complement of each other w.r.t.~$\varpi$, i.e.
\begin{equation}
\label{eq:intro:CDP_orthogonality_condition}
\ker D\Phi_1=(\ker D\Phi_2)^{\perp\varpi}.
\end{equation}
This characterization of contact dual pairs immediately leads to other two equivalent descriptions.
First, Proposition~\ref{prop:CDP_Dirac-Jacobi} rephrases the orthogonality condition~\eqref{eq:intro:CDP_orthogonality_condition} into the language of Dirac--Jacobi geometry.
Even though this rephrasing already plays a crucial role in this paper (e.g.~in the proof of Theorem~\ref{theor:transverse_structure}), the generalization of contact dual pairs to the setting of Dirac--Jacobi structures will only be addressed in a separate short note~\cite{schnitzer2019weakdualpairs}.
Second, Proposition~\ref{prop:CDPs=homogeneous_SDPs} establishes a 1-1 correspondence %, up to natural transformations,
between the contact dual pairs and the so-called \emph{homogeneous symplectic dual pairs}.

Diagram~\eqref{eq:intro:CDP} is a \emph{Howe contact dual pair} if the Lie subalgebras $\Phi_1^\ast\Gamma(L_1)$ and $\Phi_2^\ast\Gamma(L_2)$ of $\Gamma(L)$ are the centralizer of each other w.r.t.~$\{-,-\}$.
In contrast to the Lie--Weinstein definition, which is a local condition, the Howe definition is a global condition.
In parallel with the analogous results for symplectic dual pairs~\cite{MOR}, Proposition~\ref{prop:relation_Howe_Weinstein} studies the non-trivial relation between these two notions of dual pair.

The main result of the paper is the Characteristic Leaf Correspondence, according to which the local structures of two Jacobi manifolds fitting into a contact dual pair are very closely related.
It consists of three parts, in close analogy to Weinstein's results for symplectic dual pairs~\cite{We83}.
Let us assume that the underlying maps $\varphi_1$ and $\varphi_2$ in diagram~\eqref{eq:SDP} are surjective submersions (i.e.~the dual pair is \emph{full}) with connected fibers.
Then the relation $\calS_2=\varphi_2(\varphi_1^{-1}(\calS_1))$ establishes a 1-1 correspondence between the characteristic leaves $\calS_1$ of $M_1$ and $\calS_2$ of $M_2$ (see Theorem~\ref{theor:LeafCorrespondenceI}).
Further, Theorem~\ref{theor:LeafCorrespondenceII} proves that $\calS_1$ and $\calS_2$ are either both contact or both l.c.s., and describes the relation between their inherited (transitive Jacobi) structures.
Finally, Theorem~\ref{theor:transverse_structure} shows that the transverse structures to $\calS_1$ and to $\calS_2$ are anti-isomorphic.
These properties of contact dual pairs seem to suggest the introduction and the investigation of Morita equivalence for Jacobi manifolds.
This suggestive idea will be pursued by the authors in a future work.

In addition to contact groupoids, as pointed out by Theorem~\ref{theor:CDP_contact groupoid_action}, another source of examples is represented by contact reduction.
In this paper we consider contact actions of contact groupoids on contact manifolds.
As a special case, let us consider a Lie group $G$ acting freely, properly, and by contactomorphisms on $(M,\calH)$.
Assume further that the orbits are transverse to $\calH$.
Then one gets a \emph{moment map} $\moment:M\to\bbP(\frakg^\ast)$ (also called \emph{Jacobi moment map} in~\cite[Def.~2.27]{SaSe}) and the following contact dual pair
\begin{equation*}
\begin{tikzcd}
(M/G,L/G,\{-,-\}_{M/G})&(M,\calH)\arrow[l, swap, "q"]\arrow[r, "\bbJ"]&(\bbP(\frakg^\ast),\bigO_{\bbP(\frakg^\ast)}(1),\{-,-\}_{\bbP(\frakg^\ast)}),
\end{tikzcd}
\end{equation*}
where $(\bbP(\frakg^\ast),\bigO_{\bbP(\frakg^\ast)}(1),\{-,-\}_{\bbP(\frakg^\ast)})$ is the Jacobi manifold of Example~\ref{projective}, and the quotient line bundle $L/G\to M/G$ is equipped with the quotient Jacobi structure $\{-,-\}_{M/G}$.

%Plan of the paper
{\bf Structure of the paper.}
Section~\ref{sec:review} gives a brief review of contact and Jacobi geometry.
Section~\ref{sec:CDPs} introduces contact dual pairs with motivating examples coming from contact groupoids.
Section~\ref{sec:Properties} studies the first properties of contact dual pairs, like the alternative equivalent description in terms of Dirac--Jacobi geometry, the identification with homogeneous symplectic dual pairs and the relation between contact dual pairs and Howe contact dual pairs.
Section~\ref{sec:CharacteristicLeafCorrespondence} discusses the main result of the paper, namely the Characteristic Leaf Correspondence.
Section~\ref{sec:contact_groupoid_action_CDP} describes in detail how contact dual pairs naturally emerge from the reduction of contact manifolds with symmetries.
The paper also contains two appendices.
Appendix~\ref{sec:differential_operators} recalls basic facts about differential operators (DOs) and the conceptual backgrounds of the line bundle approach to contact and Jacobi geometry.
Appendix~\ref{sec:homogeneous_symplectic/Poisson} summarizes the alternative, but equivalent, approach to contact and Jacobi geometry inspired by the ``symplectization/Poissonization trick''.

\section{A Review of Contact and Jacobi Geometry}
\label{sec:review}

This section gives a very brief review of (the line bundle approach to) contact and Jacobi geometry.
It recalls the (possibly non-standard) terminology and the background material that we will systematically adopt and use later on in the paper.
In doing so, we will freely use the language of jets and differential operators (DOs) on line bundles as summarized in Appendix~\ref{sec:differential_operators}.
In addition to~\cite{kirillov} and~\cite{lichnerowicz1978jacobi}, our main references for this material are~\cite{crainic2015jacobi,le2017jacobi,LOTV,SaSe,tortorella2017phdthesis}.

\subsection{Contact Manifolds}
\label{sec:contact_structures}

For any manifold $M$, there is a canonical 1-1 correspondence between:
\begin{enumerate}[label=(\arabic*)]
	\item \emph{pre-contact distributions} $\calH$ on $M$, i.e.~hyperplane distributions $\calH\subset TM$,
	\item \emph{pre-contact forms} $\theta$ on $M$, i.e.~nowhere zero $1$-forms $\theta\in\Omega^1(M;L)$ with values in some line bundle $L\to M$, unique up to line bundle isomorphisms.
\end{enumerate}
Clearly in one direction this correspondence works by setting $L:=TM/\calH$ and $\theta(X)=X\Mod\calH$, for all $X\in\frakX(M)$, and in the opposite direction it works by setting $\calH=\ker\theta$.

Let $\calH$ be a pre-contact distribution on $M$ and let $\theta$ be a corresponding $L$-valued pre-contact form on $M$.
The associated \emph{curvature $2$-form} $\rmc_\calH\in\Gamma(\wedge^2\calH^\ast\otimes L)$ is defined by
$$\rmc_\calH(X,Y):=\theta[X,Y], \text{ for all }X,Y\in\Gamma(\calH).$$
The curvature form measures the failure of $\calH$ to be integrable.
Indeed, by the Frobenius Theorem, $\calH$ is integrable iff $\rmc_\calH=0$.
So $\calH$ is called \emph{maximally non-integrable} if $\rmc_\calH$ is non-degenerate, i.e.~the vector bundle morphism $\rmc_\calH^\flat:\calH\to\calH^\ast\otimes L$ is an isomorphism, with inverse denoted by $\rmc_\calH^\sharp:\calH^\ast\otimes L\to \calH$.
If this is the case, then $\dim M=\text{odd}$, and $\calH$ and $\theta$ are said to be, not just pre-contact, but properly \emph{contact}.

\begin{definition}
	\label{def:contact_structures}
	A \emph{contact manifold} is a manifold equipped with a \emph{contact structure} which is equivalently given by either a contact distribution $\calH$ or a corresponding contact form $\theta$.
\end{definition}

\begin{remark}
	\label{rem:pull-back_VB_valued_forms}
	For future reference, let us recall that a line bundle isomorphism $\Phi:L_1\to L_2$, covering $\varphi:M_1\to M_2$, determines the module isomorphism $\Phi^\ast:\Omega^\bullet(M_2;L_2)\to\Omega^\bullet(M_1;L_1)$, covering the algebra isomorphism $\varphi^\ast:\Omega^\bullet(M_2)\to\Omega^\bullet(M_1)$, given by $(\Phi^\ast\omega_2)_x=(\Phi_x)^{-1}\circ\omega_{2,\varphi(x)}\circ\wedge^kT_x\varphi$, for all $k\geq 0$, $\omega_2\in\Omega^k(M_2;L_2)$, and $x\in M_1$.
	Further, if $\Phi$ is just a \emph{regular line bundle morphism} (cf.~Remark~\ref{rem:regularVBmorphisms}), $\Phi^\ast$ is still well-defined, even though it is not an isomorphism.
\end{remark}

The duality between distributions and forms seen in the description of contact manifolds is reflected in the description of their contactomorphisms.

\begin{definition}
	Let $\calH_i$ be a contact distribution on $M_i$ and $\theta_i\in\Omega^1(M_i;L_i)$ be a corresponding contact form, with $i=1,2$.
	A \emph{contactomorphism} is equivalently given by either:
	\begin{enumerate}[label=(\arabic*)]
		\item
		\label{enuitem:contactomorphism:1}
		a diffeomorphism $\varphi:M_1\to M_2$ such that $\calH_2=\varphi_\ast\calH_1:=(T\varphi)\calH_1$, or
		\item
		\label{enuitem:contactomorphism:2}
		a line bundle isomorphism $\Phi:L_1\to L_2$ over $\varphi:M_1\to M_2$, such that $\theta_1=\Phi^\ast\theta_2:=\Phi^{-1}\circ\theta_2\circ(T\varphi)$.
	\end{enumerate} 
\end{definition}

Let $(M,\calH)$ be a contact manifold and $\theta\in\Omega(M;L)$ be a contact form corresponding to the contact distribution $\calH$.
Denote by $\frakX(M,\calH)\subset\frakX(M)$ the Lie algebra of \emph{infinitesimal contactomorphisms} of $(M,\calH)$, i.e.~those vector fields whose flow consists of local contactomorphisms.
It turns out that $\frakX(M,\calH)$ is formed exactly by the \emph{contact vector fields} of $(M,\calH)$, i.e.~those $X\in\frakX(M)$ such that $[X,\Gamma(\calH)]\subset\Gamma(\calH)$, and so it fits in the short exact sequence of $\bbR$-linear maps
\begin{equation}
\label{eq:contactHamiltonianVFs}
\begin{tikzcd}
0\arrow[r]&\frakX(M,\calH)\arrow[r, "\text{incl}"]&\frakX(M)\arrow[r, "\phi"]&\Gamma(\calH^\ast\otimes L)\arrow[r]&0,
\end{tikzcd}
\end{equation}
where $\frakX(M)\to\Gamma(\calH^\ast\otimes L)$, $X\mapsto\phi_X$, is the 1st-order linear differential operator (DO) defined by $\phi_X(Y)=\theta[X,Y]$, for all $Y\in\Gamma(\calH)$.
Further $\rmc_\calH^\sharp:\calH^\ast\otimes L\to\calH\subset TM$ splits~\eqref{eq:contactHamiltonianVFs}, so that $\frakX(M)=\Gamma(\calH)\oplus\frakX(M,\calH)$ and there is a unique $\bbR$-linear isomorphism 
\begin{equation*}
\Gamma(L)\overset{\sim}{\to}\frakX(M,\calH),\ \lambda\mapsto\calX_\lambda,
\quad\text{such that}\ \theta(\calX_\lambda)=\lambda.
\end{equation*}
%, for all $\lambda\in\Gamma(L)$.
Additionally, $\calX_{(-)}$ is a 1st-order DO from $L$ to $TM$, sending each line bundle section $\lambda$ to its associated \emph{Hamiltonian vector field} $\calX_\lambda$.
Indeed, for all $\lambda\in\Gamma(L)$ and $f\in C^\infty(M)$, one gets that
\begin{equation}
\label{eq:contact_vector_fields:DO}
\calX_{f\lambda}=f\calX_\lambda+\rmc_\calH^\sharp((\rmd f)|_\calH\otimes\lambda).%,\quad\text{for all}\ \lambda\in\Gamma(L)\ \text{and}\ f\in C^\infty(M).
\end{equation}
Hence a contact structure determines the bracket
\begin{equation*}
\{-,-\}:\Gamma(L)\times\Gamma(L)\to\Gamma(L),\ (\lambda,\mu)\mapsto\{\lambda,\mu\}:=\theta(\calX_\lambda,\calX_\mu),
\end{equation*}
which is both a Lie bracket and a 1st-order linear bi-DO from $L$ to $L$, and so it is a Jacobi structure on $L$ (cf.~Definition~\ref{def:Jacobi_structures}).
Actually, $\calH$ and $\theta$ are fully encoded by their \emph{associated Jacobi structure} $\{-,-\}$.

\begin{example}[\textbf{The trivial line bundle case}]
%\paragraph{\textbf{The trivial line bundle case}}
Let $\calH$ be a contact distribution on $M$ and let $\theta$ be a corresponding $L$-valued contact form on $M$.
Assume that $L$ is the trivial line bundle $\bbR_M:=M\times\bbR\to M$.
%or equivalently there is a \textcolor{red}{distinguished line distribution} complementary to $\calH$ on $M$.
In this case $\theta$ becomes a nowhere zero \emph{real-valued} $1$-form on $M$ and $\rmc_\calH$ coincides with $-(\rmd\theta)|_\calH$.
So the maximal non-integrability of $\calH$ reduces to the ordinary condition defining \emph{coorientable contact structures}, i.e.
\begin{equation*}
\theta\wedge(\rmd\theta)^n\ \text{is a volume form on}\ M^{2n+1}.
\end{equation*}
Further, in this special coorientable case, there is a distinguished contact vector field which is nowhere tangent to $\calH$, the so-called \emph{Reeb vector field} $E:=\calX_1$.

Consider, for $i=1,2$, two coorientable contact manifolds $(M_i,\calH_i)$, with $\calH_i=\ker\theta_i$, for some $\theta_i\in\Omega^1(M_i)$.
Then a contactomorphism $(M_1,\calH_1)\to(M_2,\calH_2)$ can also be described as a pair $(\varphi,a)$ formed by a diffeomorphism $\varphi:M_1\to M_2$ and a nowhere zero function $a\in C^\infty(M_1)$, the so-called \emph{conformal factor}, such that $\theta_1=a\varphi^\ast\theta_2$.
So, in this setting, a contactomorphism is also denoted by $(\varphi,a):(M_1,\calH_1)\to(M_2,\calH_2)$ and if, additionally, $a=1$, then it is said to be a \emph{strict contactomorphism}.

\end{example}

\subsection{Jacobi Manifolds}
\label{sec:Jacobi_structures}

\begin{definition}
	\label{def:Jacobi_structures}
	A \emph{Jacobi bundle}~\cite{marle} over a manifold $M$ is a line bundle $L\to M$ equipped with a \emph{Jacobi structure} (or \emph{Jacobi bracket}), i.e.~a Lie bracket $\{-,-\}:\Gamma(L)\times\Gamma(L)\to\Gamma(L)$ which additionally satisfies the following two equivalent conditions
	\begin{enumerate}[label=(\arabic*)]
		\item it is a \emph{1st-order linear bi-differential operator} on $L\to M$,
		\item it is \emph{local}, i.e.~$\operatorname{supp}(\{\lambda,\mu\})\subset\operatorname{supp}(\lambda)\cap\operatorname{supp}(\mu)$, for all $\lambda,\mu\in\Gamma(L)$.
		%\begin{equation*}
		%\operatorname{supp}(\{u,v\})\subset\operatorname{supp}(u)\cap\operatorname{supp}(v).
		%\end{equation*}
	\end{enumerate}
	Then a \emph{Jacobi manifold} $(M,L,\{-,-\})$ is a manifold $M$ with a Jacobi bundle $(L,\{-,-\})$ over it.
\end{definition}

In this paper, for any line bundle $L\to M$, we will use the following identification of the Jacobi structures on $L\to M$ with the Maurer--Cartan elements of the graded Lie algebra $((\calD^\bullet L)[1],\ldsb-,-\rdsb)$ formed by the multi-DOs from $L$ to $L$ w.r.t.~the Schouten--Jacobi bracket (cf.~Appendix~\ref{sec:gauge_algebroid}).

\begin{proposition}[{\cite[Proposition~2.7]{tortorella2017phdthesis}}]
	\label{prop:Jacobi_bi-DOs}
	For any line bundle $L\to M$, the relation $$\{\lambda,\mu\}=\J(j^1\lambda,j^1\mu),\ \text{for all}\ \lambda,\mu\in\Gamma(L),$$
	establishes a 1-1 correspondence between the Jacobi structures $\{-,-\}$ on $L\to M$ and the \emph{Jacobi bi-DOs} $\J$ on $L\to M$, i.e.~those $\J\in\calD^2L$ such that $\ldsb\J,\J\rdsb=0$.
\end{proposition}

\begin{remark}
	\label{rem:regularVBmorphisms}
	For the reader's convenience we recall here the notion of regular vector bundle morphism.
	Specifically, a vector bundle morphism $\Phi:E_1\to E_2$, covering a smooth map $\varphi:M_1\to M_2$, is called \emph{regular} if its restriction on fibers $\Phi_x:E_{1,x}\to E_{2,\varphi(x)}$ is a linear isomorphism for all $x\in M_1$.
	Consequently, $\Phi^\ast:\Gamma(E_2)\to\Gamma(E_1),\lambda\mapsto\Phi^\ast \lambda$, \emph{the pull-back of sections along $\Phi$}, is well-defined by setting \begin{equation*}
	(\Phi^\ast \lambda)_x:=(\Phi_x)^{-1}(\lambda_{\varphi(x)})\in E_{1,x},\ \text{for all}\ x\in M_1\ \text{and}\ \lambda\in\Gamma(E_2).
	\end{equation*}
	Equivalently, the regular vector bundle morphism $\Phi:E_1\to E_2$ can also be seen as a pair formed by a smooth map $\varphi:M_1\to M_2$ and a vector bundle isomorphism $F:\varphi^\ast E_2\to E_1$, covering the identity map $\id_{M_1}:M_1\to M_1$, so that now $\Phi^\ast \lambda=F\circ(\varphi^\ast \lambda)$, for all $\lambda\in\Gamma(E_2)$.
	Further, it is easy to see that vector bundles with regular vector bundle morphisms form a category.
\end{remark}

\begin{definition}
	\label{def:Jacobi_morphism}
	A \emph{Jacobi morphism} $\Phi:(M_1,L_1,\{-,-\}_1)\longrightarrow (M_2,L_2,\{-,-\}_2)$ is a regular line bundle morphism $\Phi:L_1\to L_2$, covering $\varphi:M_1\to M_2$, such that, for all $\lambda,\mu\in\Gamma(L_2)$,
	% the pull-back of sections $\Phi^\ast:\Gamma(L_2)\to\Gamma(L_1)$ is a Lie algebra morphism
	\begin{equation}
	\label{eq:def:Jacobi_morphism}
	\Phi^\ast\{\lambda,\mu\}_2=\{\Phi^\ast\lambda,\Phi^\ast\mu\}_1.%,\ \text{for all}\ u,v\in\Gamma(L_2).
	\end{equation}
\end{definition}

Let $(M,L, \{-,-\})$ be a Jacobi manifold.
Denote by $\J\in\calD^2 L$ the Jacobi bi-DO corresponding to $\{-,-\}$ (cf.~Proposition~\ref{prop:Jacobi_bi-DOs}) and by $\J^\sharp:J^1L\to DL$ the associated vector bundle morphism over $\id_M:M\to M$, with $DL=(J^1L)^*\otimes L\to M$ denoting the gauge algebroid of $L$ (cf.~Appendix~\ref{sec:gauge_algebroid}).
%Then the associated vector bundle morphism $\J^\sharp:J^1L\to DL$, over $\id_M:M\to M$, is defined by
%\begin{equation}
%\langle\beta,\J^\sharp(\alpha)\rangle:=
%\J(\alpha,\beta),\quad \textrm{for all}\ \alpha,\beta\in\Omega^1_L=\Gamma(J^1L).
%\end{equation}
Any section $\lambda\in\Gamma(L)$ determines a \emph{Hamiltonian DO} $\Delta_\lambda:=\J^\sharp(j^1\lambda)$ and a \emph{Hamiltonian vector field} 
\begin{equation}\label{calc1}
\calX_\lambda:=(\sigma\circ\J^\sharp)(j^1\lambda)=\sigma(\Delta_\lambda)\in\frakX(M), 
\end{equation}
with $\sigma:DL\to TM$ denoting the symbol.
They are also equivalently defined by the following identities:
\begin{equation*}
\Delta_\lambda\mu=\{\lambda,\mu\}\quad\text{\and}\quad\{\lambda,f\mu\}=\calX_\lambda(f)\mu+f\{\lambda,\mu\},
\end{equation*}
for all $\lambda,\mu\in\Gamma(L)$ and $f\in C^\infty(M)$.
So one gets the Lie algebra morphisms $\Gamma(L)\to\calD(L),\ \lambda\mapsto\Delta_\lambda$, and $\Gamma(L)\to\frakX(M),\ \lambda\mapsto\calX_\lambda$, which are also 1st-order linear DO from $L$ to $DL$ and $TM$ respectively.

\begin{remark}
	\label{rem:Jacobi_morphism}
	For further reference, notice that, if $\Phi:(M_1,L_1,\{-,-\}_1)\to(M_2,L_2,\{-,-\}_2)$ is a Jacobi morphism covering $\varphi:M_1\to M_2$, then, for any section $\lambda\in\Gamma(L_2)$, the following identities hold
	\begin{equation}
	\label{eq:rem:Jacobi_morphism}
	(D\Phi)\circ\Delta_{\Phi^\ast\lambda}=\Delta_{\lambda}\circ\varphi\quad\text{and}\quad (T\varphi)\circ\calX_{\Phi^\ast\lambda}=\calX_{\lambda}\circ\varphi.
	\end{equation}
	Indeed, the first identity above is just a rephrasing of Equation~\eqref{eq:def:Jacobi_morphism} defining the Jacobi morphisms, and the second identity follows from the first one by taking the symbol componentwise.
\end{remark}

\begin{example}[\bf{The projectivization $\mathbb{P}(\frakg^*)$}~\cite{SaSe}]\label{projective}
 Given a Lie algebra $\frakg$, the projectivization 
$\mathbb{P}(\frakg^*)=(\frakg^*\setminus\{0\})/\mathbb{R}^\times$
 inherits a Jacobi bracket $\{-,-\}_{\bbP(\frakg^\ast)}$ on $\bigO_{\bbP(\frakg^\ast)}(1)$, 
the dual of the tautological line bundle $\tau$ over $\bbP(\frakg^*)$ ($\tau_{[\mu]}=\langle\mu\rangle,\ \mu\in\frakg^*\setminus\{0\}$). 
The bracket comes from the natural identification of $\Gamma(\bigO_{\bbP(\frakg^\ast)}(1))$ with $C^\infty_{\rm hom}(\mathfrak{g}^*\setminus\{0\})$, the space of (degree $1$) homogeneous smooth functions on $\mathfrak{g}^*\setminus\{0\}$, 
and the fact that $C^\infty_{\rm hom}(\mathfrak{g}^*\setminus\{0\})$ is a subalgebra of $C^\infty(\mathfrak{g}^*\setminus\{0\})$ with the linear Poisson bracket $\{-,-\}_+$. 
The identification takes a section $\beta\in\Gamma(\bigO_{\bbP(\frakg^\ast)}(1))$, and identifies it with the homogeneous function $F_\beta$ on $ \mathfrak{g}^*\setminus\{0\}$
defined by 
$F_\beta(\mu):=\langle\beta({[\mu]}),\mu\rangle$.  With this, the projection map 
$\frakg^*\setminus\{0\}\to \mathbb{P}(\frakg^*),\ \mu\mapsto[\mu]$
becomes a Jacobi morphism with bundle component given by
\begin{equation*}\Phi:\mathbb{R}\times (\frakg^*\setminus\{0\})\to \bigO_{\bbP(\frakg^\ast)}(1),\quad
\Phi(t,\mu):\langle\mu\rangle\to\mathbb{R}, \ s\mu\mapsto st.
\end{equation*}
Alternatively, $(\bbP(\frakg^\ast),\bigO_{\bbP(\frakg^\ast)}(1),\{-,-\}_{\bbP(\frakg^\ast)})$ can be seen as the dehomogenization of $(\frakg^\ast,\{-,-\}_+)$ (cf.~Appendix \ref{sec:homogeneous_symplectic/Poisson}).
\end{example}

\subsection*{The trivial line bundle case.}
	A \emph{Jacobi pair}~\cite{lichnerowicz1978jacobi} on a manifold $M$ is a pair $(\Pi,E)$ formed by a bivector field $\Pi\in\frakX^2(M)$ and a vector field $E\in\frakX(M)$ such that
	$\ldsb\Pi,\Pi\rdsb=2E\wedge\Pi\ \text{and}\ \ldsb E,\Pi\rdsb=0$,
	%\begin{equation*}
	%\ldsb\Lambda,\Lambda\rdsb=2E\wedge\Lambda\ \text{and}\ \ldsb E,\Lambda\rdsb=0,
	%\end{equation*}
	where $\ldsb-,-\rdsb$ denotes the Schouten--Nijenhuis bracket on multivector fields on $M$.
%\end{definition}

\begin{proposition}[{\cite[Lemma 3]{kirillov}}]
	\label{prop:Jacobi_pairs}
	For any manifold $M$, the following relation
	\begin{equation*}
	\{f,g\}=\Pi(\rmd f,\rmd g)+i_E(f\rmd g-g\rmd f),\ \text{for all}\ f, g\in C^\infty(M),
	\end{equation*}
	gives a 1-1 correspondence between Jacobi structures $\{-,-\}$ on $\bbR_M\to M$ and Jacobi pairs $(\Pi,E)$ on $M$.
\end{proposition}

Let $\{-,-\}$ be a Jacobi structure on $\bbR_M\to M$ and let $(\Pi,E)$ be the corresponding Jacobi pair on $M$.
Then the Jacobi manifold $(M,\bbR_M,\{-,-\})$ is also denoted by $(M,\Pi,E)$.
The vector field $E$ measures the failure of $\{-,-\}$ to be a Poisson structure or, equivalently, of $\Pi$ to be a Poisson bivector.
Indeed,
\begin{equation*}
\{f,gh\}-\{f,g\}h-g\{f,h\}=ghE(f),\quad\text{for all}\ f,g,h\in C^\infty(M),
\end{equation*}
and so $\{-,-\}$ satisfies the Leibniz rule if and only if $E=0$.
Furthermore, $E=\calX_1$, i.e.~$E$ is the distinguished Hamiltonian vector field corresponding to the distinguished section $1\in C^\infty(M)=\Gamma(\bbR_M)$.
In general, for any $f\in C^\infty(M)$, the associated Hamiltonian vector field $\calX_f$ can be rewritten as
\begin{equation*}
\calX_f=\Pi^\sharp(\rmd f)+f E.
\end{equation*}

Let $(\Pi_i,E_i)$ be a Jacobi pair on $M_i$, with corresponding Jacobi structures $\{-,-\}_i$ on $\bbR_{M_i}\to M_i$, for $i=1,2$.
In view of Remark~\ref{rem:regularVBmorphisms}, a Jacobi morphism $\Phi:(M_1,\bbR_{M_1},\{-,-\}_1)\to(M_2,\bbR_{M_2},\{-,-\}_2)$ reduces to a pair $(\varphi,a)$ formed by a map $\varphi:M_1\to M_2$ and a nowhere zero function $a\in C^\infty(M_1)$, the so-called \emph{conformal factor}, such that, for all $f,g\in C^\infty(M_2)$,
\begin{equation*}
\{a\varphi^\ast f,a\varphi^\ast g\}_1=a\varphi^\ast\{f,g\}_2.
\end{equation*}
So, in this setting, a Jacobi morphism is also denoted by $(\varphi,a):(M_1,\Pi_1,E_1)\to(M_2,\Pi_2,E_2)$ and if, additionally, $a=1$, then it is said to be a \emph{strict Jacobi morphism}.
	
\subsection*{Non-degenerate Jacobi structures.}

As seen in Section~\ref{sec:contact_structures}, a contact structure is fully encoded by its associated Jacobi structure.
In this way, as we are going to recall below following~\cite[Section 3]{vitagliano2018djbundles}, contact structures identify with the so-called non-degenerate Jacobi structures.

\begin{definition}
	\label{def:non-degenerate_Jacobi}
	A Jacobi structure $\{-,-\}$ on $L\to M$ is called \emph{non-degenerate} if the corresponding Jacobi bi-DO $\J\in\calD^2 L$ is non-degenerate, i.e.~the associated vector bundle morphism $\J^\sharp:J^1L\to DL$ is an isomorphism.
\end{definition}

Let $L\to M$ be a line bundle.
The relation $\J^\sharp=\varpi^\sharp$ establishes a 1-1 correspondence between non-degenerate bi-DO $\J\in\calD^2 L=\Gamma(\wedge^2(J^1L)^\ast\otimes L)$ and non-degenerate $L$-valued Atiyah forms $\varpi\in\Omega_L^2=\Gamma(\wedge^2(DL)^\ast\otimes L)$.
Further it turns out that, within this correspondence, the Jacobi identity for $\{-,-\}$ (or equivalently $\ldsb\J,\J\rdsb=0$) holds if and only if $\varpi$ is $d_D$-closed.
This leads to the following.

\begin{proposition}
	\label{prop:non-degenerateJacobi=symplecticAtiyah}
	There is a canonical 1-1 correspondence between non-degenerate Jacobi structures $\J=\{-,-\}$ on $L\to M$ and $L$-valued symplectic Atiyah forms $\varpi$.
\end{proposition}
The relation $\varpi=d_D(\theta\circ\sigma)$, or equivalently $\iota_{\mathbbm{1}}\varpi=\theta\circ\sigma$, establishes a 1-1 correspondence between the $L$-valued pre-contact forms $\theta$ on $M$ and those $d_D$-closed $L$-valued Atiyah $2$-forms $\varpi$ such that $(\iota_{\mathbbm{1}}\varpi)_x\neq0_x\in( D_xL)^\ast\otimes L_x\simeq J^1_xL$, for all $x\in M$.
Further, within this correspondence, $\theta$ is contact, i.e.~$\ker\theta$ is maximally non-integrable, if and only if $\varpi$ is non-degenerate.
This leads to the following.

\begin{proposition}[{\cite[Prop.~3.3]{vitagliano2018djbundles}}]
	\label{prop:contact=symplecticAtiyah}
	There is a canonical 1-1 correspondence between $L$-valued symplectic Atiyah forms $\varpi$ and $L$-valued contact forms $\theta$.
\end{proposition}

From the combination of Propositions~\ref{prop:non-degenerateJacobi=symplecticAtiyah} and~\ref{prop:contact=symplecticAtiyah}, one ends up with the following. 
\begin{proposition}[{\cite[Sect.~4]{kirillov}}]
	\label{prop:non-degenerateJacobi=contact}
	For any line bundle $L\to M$, there is a canonical 1-1 correspondence between non-degenerate Jacobi structures on $L$ and $L$-valued contact forms on $M$.
\end{proposition}

\subsection{Locally Conformal Symplectic Structures}
\label{sec:lcs_structures}

Besides Poisson and contact, another class of Jacobi structures is given by the locally conformal symplectic structures.
Following~\cite[App.~A]{vitagliano2015algebras}, and so slightly generalizing~\cite{vaisman1985}, we adopt a line bundle approach to these structures.

\begin{definition}
	A \emph{locally conformal symplectic} (or \emph{l.c.s.}) structure on a manifold $M$ consists of a line bundle $L\to M$, a representation $\nabla$ of $TM$ on $L$, and a $\rmd_\nabla$-closed non-degenerate $2$-form $\omega\in\Omega^2(M;L)$.
	Then a \emph{l.c.s.~manifold} $(M,L,\nabla,\omega)$ is a manifold $M$ equipped with a l.c.s.~structure $(L,\nabla,\omega)$ over it.
\end{definition}

Clearly a l.c.s.~manifold $M$ is even-dimensional.

In the coorientable case, i.e.~when $L=\bbR_M$, a l.c.s.~structure reduces to a pair $(\eta,\omega)\in\Omega^1(M)\times\Omega^2(M)$ s.~t.~$\eta$ is closed, $\omega$ is non-degenerate, and $\rmd\omega+\omega\wedge\eta=0$.
So, in the coorientable case, one recovers the notion of l.c.s.~structure as  given in~\cite{vaisman1985}.
In particular, a symplectic structure is nothing but a l.c.s.~structure with $L=\bbR_M$ and $\nabla_X=X$.

For any l.c.s.~structure $(L,\nabla,\omega)$ on $M$, there exists an associated Jacobi structure $\{-,-\}$ on $L\to M$ defined in the following way.
First, for any section $\lambda\in\Gamma(L)$, there exists a Hamiltonian vector field 
\begin{equation}\label{calc2}
\calX_\lambda:=\omega^\sharp(\rmd_\nabla\lambda)\in\frakX(M).
\end{equation}
Then the \emph{associated Jacobi structure} $\{-,-\}:\Gamma(L)\times\Gamma(L)\to\Gamma(L)$ is defined by
\begin{equation*}
\{\lambda,\mu\}:=\omega(\calX_\lambda,\calX_\mu)=(\rmd_\nabla\lambda)(\calX_\mu).
\end{equation*}
Such a bracket satisfies the Jacobi identity because $\rmd_\nabla\omega=0$, and it is a bi-DO since $\calX_{(-)}:\Gamma(L)\to\frakX(M)$ is a DO from $L$ to $TM$ with, in particular,
\begin{equation}
\label{eq:lcs_Ham_vector_fields:DO}
\calX_{f\lambda}=f\calX_\lambda+\omega^\sharp(\rmd f\otimes\lambda),
\end{equation}
for all $\lambda\in\Gamma(L)$ and $f\in C^\infty(M)$.
Further, it is easy to see that this Jacobi structure fully encodes the l.c.s.~structure we started with.
For more details, we refer the reader to~\cite[App.~A]{vitagliano2015algebras}.

\section{Contact Dual Pairs}
\label{sec:CDPs}

In this Section we initiate our investigation of contact dual pairs.
More specifically, Definitions~\ref{def:CDP} and~\ref{def:H_CDP} introduce Lie--Weinstein contact dual pairs (simply called contact dual pairs) and Howe contact dual pairs respectively.
The latter are the local version and the global version respectively of dual pairs of Jacobi morphisms.
Further, as the main source of motivating examples for contact dual pairs, Theorem~\ref{theor:Jacobi} shows that any contact groupoid naturally gives rise to a contact dual pair. 

\subsection{Contact Dual Pairs}
\label{sec:LW_CDPs}

Let $(M_i,L_i,\{-,-\}_i)$ be Jacobi manifolds, for $i=1,2$, and let $(M,\calH)$ be a contact manifold.
Let $\theta\in\Omega^1(M;L)$ be a contact form corresponding to $\calH$. We denote by $\rmc_\calH$ the associated curvature form and by $\{-,-\}_\calH$ the corresponding non-degenerate Jacobi structure on $L\to M$.
Consider a pair of Jacobi morphisms
\begin{equation}
\label{eq:CDP}
\begin{tikzcd}
(M_1,L_1,\{-,-\}_1)&(M,\calH)\ar[l, swap, "\Phi_1"]\ar[r, "\Phi_2"]&(M_2,L_2,\{-,-\}_2),
\end{tikzcd}
\end{equation}
with underlying maps $M_1\overset{\varphi_1}{\longleftarrow} M\overset{\varphi_2}{\longrightarrow} M_2$.

\begin{definition}
	\label{def:CDP}
	Diagram~\eqref{eq:CDP} forms a \emph{contact dual pair} if:
	\begin{enumerate}[label=(\arabic*)]
		\item\label{enumitem:def:CDP_1} $\calH$ is transverse to both $\ker T\varphi_1$ and $\ker T\varphi_2$, i.e.
		\begin{equation}
		\label{eq:def:CDP_1}
		\calH+\ker T\varphi_1=TM=\calH+\ker T\varphi_2,
		\end{equation}
		\item\label{enumitem:def:CDP_2} $\Phi_1^\ast\lambda_1$ and $\Phi_2^\ast\lambda_2$ Jacobi commute, for all $\lambda_1\in\Gamma(L_1)$ and $\lambda_2\in\Gamma(L_2)$, i.e.
		\begin{equation}
		\label{eq:def:CDP_2}
		\{\Phi_1^\ast\Gamma(L_1),\Phi_2^\ast\Gamma(L_2)\}_\calH=0,
		\end{equation}
		\item\label{enumitem:def:CDP_3} the $\varphi_1$-vertical part $\calH_1:=\calH\cap\ker T\varphi_1$ and the $\varphi_2$-vertical part $\calH_2:=\calH\cap\ker T\varphi_2$ of $\calH$ are the $\rmc_\calH$-orthogonal complement of each other, i.e.
		\begin{equation}
		\label{eq:def:CDP_3}
		(\calH_1)^{\perp{\rmc_\calH}}=\calH_2.
		\end{equation}
	\end{enumerate}
Further, the contact dual pair~\eqref{eq:CDP} is called \emph{full} if both of the underlying maps $\varphi_1:M\to M_1$ and $\varphi_2:M\to M_2$ are surjective submersions.
\end{definition}

For the first motivating example of this definition, namely contact groupoids, we refer the reader to Section~\ref{sec:contact_groupoids}. 
In the following, with reference to the pair of Jacobi morphisms~\eqref{eq:CDP}, we will denote by $\calP_i\subset\Gamma(L)$ the $C^\infty(M_i)$-module and Lie subalgebra formed by the pull-back sections along $\Phi_i$:
\begin{equation}\label{calpe}
	\calP_i:=\Phi_i^\ast\Gamma(L_i)=\{\Phi_i^\ast\lambda_i:\lambda_i\in\Gamma(L_i)\},\text{ for }i=1,2.
\end{equation}

\begin{proposition}
	\label{prop:LW_CDP:tangent_fibers}
	Let~\eqref{eq:CDP} be a contact dual pair.
	The distribution $\ker T\varphi_1$ (resp.~$\ker T\varphi_2$) is generated by the contact vector fields $\calX_\lambda$, with $\lambda\in\calP_2$ (resp.~$\lambda\in\calP_1$), i.e.
	\begin{align}
	\label{eq:prop:LW_CDP:tangent_fibers:1}
	\ker T\varphi_1 = \operatorname{span} \{\calX_{\Phi_2^\ast\lambda_2}:\lambda_2\in\Gamma(L_2)\} \quad \text{and}\quad
	\ker T\varphi_2 = \operatorname{span} \{\calX_{\Phi_1^\ast\lambda_1}:\lambda_1\in\Gamma(L_1)\}.
	\end{align}
	In particular, the maps $\varphi_1:M\to M_1$ and $\varphi_2:M\to M_2$ have constant rank, with
	\begin{equation}
	\label{eq:prop:LW_CDP:tangent_fibers:2}
	1+\rank\varphi_1+\rank\varphi_2=\dim M.
	\end{equation}
	If additionally the contact dual pair~\eqref{eq:CDP} is full, the following dimensional relation holds
	\begin{equation}
	\label{eq:dimensional_relation}
	1+\dim M_1+\dim M_2=\dim M.
	\end{equation}
\end{proposition}

\begin{proof}
	Fix $x\in M$.
	Using Equation~\eqref{eq:contact_vector_fields:DO} and the identity $\calH_i^{\perp\rmc_\calH}=\rmc_\calH^\sharp(\calH_i^\circ)$, one gets the following
	\begin{equation}
	\label{expres}
	\{\calX_{\lambda,x}:\lambda\in\calP_i\}=\langle\calX_{\Phi_i^\ast\mu_i,x}\rangle\oplus\calH_{i,x}^{\perp\rmc_\calH},
	\end{equation}
	for any choice of local frames $\mu_i$ of $L_i\to M_i$ around $\varphi_i(x)$, with $i=1,2$.
	Condition~\ref{enumitem:def:CDP_2} in Definition~\ref{def:CDP} allows to compute
	\begin{equation*}
	\calX_{\Phi_1^\ast\lambda_1}(\varphi_2^\ast f_2)\cdot\Phi_2^\ast\lambda_2=\{\Phi_1^\ast\lambda_1,\Phi_2^\ast(f_2\lambda_2)\}_\calH-(\varphi_2^\ast f_2)\cdot\{\Phi_1^\ast\lambda_1,\Phi_2^\ast\lambda_2\}_\calH\overset{\ref{enumitem:def:CDP_2}}{=}0,
	\end{equation*}
	for all $\lambda_1\in\Gamma(L_1)$, $\lambda_2\in\Gamma(L_2)$ and $f_2\in C^\infty(M_2)$, and similarly exchanging indices $1$ and $2$.
	So one gets
	\begin{equation}
	\label{eq:CDP_4th_condition}
	\{\calX_{\lambda,x}:\lambda\in\calP_1\}\subset\ker T_x\varphi_2\quad \text{and}\quad \{\calX_{\lambda,x}:\lambda\in\calP_2\}\subset\ker T_x \varphi_1.
	\end{equation}
	The latter inclusions are actually equalities.
	Indeed, using Condition~\ref{enumitem:def:CDP_1} in Definition~\ref{def:CDP}, one gets that
	\begin{equation}
	\label{eq:proof:prop:LW_CDP:tangent_fibers:1stCondition}
	\dim\calH_{i,x}=-1-\rank(\varphi_i)_x+\dim M,
	\end{equation}
	for $i=1,2$, and then, using Condition~\ref{enumitem:def:CDP_3} in Definition~\ref{def:CDP}, one can compute
	\begin{equation*}
	\dim\{\calX_{\lambda,x}:\lambda\in\calP_1\}\overset{\eqref{expres}}{=}1+\dim\calH_{1,x}^{\perp\rmc_\calH}\overset{\ref{enumitem:def:CDP_3}}{=}1+\dim\calH_{2,x}\overset{\eqref{eq:proof:prop:LW_CDP:tangent_fibers:1stCondition}}{=}\dim(\ker T_x\varphi_2),
	\end{equation*}
	and similarly exchanging indices $1$ and $2$.
	Now, the latter and Equation~\eqref{eq:CDP_4th_condition} prove that Equation~\eqref{eq:prop:LW_CDP:tangent_fibers:1} holds.
	Finally from Equation~\eqref{eq:proof:prop:LW_CDP:tangent_fibers:1stCondition} it also follows that 
	\begin{equation*}
	-1+\dim M=\rank\calH\overset{\ref{enumitem:def:CDP_3}}{=}\dim\calH_{1,x}+\dim\calH_{2,x}\overset{\eqref{eq:proof:prop:LW_CDP:tangent_fibers:1stCondition}}{=}-2+2\dim M-\rank(\varphi_1)_x-\rank(\varphi_2)_x.
	\end{equation*}
	This shows that also Equation~\eqref{eq:prop:LW_CDP:tangent_fibers:2} holds, and so concludes the proof.
\end{proof}

For future reference, we single out also the following by-product of the proof of Proposition~\ref{prop:LW_CDP:tangent_fibers}.
\begin{corollary}
	\label{cor:LW_CDP:tangent_fibers}
	In a contact dual pair~\eqref{eq:CDP}, for any $x\in M$, the following identities hold
	\begin{equation}
	\label{eq:cor:LW_CDP:tangent_fibers}
		\ker T_x\varphi_1=\langle\calX_{\Phi_2^\ast\lambda_2,x}\rangle \oplus \calH_{2,x}^{\perp\rmc_\calH} \quad\text{and} \quad \ker T_x\varphi_2=\langle\calX_{\Phi_1^\ast\lambda_1,x}\rangle \oplus \calH_{1,x}^{\perp\rmc_\calH},
	\end{equation}
	where $\lambda_1$ (resp.~$\lambda_2$) is any local frame of $L_1\to M_1$ (resp.~$L_2\to M_2$) around $\varphi_1(x)$ (resp.~$\varphi_2(x)$).
\end{corollary}

\begin{example}[{\bf The trivial line bundle case}]
	\label{ex:LW_CDP:trivial_line_bundle_case}
Let us assume that, in diagram~\eqref{eq:CDP}, all the line bundles $L=TM/\calH$, $L_1$ and $L_2$ are the trivial ones.
So that, in particular, the contact structure on $M$ is coorientable, i.e.~$\calH=\ker\theta$ for some distinguished $\theta\in\Omega^1(M)$.
Further, for $i=1,2$, the Jacobi structure $\{-,-\}_i$ on $L_i$ reduces to a Jacobi pair $(\Pi_i,E_i)$ on $M_i$ and the Jacobi morphism $\Phi_i$ reduces to the pair formed by the smooth map $\varphi_i:M\to M_i$ and the conformal factor $a_i\in C^\infty(M)$.
Therefore diagram~\eqref{eq:CDP} can also be depicted in the following way
\begin{equation}
\label{eq:CDP_trivial_line_bundle}
\begin{tikzcd}
(M_1,\Pi_1,E_1)&(M,\theta)\ar[l, swap, "{(\varphi_1,a_1)}"]\ar[r, "{(\varphi_2,a_2)}"]&(M_2,\Pi_2,E_2).
%(M_1,\Lambda_1,E_1)&(M,\theta)\ar[l, swap, "(\varphi_1,\Phi'_1)"]\ar[r, "(\varphi_2,\Phi'_2)"]&(M_2,\Lambda_2,E_2)
\end{tikzcd}
\end{equation}
Under these assumptions, it is easy to see that, in Definition~\ref{def:CDP}, Condition~\ref{enumitem:def:CDP_2} can be replaced by the following pair of conditions:
\begin{enumerate}[label=(2.\alph*)]
	\item\label{enumitem:LW_CDP_2a}
	the conformal factors $a_1$ and $a_2$ Jacobi commute, i.e.~$\{a_1,a_2\}_\calH=0$,
	\item\label{enumitem:LW_CDP_2b}
	$\calX_{a_1}\in\Gamma(\ker T\varphi_2)$ and $\calX_{a_2}\in\Gamma(\ker T\varphi_1)$.
\end{enumerate}
Diagram~\eqref{eq:CDP_trivial_line_bundle} is called a \emph{strict contact dual pair}
if in addition $a_1=a_2=1$, i.e.~both of the  maps $\varphi_1$ and $\varphi_2$ are strict Jacobi morphisms.
In this case, the Reeb vector field $E=\calX_{a_1}=\calX_{a_2}$ belongs to $\ker T\varphi_1\cap \ker T\varphi_2$, and so  $E_1=E_2=0$ (cf.~Remark~\ref{rem:Jacobi_morphism}), and the Jacobi structures on $M_1$ and $M_2$ are properly Poisson.

Now, if~\eqref{eq:CDP_trivial_line_bundle} is a contact dual pair, the direct sum decompositions~\eqref{eq:cor:LW_CDP:tangent_fibers} hold globally, i.e.
\begin{equation}
\label{eq:LW_CDP:tangent_fibers:trivial_line_bundle}
\ker T\varphi_1=\langle\calX_{a_2}\rangle \oplus \calH_2^{\perp\rmc_\calH} \quad\text{and} \quad \ker T\varphi_2=\langle\calX_{a_1}\rangle \oplus \calH_1^{\perp\rmc_\calH},
\end{equation}
where the RHS are the \emph{pseudo-orthogonal}~\cite{libermann1991legendre} of $\ker T\varphi_2$ and $\ker T\varphi_1$ w.r.t.~$(a_2)^{-1}\theta$ and $(a_1)^{-1}\theta$ respectively.
Consequently, in the Definition~\ref{def:CDP}, Condition~\ref{enumitem:def:CDP_3} can be replaced by~\eqref{eq:LW_CDP:tangent_fibers:trivial_line_bundle}.
\end{example}

	The notion of contact dual pair is \emph{local} in the following sense.
Diagram~\eqref{eq:CDP} is a contact dual pair if and only if, for any $x\in M$ and any open neighbourhood $V$ of $x$ in $M$, there exists an open neighbourhood $U$ of $x$ in $V$ such that the following is a contact dual pair
\begin{equation}
\label{eq:LW_CDP_local}
\begin{tikzcd}
(M_1,L_1,\{-,-\}_1)&(U,\calH|_U)\ar[l, swap, "\Phi_1"]\ar[r, "\Phi_2"]&(M_2,L_2,\{-,-\}_2).
\end{tikzcd}
\end{equation}
However,  there is also the notion of \emph{Howe contact dual pair} which is global.
We notice that Condition~\ref{enumitem:def:CDP_2} in Definition~\ref{def:CDP} can be equivalently rewritten as
	\begin{equation}
	\label{eq:rem:def:LW_CDP:2_bis}
	\calP_1\subset\calP_2^c\quad(\text{or}\quad\calP_2\subset\calP_1^c),
	\end{equation}
	where the superscript ${}^c$ denotes the \emph{centralizer} in $\Gamma(L)$ w.r.t.~the Jacobi bracket $\{-,-\}_\calH$.
Inspired by this, 
as in the symplectic setting, we introduce the following global version of contact dual pair.

	\begin{definition}
	\label{def:H_CDP}
	Diagram~\eqref{eq:CDP} is a \emph{Howe contact dual pair} if the following conditions hold:
	\begin{enumerate}[label=(\arabic*)]
		\item
		\label{enumitem:def:H_CDP:1}
		$\calH$ is transverse to both $\ker T\varphi_1$ and $\ker T\varphi_2$, i.e.~$\calH+\ker T\varphi_1=TM=\calH+\ker T\varphi_2$,
		\item
		\label{enumitem:def:H_CDP:2}
		$\calP_1$ and $\calP_2$ centralize each other w.r.t.~the Jacobi bracket $\{-,-\}_\calH$, i.e.
		\begin{equation}
		\label{eq:enumitem:def:H_CDP:2}
		\calP_1^c=\calP_2\qquad\text{and}\qquad\calP_2^c=\calP_1.
		\end{equation}
	\end{enumerate}
	As above, the Howe contact dual pair~\eqref{eq:CDP} is called \emph{full} if both of the underlying maps $\varphi_1:M\to M_1$ and $\varphi_2:M\to M_2$ are surjective submersions.
\end{definition}

	In Section~\ref{sec:relation_Howe_Weinstein}, we will 
	%give a proper definition of such notion, and additionally we will 
	describe the relationship between these two notions of contact dual pairs:
	the local one in Definition \ref{def:CDP} and the global one in Definition \ref{def:H_CDP}.

%%%%%%%%%%%%%%%%%%%%%%%%%%%%%%%%%%%%%%%%%%%%%%%%

\subsection{Contact Groupoids}
\label{sec:contact_groupoids}

The first main motivating examples for our definition of contact dual pairs are contact groupoids~\cite{dazord,KS}.
Contact groupoids play an analogous role in Jacobi geometry to the one played by symplectic groupoids in Poisson geometry.
They arise as ``desingularizations'' of Jacobi manifolds, in the sense that, up to certain technical conditions, a Jacobi manifold $M$ integrates to a contact groupoid.

Let $\calG\rightrightarrows\calG_0$ be a Lie groupoid with structure maps $s,t:\calG\to\calG_0$, $m:\calG^{(2)}:=\calG{}_s{\times}_t\calG\to\calG$, $u:\calG_0\to\calG$, and $i:\calG\to\calG$.
One says that a distribution $\calH\subset T\calG$ is \emph{multiplicative}, if $\calH$ is a wide Lie subgroupoid of the tangent groupoid $T\calG\rightrightarrows T\calG_0$, or in other words $T\calG_0\subset\calH$, $\calH\cdot\calH\subset\calH$ and $\calH^{-1}\subset\calH$, where the multiplication and the inversion are those of $T\calG\rightrightarrows T\calG_0$.
Correspondingly, given a representation $L_0\to \calG_0$ of $\calG$, a form $\theta\in\Omega^1(\calG;t^*L_0)$ is \emph{multiplicative} if	
\begin{equation}
	\label{eq:mult_contact_form:CrainicSalazar}
	(m^\ast\theta)_{(g,h)}=(\operatorname{pr}_1^\ast\theta)_{(g,h)}+g\cdot(\operatorname{pr}_2^\ast\theta)_{(g,h)},
	\end{equation}
	for all $(g,h)\in\calG^{(2)}$, where $\operatorname{pr}_1,\operatorname{pr}_2:\calG^{(2)}\to\calG$ denote the standard projections.
	
In this setting, the duality between forms and distributions in the description of contact structures viewed in Section~\ref{sec:contact_structures} specializes in the following way.

\begin{lemma}
	[{\cite[Lemmas 3.6 and 3.7]{crainic2015multiplicative}}]
	\label{prop:multiplicative_contact_structures}
	Let $\theta\in\Omega^1(\calG;t^*L_0)$ by a multiplicative form.
	If $\theta$ is regular in the sense of~\cite{crainic2015multiplicative}, i.e.~$\theta_g:T_gM\to L_{0,tg}$ is surjective, for all $g\in\calG$, then the hyperplane distribution $\calH:=\ker\theta$ is multiplicative, and moreover, any multiplicative hyperplane distribution arises in this way.
	In particular, there is a 1-1 correspondence between multiplicative contact distributions and multiplicative contact forms. 
\end{lemma}

As part of this lemma, the fact that $\calH$ is wide implies that $\calH$ is transverse to both the $s$ and the $t$ fibers:
\begin{equation}
\label{eq:CDP_from_contact_grpd:1st_condition}
\calH+\ker Ts=T\calG=\calH+\ker Tt.
\end{equation}
Hence $L:=T\calG/\calH$ is canonically isomorphic to $\ker Ts/\calH^s$ and to $\ker Tt/\calH^t$, where $\calH^s:=\calH\cap \ker Ts$ and $\calH^t:=\calH\cap \ker Tt$.
Setting $L_0:=L|_{\calG_0}$, left and right translations induce regular line bundle morphisms $S:L\to L_0$ and $T:L\to L_0$ respectively, covering $s:\calG\to\calG_0$ and $t:\calG\to\calG_0$ respectively.
This gives rise to the representation of $\calG$ on $L_0$, such that 
$g\cdot (S_g\lambda_g)=T_g\lambda_g$, for all $g\in\calG$ and $\lambda_g\in L_g$, and to the line bundle isomorphism $L\to t^\ast L_0,\ \lambda_g\mapsto(g,T_g\lambda_g)$.
Then the quotient bundle map $\theta:T\calG\to L,\ u_g\mapsto u_g+\calH_g$, can be interpreted as a form $\theta\in\Omega^1(\calG;t^\ast L_0)$ which turns out to be multiplicative (cf.~\cite[Lemma 3.7]{crainic2015multiplicative}). 

The above Lemma~\ref{prop:multiplicative_contact_structures} motivates the next definition.
\begin{definition}
	\label{def:contact_groupoid}
	A \emph{contact groupoid} is a Lie groupoid $\calG\rightrightarrows\calG_0$ equipped with a \emph{multiplicative contact structure} which is equivalently given by either a multiplicative contact distribution $\calH$ or a multiplicative contact form $\theta$ (cf.~Lemma~\ref{prop:multiplicative_contact_structures}).
\end{definition}

\begin{example}\label{projective2}
	Let $\calG\rightrightarrows\calG_0$ be a Lie groupoid, with Lie algebroid $A\Rightarrow M$.
	It is easy to see that the symplectic groupoid of $A^\ast$, i.e.~the cotangent groupoid $T^\ast\calG\rightrightarrows A^\ast$ equipped with the canonical symplectic form $\omega_\calG$, is homogeneous.
	Hence, by the dehomogenization procedure in Appendix~\ref{sec:homogeneous_symplectic/Poisson}, one obtains the contact groupoid $(\bbP(T^\ast\calG),\theta)\rightrightarrows (\bbP(A)^\ast,\bigO_{\bbP(A^\ast)}(1),\{-,-\})$ (cf.~Corollary~\ref{cor:homogeneous_symplectic_groupoids}).
	In particular, when $A=\frakg$ is a Lie algebra and $\calG=G$ is a Lie group with Lie algebra $\frakg$, then the \emph{projectivized cotangent bundle} $\bbP(T^\ast G)\rightrightarrows\bbP(\frakg^\ast)$ is a contact groupoid where the Jacobi structure on $\bbP(\frakg^\ast)$ is that described on Example~\ref{projective}. 
\end{example}

\begin{remark}
	\label{rem:multiplicative_forms}
	Keeping the same notations of Lemma~\ref{prop:multiplicative_contact_structures}, we point out, for later use in Section~\ref{sec:contact_groupoid_action_CDP}, that the regular line bundle morphisms $S,T:L\to L_0$, covering respectively $s,t:\calG\rightrightarrows\calG_0$, can be viewed as the source and the target of a certain trivial-core line bundle groupoid $L\rightrightarrows L_0$ over $\calG\rightrightarrows\calG_0$ (for the definition of trivial-core line bundle groupoid see~\cite[Def.~4.1]{ETV2016} and references therein).
	The latter is determined by the property that the line bundle isomorphism $L\to t^\ast L_0$ induced by $T$ is a groupoid isomorphism from $L\rightrightarrows L_0$ to the action groupoid $t^\ast L_0\rightrightarrows L_0$ associated with the representation of $\calG$ on $L_0$.
	So, understanding the isomorphism $L\simeq_T t^\ast L_0$, the structure maps of $L\rightrightarrows L_0$ are:
	\begin{equation*}
	S(g,v)=g^{-1}\cdot v,\ T(g,v)=v,\ M((g,v),(h,g^{-1}\cdot v))=(gh,v),\ I(g,v)=(g^{-1},g^{-1}\cdot v),\ U(v)=(tg,v),
	\end{equation*}
	for all $g\in\calG$ and $v\in L_{0,tg}$.
	Finally, if we now look at $\theta$ as taking values in $L\simeq_T t^\ast L_0$, then, through a straightforward computation, Equation~\eqref{eq:mult_contact_form:CrainicSalazar} can be equivalently rephrased as follows 
	\begin{equation}
	\label{eq:enumitem:mult_contact_form}
	M^\ast\theta=\operatorname{Pr}_1^\ast\theta+\operatorname{Pr}_2^\ast\theta,
	\end{equation}
	where $\operatorname{Pr}_1,\operatorname{Pr}_2\colon L{}_S{\times}_T L\to L$ are the projections.
	The latter will be useful later in proving Proposition~\ref{prop:Lie_group_actions}.
\end{remark}

\begin{remark}
	As a continuation of Remark~\ref{rem:multiplicative_forms}, notice also that there is a unique full-core line bundle groupoid $L\rightrightarrows\textbf{0}_{\calG_0}$ such that $L\to t^\ast L_0$ is a line bundle groupoid isomorphism from $L\rightrightarrows\textbf{0}_{\calG_0}$ to the line bundle groupoid $t^\ast L_0\rightrightarrows \textbf{0}_{\calG_0}$ whose structure maps are the following
	\begin{equation*}
	S(g,v)=0_{s(g)},\ T(g,v)=0_{t(g)},\ M((g,v),(h,w))=(gh,v+g\cdot w),\ I(g,v)=(g^{-1},-g^{-1}\cdot v),\ U(0_x)=(x,0_x).
	\end{equation*}
	%Again, as it is well-know, any full-core LB groupoid arises in this way.
	Now, as pointed out by Drummond and Egea~\cite{drummond2019multiplicative}, Condition~\eqref{eq:mult_contact_form:CrainicSalazar} (and so also Condition~\eqref{eq:enumitem:mult_contact_form}) is equivalent to the multiplicativity of $\theta$ viewed as a form on $\calG$ with values in the line bundle groupoid $L\rightrightarrows\textbf{0}_{\calG_0}$.
\end{remark}

Contact groupoids naturally give rise to  contact dual pairs as explained in the following Theorem~\ref{theor:Jacobi}.
Let $\calG\rightrightarrows\calG_0$ be a contact groupoid, with multiplicative distribution $\calH$ and corresponding multiplicative contact form $\theta\in\Omega^1(\calG;t^*L_0)$, where $L:=T\calG/\calH$, and $L_0:=L|_{\calG_0}$.
Denote by $\{-,-\}$ the Jacobi structure on $L\to\calG$ associated to the multiplicative contact structure on $\calG$.
We need first to recall the following.

\begin{proposition}[{\cite[Th.~1]{crainic2015jacobi}\cite[Sect.~3.5]{cabrera2018explicit}}]
	\label{prop:Jacobi_structures_from_contact_grpds}
	There exists a unique Jacobi structure $\{-,-\}_0$ on $L_0\to\calG_0$ such that the following two equivalent conditions are satisfied:
\begin{enumerate}[label=(\arabic*)]
	\item $T:(\calG,L,\{-,-\})\to(\calG_0,L_0,\{-,-\}_0)$ is a Jacobi morphism, covering $t:\calG\to\calG_0$,
	\item $S:(\calG,L,\{-,-\})\to(\calG_0,L_0,-\{-,-\}_0)$ is a Jacobi morphism, covering $s:\calG\to\calG_0$. 
\end{enumerate}
\end{proposition}

\begin{theorem}
	\label{theor:Jacobi}
	%Let $(\calG,\calH)\rightrightarrows\calG_0$ be a contact groupoid, with corresponding contact form $\theta\in\Omega^1(\calG;L)$.
	The regular line bundle morphisms $S:L\to L_0$ and $T:L\to L_0$, covering $s:\calG\to\calG_0$ and $t:\calG\to\calG_0$ respectively, form a full contact dual pair:
	\begin{equation}
	\label{eq:theor:Jacobi}
	\begin{tikzcd}
	(\calG_0,L_0,-\{-,-\}_0)&(\calG,L,\{-,-\})\arrow[l, swap, "S"]\arrow[r, "T"]&(\calG_0,L_0,\{-,-\}_0).
	\end{tikzcd}
	\end{equation}
\end{theorem}

\begin{proof}
	In view of Proposition~\ref{prop:Jacobi_structures_from_contact_grpds}, diagram~\eqref{eq:theor:Jacobi} is a well-defined pair of Jacobi morphisms covering surjective submersions.
	Concerning the conditions in Definition~\ref{def:CDP}, it satisfies Condition~\ref{enumitem:def:CDP_1} because of Equation~\eqref{eq:CDP_from_contact_grpd:1st_condition}, and Condition~\ref{enumitem:def:CDP_3} because, as it is well-known (cf.~\cite[Proposition 5.1]{crainic2015jacobi}),
	\begin{equation}
		\label{eq:CDP_from_contact_grpd:3rd_condition}
		\calH^s=(\calH^t)^{\perp\rmc_\calH}.
	\end{equation}
	Furthermore, in a contact groupoid, $\calX_{(-)}:\Gamma(L)\to\mathfrak{X}(\calG,\calH)$ induces the following $\bbR$-linear isomorphisms
\begin{equation*}
	\Gamma(L_0)\to\mathfrak{X}(\calG,\calH)^l,\ \lambda\mapsto\calX_{S^\ast\lambda}\quad\text{and}\quad\Gamma(L_0)\to\mathfrak{X}(\calG,\calH)^r,\ \lambda\mapsto\calX_{T^\ast\lambda}, 
\end{equation*}
where $\mathfrak{X}(\calG,\calH)^l$ (resp.~$\mathfrak{X}(\calG,\calH)^r$) denotes the space of left (resp.~right) invariant contact vector fields of $(\calG,\calH)$
(see~\cite[Corollary 5.2]{crainic2015jacobi} up to the appropriate modifications for the map $S$).
From this and the fact that, in any Lie groupoid, left invariant vector fields commute with right invariant vector fields, we obtain that
\begin{equation*}
\label{eq:CDP_from_contact_grpd:2nd_condition}
\{S^\ast\Gamma(L_0),T^\ast\Gamma(L_0)\}_\calH=0.
\end{equation*}
So, diagram~\eqref{eq:theor:Jacobi} also satisfies Condition~\ref{enumitem:def:CDP_2} in Definition~\ref{def:CDP}, and this completes the proof.
\end{proof}

\begin{example}[\textbf{The trivial line bundle case}]
	%\paragraph{\textbf{The trivial line bundle case}}
	Let $\calH$ be a multiplicative contact structure on a Lie groupoid $\calG\rightrightarrows\calG_0$ as above.
	Further let us assume that the contact structure is coorientable, i.e.~all the involved line bundles are the trivial ones.
	In this setting, the representation of $\calG\to\calG_0$ on $\bbR_{\calG_0}\to\calG_0$ reduces to multiplicative function $f\in C^\infty(\calG)$, where multiplicativity means $m^\ast f=\operatorname{pr}_1^\ast f+\operatorname{pr}_2^\ast f$, and the multiplicative contact form reduces to an ordinary real-valued contact form $\theta\in\Omega^1(\calG)$ such that
	\begin{equation}
	\label{eq:multiplicative_coorientable_contact_form}
	m^\ast\theta=\operatorname{pr}_1^\ast\theta+(\operatorname{pr}_1^\ast e^f)\operatorname{pr}_2^\ast\theta.
	\end{equation}
	Moreover, the regular line bundle morphisms $I:\bbR_\calG\to\bbR_\calG$, $S:\bbR_\calG\to\bbR_{\calG_0}$ and $T:\bbR_\calG\to\bbR_{\calG_0}$ reduce to the pairs $(i,e^f)$, $(s,e^f)$ and $(t,1)$ respectively.
	So, the Jacobi structure $\{-,-\}_0$ on $\bbR_{\calG_0}\to\calG_0$ is characterized by the property that both target $t$ and source $s$ are Jacobi maps, from $(\calG_0,\theta)$ to $(\calG_0,\{-,-\}_0)$, with conformal factors $1$ and $-e^f$ respectively.
\end{example}

\section{Properties of Contact Dual Pairs}
\label{sec:Properties}

This section continues the study of contact dual pairs.
The interpretation of contact forms with values in a line bundle $L\to M$ as $L$-valued symplectic Atiyah forms (cf.~Proposition~\ref{prop:contact=symplecticAtiyah}) leads to a more compact and geometrically insightful description of contact dual pairs.
Indeed, as proven in Proposition~\ref{prop:LW_CDP_Atiyah_forms}, the list of the three different conditions in Definition~\ref{def:CDP} can be summarized by a single orthogonality condition.
This immediately leads to other equivalent descriptions of contact dual pairs.
First, Proposition~\ref{prop:CDP_Dirac-Jacobi} recasts the notion of contact dual pairs in the language of Dirac--Jacobi geometry.
Second, using the ``symplectization/Poissonization trick'' (see Appendix~\ref{sec:homogeneous_symplectic/Poisson}), Proposition~\ref{prop:CDPs=homogeneous_SDPs} establishes a 1-1 correspondence %, up to natural transformations, 
between the contact dual pairs and the homogeneous symplectic dual pairs (in the sense of Definition~\ref{def:homogeneous_SDPs}).
Finally, Proposition~\ref{prop:relation_Howe_Weinstein} investigates the non-trivial relation existing between the local version (Definition~\ref{def:CDP}) and the global version (Definition~\ref{def:H_CDP}) of contact dual pairs.

\subsection{Contact dual pairs and symplectic Atiyah forms}
\label{sec:LW_CDP_Atiyah_forms}

The aim of this section is to obtain a more compact description of contact dual pairs by making use of symplectic Atiyah forms.

Let $(M,\calH)$ be a contact manifold, with $\calH=\ker\theta$ for a contact form $\theta\in\Omega^1(M;L)$.
Denote by $\J=\{-,-\}_\calH$ the corresponding non-degenerate Jacobi structure on $L\to M$ and by $\varpi$ the corresponding $L$-valued symplectic Atiyah form (cf.~Propositions~\ref{prop:contact=symplecticAtiyah} and~\ref{prop:non-degenerateJacobi=contact}).
Further consider a pair of Jacobi morphisms
\begin{equation}
\label{eq:LW_CDP_Atiyah_forms:diagram}
\begin{tikzcd}
(M_1,L_1,\{-,-\}_1)&(M,\calH)\arrow[l, swap, "\Phi_1"]\arrow[r, "\Phi_2"]&(M_2,L_2,\{-,-\}_2),
\end{tikzcd}
\end{equation}
with underlying maps $\!\!\begin{tikzcd}M_1&M\arrow[l, swap, "\varphi_1"]\arrow[r, "\varphi_2"]&M_2\end{tikzcd}\!\!.$

\begin{lemma}
	\label{lem:technical_lemma}
	For all $x\in M$ and $i=1,2$, the following identities hold
	\begin{align}
	\label{eq:LW_CDP_Atiyah_forms:1}
	(\ker D_x\Phi_i)^\circ&=\left\{j^1_x(\Phi_i^\ast\lambda_i):\lambda_i\in\Gamma(L_i)\right\}\subset J^1_xL,\\
	\label{eq:LW_CDP_Atiyah_forms:2}
	(\ker D_x\Phi_i)^{\perp\varpi}&=\left\{(\Delta_{\Phi_i^\ast\lambda_i})_x:\lambda_i\in\Gamma(L_i)\right\}\subset D_xL,
	\end{align}
	where the superscript ${}^\circ$ denotes the annihilator w.r.t.~the  natural duality pairing $\langle-,-\rangle:DL\otimes J^1L\to L$. 
	Moreover, for all $x\in M$ and $i=1,2$, one also gets the following identities:
	\begin{equation}
	\label{eq:LW_CDP_Atiyah_forms:3}
	\sigma^{-1}(\calH_x)=\langle\mathbbm{1}_x\rangle^{\perp\varpi}\quad \text{and}\quad \sigma^{-1}(\ker T_x\varphi_i)=\langle\mathbbm{1}_x\rangle\oplus\ker D_x\Phi_i,
	\end{equation}
	with $\sigma:DL\to TM$ denoting the symbol map.
\end{lemma}

\begin{proof}
	Working with local coordinates on $M_i$ and local frames of $L_i\to M_i$, it is easy to see that Equation~\eqref{eq:LW_CDP_Atiyah_forms:1} holds.
	Then Equation~\eqref{eq:LW_CDP_Atiyah_forms:2} follows immediately from Equation~\eqref{eq:LW_CDP_Atiyah_forms:1} because of the fact that $\varpi^\sharp=\J^\sharp$ (cf.~Proposition~\ref{prop:non-degenerateJacobi=symplecticAtiyah}) and $(\ker D\Phi_i)^{\perp\varpi}=\varpi^\sharp((\ker D\Phi_i)^\circ)$.
	Finally Equation~\eqref{eq:LW_CDP_Atiyah_forms:3} is a straightforward consequence of the identities $\theta\circ\sigma=\iota_{\mathbbm{1}}\varpi$, $\sigma\circ D\Phi_i=T\varphi_i\circ\sigma$ and $(D_x\Phi_i)\mathbbm{1}_x=\mathbbm{1}_{\varphi_i(x)}$.
\end{proof}

\begin{lemma}
	\label{lem:LW_CDP_Atiyah_forms:2}
	Condition~\ref{enumitem:def:CDP_2} in Definition~\ref{def:CDP} holds if and only if $(\ker D\Phi_1)^{\perp\varpi}\subset\ker D\Phi_2$.
\end{lemma}

\begin{proof}
	It is a straightforward consequence of the identities~\eqref{eq:LW_CDP_Atiyah_forms:1} and \eqref{eq:LW_CDP_Atiyah_forms:2}.
\end{proof}

\begin{lemma}
	\label{lem:LW_CDP_Atiyah_forms:1}
	For $i=1,2$, $\calH$ is transverse to $\ker T\varphi_i$ if and only if $\langle\mathbbm{1}\rangle^{\perp\varpi}$ is transverse to $\ker D\Phi_i$.
\end{lemma}

\begin{proof}
	Since $\sigma^{-1}(\calH)=\ker(\iota_{\mathbbm{1}}\varpi)=\langle\mathbbm{1}\rangle^{\perp\varpi}$ and $\sigma^{-1}(\ker T\varphi_i)=\langle\mathbbm{1}\rangle\oplus\ker D\Phi_i$, one can compute:
	\begin{equation*}
	\sigma^{-1}(\calH+\ker T\varphi_i)=\langle\mathbbm{1}\rangle^{\perp\varpi}+\ker D\Phi_i.
	\end{equation*}
	Hence, since $\sigma$ is surjective, one gets that  $\calH+\ker T\varphi_i=TM$ if and only if $\langle\mathbbm{1}\rangle^{\perp\varpi}+\ker D\Phi_i=DL$.
\end{proof}

\begin{proposition}
	\label{prop:LW_CDP_Atiyah_forms}
	Diagram~\eqref{eq:LW_CDP_Atiyah_forms:diagram} is a contact dual pair if and only if  $\ker D\Phi_1$ and $\ker D\Phi_2$ are the orthogonal complement of each other w.r.t.~$\varpi$, i.e.
	\begin{equation}
		\label{eq:prop:LW_CDP_Atiyah_forms}
		(\ker D\Phi_1)^{\perp\varpi}=\ker D\Phi_2.
	\end{equation}
\end{proposition}

\begin{proof}
	Since $\mathbbm{1}\notin\ker D\Phi_i$, for $i=1,2$, we get that, if $(\ker D\Phi_1)^{\perp\varpi}=\ker D\Phi_2$, then $(\ker D\Phi_1)^{\perp\varpi}\subset\ker D\Phi_2$ and $\langle\mathbbm{1}\rangle^{\perp\varpi}$ is transverse to both $\ker D\Phi_1$ and $\ker D\Phi_2$.
	So, in view of Lemmas~\ref{lem:LW_CDP_Atiyah_forms:2} and~\ref{lem:LW_CDP_Atiyah_forms:1}, we can assume that Conditions~\ref{enumitem:def:CDP_1} and~\ref{enumitem:def:CDP_2} in Definition~\ref{def:CDP} hold, and prove that, under this assumption, $\ker D\Phi_1=(\ker D\Phi_2)^{\perp\varpi}$ becomes equivalent to Condition~\ref{enumitem:def:CDP_3} in Definition~\ref{def:CDP}.
	
	Since $\sigma^{-1}\calH=\langle\mathbbm{1}\rangle^{\perp\varpi}$, the symbol map $\sigma:DL\to TM$ induces a vector bundle isomorphism $\langle\mathbbm{1}\rangle^{\perp\varpi}/\langle\mathbbm{1}\rangle\longrightarrow \calH$, covering the identity map $\id_M:M\to M$.
	Further, since $\sigma$ preserves the Lie brackets, and $\varpi:=-d_D(\theta\circ\sigma)$, it is easy to see that this isomorphism identifies the curvature form $\rmc_\calH$ on $\calH$ with the non-degenerate $2$-form induced by $\varpi$ on $\langle\mathbbm{1}\rangle^{\perp\varpi}/\langle\mathbbm{1}\rangle$.
	Hence we get that
	\begin{equation}
	\label{eq:proof:prop:LW_CDP_Atiyah_forms:1}
	\calH_1=(\calH_2)^{\perp\rmc_\calH}\Longleftrightarrow\sigma^{-1}(\calH_1)=(\sigma^{-1}(\calH_2))^{\perp\varpi},
	\end{equation}
	where, let us recall, $\calH_i:=\calH\cap\ker T\varphi_i$.
	For $i=1,2$, using Equation~\eqref{eq:LW_CDP_Atiyah_forms:3}, we can easily compute
	\begin{equation*}
	\sigma^{-1}\calH_i=\sigma^{-1}(\calH)\cap \sigma^{-1}(\ker T\varphi_i)\overset{\eqref{eq:LW_CDP_Atiyah_forms:3}}{=}\langle\mathbbm{1}\rangle^{\perp\varpi}\cap(\langle\mathbbm{1}\rangle\oplus\ker D\Phi_i)=\langle\mathbbm{1}\rangle\oplus(\langle\mathbbm{1}\rangle^{\perp\varpi}\cap\ker D\Phi_i).
	\end{equation*}
	Additionally, using Lemma~\ref{lem:LW_CDP_Atiyah_forms:1} and Condition~\ref{enumitem:def:CDP_1} in Definition~\ref{def:CDP}, we can also compute:
	\begin{equation*}
	(\sigma^{-1}\calH_i)^{\perp\varpi}=\langle\mathbbm{1}\rangle^{\perp\varpi}\cap(\langle\mathbbm{1}\rangle+(\ker D\Phi_i)^{\perp\varpi})=\langle\mathbbm{1}\rangle\oplus(\langle\mathbbm{1}\rangle^{\perp\varpi}\cap(\ker D\Phi_i)^{\perp\varpi}).
	\end{equation*}
	Hence we get the following chain of equivalences
	\begin{equation}
	\label{eq:proof:prop:LW_CDP_Atiyah_forms:2}
	\sigma^{-1}\calH_1=(\sigma^{-1}\calH_2)^{\perp\varpi}\Longleftrightarrow \langle\mathbbm{1}\rangle^{\perp\varpi}\cap\ker D\Phi_1=\langle\mathbbm{1}\rangle^{\perp\varpi}\cap(\ker D\Phi_2)^{\perp\varpi}\Longleftrightarrow\ker D\Phi_1=(\ker D\Phi_2)^{\perp\varpi}
	\end{equation}
	where, in the last step, we have used Lemma~\ref{lem:LW_CDP_Atiyah_forms:2}, Condition~\ref{enumitem:def:CDP_2} in Definition~\ref{def:CDP}, and again the fact that $\mathbbm{1}\notin\ker D\Phi_i$.
	Finally, combining together Equations~\eqref{eq:proof:prop:LW_CDP_Atiyah_forms:1} and~\eqref{eq:proof:prop:LW_CDP_Atiyah_forms:2}, we conclude that $\calH_1=(\calH_2)^{\perp\rmc_\calH}$ if and only if $\ker D\Phi_1=(\ker D\Phi_2)^{\perp\varpi}$.
\end{proof}

As a consequence of Proposition~\ref{prop:LW_CDP_Atiyah_forms} and Equation~\eqref{eq:LW_CDP_Atiyah_forms:2} in Lemma~\ref{lem:technical_lemma}, one obtains the following.
\begin{corollary}
	\label{cor:LW_CDP:fiber_tangent_derivations}
	Let diagram~\eqref{eq:CDP} be a contact dual pair.
	The subbundle $\ker D\Phi_1$ (resp.~$\ker D\Phi_2$) is generated by the Hamiltonian DOs $\Delta_\lambda$, with $\lambda\in\calP_2$ (resp.~$\lambda\in\calP_1$), i.e.
	\begin{align}
	\label{eq:cor:LW_CDP:fiber_tangent_derivations}
	\ker D\Phi_1 = \operatorname{span} \{\Delta_{\Phi_2^\ast\lambda_2}:\lambda_2\in\Gamma(L_2)\} \quad \text{and}\quad
	\ker D\Phi_2 = \operatorname{span} \{\Delta_{\Phi_1^\ast\lambda_1}:\lambda_1\in\Gamma(L_1)\}.
	\end{align}
\end{corollary}

The compact characterization of contact dual pairs provided by Proposition~\ref{prop:LW_CDP_Atiyah_forms} admits the following, similarly compact, re-interpretation in terms of Dirac--Jacobi geometry which will play a central role in the proof of Theorem~\ref{theor:transverse_structure}.
We suggest the reader to give a look at Appendix~\ref{sec:Dirac-Jacobi} (and references therein) for definitions and properties of the operations on Dirac--Jacobi structures (like gauge transformations, push-forward and pull-back) which appear in both the statement and the proof of the following proposition.
\begin{proposition}
	\label{prop:CDP_Dirac-Jacobi}
	The following pair of Jacobi morphisms, defined on the same contact manifold,
	\begin{equation*}
		\begin{tikzcd}
			(M_1,L_1,\J_1=\{-,-\}_1)&(M,L,\theta)\ar[l, swap, "\Phi_1"]\ar[r, "\Phi_2"]&(M_2,L_2,\J_2=\{-,-\}_2),
		\end{tikzcd}
	\end{equation*}
	forms a contact dual pair if and only if
	\begin{equation}
		\label{eq:prop:CDP_Dirac-Jacobi}
		\Phi_1^!\operatorname{Gr}\J_1=\calR_\varpi\Phi_2^!\operatorname{Gr}(-\J_2),
	\end{equation}
	where $\varpi\in\Omega^2_L$ denotes the symplectic Atiyah form corresponding to $\theta$ (cf.~Proposition~\ref{prop:contact=symplecticAtiyah}).
\end{proposition}

\begin{proof}
	Let us start showing that Equation~\eqref{eq:prop:LW_CDP_Atiyah_forms} implies Equation~\eqref{eq:prop:CDP_Dirac-Jacobi}.
	First of all, through a straightforward computation, it is easy to see that, for $i=1,2$,
	\begin{equation}
		\label{eq:proof:prop:CDP_Dirac-Jacobi:2a}
		{\Phi_i}^!\Phi_{i!}\operatorname{Gr}\varpi=\ker D\Phi_i+\operatorname{Gr}\varpi\cap(\ker D\Phi_i)^\perp=\ker D\Phi_i+\calR_\varpi((\ker D\Phi_i)^{\perp\varpi}),
	\end{equation}
	where the superscripts ${}^{\perp\varpi}$ (resp.~${}^\perp$) denotes the orthogonal complement in the gauge algebroid $DL$ (resp.~the omni-Lie algebroid $\bbD L$) w.r.t.~the symplectic Atiyah form $\varpi$ (resp.~the standard bilinear form $\ldab-,-\rdab$).
	Further, since $\Phi_i$ is a Jacobi morphism, one gets that, for $i=1,2$,
	\begin{equation}
		\label{eq:proof:prop:CDP_Dirac-Jacobi:2b}
		{\Phi_i}^!\operatorname{Gr}\J_i={\Phi_i}^!\Phi_{i!}\operatorname{Gr}\varpi.
	\end{equation}
	Now, combining together Equations~\eqref{eq:prop:LW_CDP_Atiyah_forms}, \eqref{eq:proof:prop:CDP_Dirac-Jacobi:2a} and~\eqref{eq:proof:prop:CDP_Dirac-Jacobi:2b}, one immediately obtains Equation~\eqref{eq:prop:CDP_Dirac-Jacobi}.
	
	Let us continue showing that Equation~\eqref{eq:prop:CDP_Dirac-Jacobi} implies Equation~\eqref{eq:prop:LW_CDP_Atiyah_forms}.
	Since $\Phi_i$ is a Jacobi morphism, it is easy to see that, in particular,
	\begin{gather*}
		\varpi^\sharp\operatorname{pr}_J{\Phi_1}^!\operatorname{Gr}\J_1=(\ker D\Phi_1)^{\perp\varpi},\\
		\varpi^\sharp\operatorname{pr}_J\calR_\varpi{\Phi_2}^!\operatorname{Gr}(-\J_2)=\ker D\Phi_2.
	\end{gather*}
	From the latter, and Equation~\eqref{eq:prop:CDP_Dirac-Jacobi}, easily follows Equation~\eqref{eq:prop:LW_CDP_Atiyah_forms}.
\end{proof}

%\begin{remark}
%	Notice that, another proof of Equation~\eqref{eq:prop:LW_CDP:tangent_fibers:1} in Proposition~\ref{prop:LW_CDP:tangent_fibers} can be obtained now as a consequence of Lemma~\ref{lem:technical_lemma} and Proposition~\ref{prop:LW_CDP_Atiyah_forms}.
%\end{remark}

\subsection{Contact dual pairs and homogeneous symplectic dual pairs}
\label{sec:Homogeneization_CDP}

This section aims at identifying %, up to natural transformations, 
the contact dual pairs with the so-called homogeneous symplectic dual pairs.
In doing so we will freely use the ``symplectization/Poissonization trick'' as summarized in Appendix~\ref{sec:homogeneous_symplectic/Poisson}.

\begin{definition}
	\label{def:homogeneous_SDPs}
	A \emph{homogeneous symplectic dual pair} is a symplectic dual pair (cf.~Section~\ref{sec:intro})
	\begin{equation*}
	\begin{tikzcd}
	(P_1,\{-,-\}_1)&(P,\omega)\arrow[l, swap, "\varphi_1"]\arrow[r, "\varphi_2"]&(P_2,\{-,-\}_2)
	\end{tikzcd}
	\end{equation*}
	such that the manifolds $P$, $P_1$ and $P_2$ are equipped with principal $\bbR^\times$-bundle structures, and both the symplectic and Poisson structures, as well as the Poisson maps are homogeneous (cf.~Definition~\ref{def:homogeneous_symplectic/Poisson}).
\end{definition}

The following result extends, from Lie groupoids to dual pairs, the 1-1 correspondence existing between contact groupoids and homogeneous symplectic groupoids (cf.~Corollary~\ref{cor:homogeneous_symplectic_groupoids}).

\begin{proposition}
	\label{prop:CDPs=homogeneous_SDPs}
	The homogenization functor (and the dehomogenization functor, in the opposite direction) establishes a 1-1 correspondence between 1) contact dual pairs and 2) homogeneous symplectic dual pairs.
%	\begin{enumerate}[label=(\arabic*)]
%	\item contact dual pairs, and
%	\item homogeneous symplectic dual pairs.
%	\end{enumerate}
\end{proposition}

\begin{proof}
	In view of Proposition~\ref{prop:equivalence_categories:2} %, up to natural transformations, 
	the homogenization functor (and the dehomogenization functor, in the opposite direction) establishes a 1-1 correspondence between:
	\begin{itemize}
		\item pairs of Jacobi morphisms defined on the same contact manifold
			\begin{equation}
			\begin{tikzcd}
			(M_1,L_1,\{-,-\}_1)&(M,\calH)\arrow[l, swap, "\Phi_1"]\arrow[r, "\Phi_2"]&(M_2,L_2,\{-,-\}_2),
			\end{tikzcd}
			\end{equation}
			where $\calH=\ker\theta$, for $\theta\in\Omega^1(M;L)$, with corresponding symplectic Atiyah form $\varpi\in\Omega_L^2$,
		\item pairs of homogeneous Poisson maps defined on the same homogeneous symplectic manifold
			\begin{equation}
			\begin{tikzcd}
			(\widetilde{L}_1,\{-,-\}_{\widetilde{L}_1})&(\widetilde{L},\widetilde{\varpi})\arrow[l, swap, "\widetilde{\Phi}_1"]\arrow[r, "\widetilde{\Phi}_2"]&(\widetilde{L}_2,\{-,-\}_{\widetilde{L}_2}).
			\end{tikzcd}
			\end{equation}
	\end{itemize}
	Now, in view of Proposition~\ref{prop:LW_CDP_Atiyah_forms}, it only remains to check that, in this setting, the following conditions are equivalent:
	\begin{itemize}
		\item $\ker D\Phi_1$ and $\ker D\Phi_2$ are the orthogonal complement of each other w.r.t.~$\varpi$,
		\item $\ker T\widetilde{\Phi_1}$ and $\ker T\widetilde{\Phi_2}$ are the orthogonal complement of each other w.r.t.~$\widetilde{\varpi}$.
	\end{itemize}
	Denote, for $i=1,2$, by $\pi:\widetilde{L}\to M$ and $\pi_i:\widetilde{L}_i\to M_i$ the bundle projections, and by $\varphi_i:M\to M_i$ the underlying maps for $\Phi_i:L\to L_i$.
	Fix arbitrarily $x\in M$ and $\nu\in\widetilde{L}_x$.
	By the very definition of homogenization functor and symplectization (cf.~Section~\ref{sec:homogeneous_symplectic/Poisson}), the $\bbR$-linear isomorphism $\widehat{\pi}_\nu:T_\nu\widetilde{L}\to D_xL$ transforms the symplectic form $\widetilde{\varpi}_\nu$ on $T_\nu\widetilde{L}$ into the symplectic form $\langle\nu,\varpi_x\rangle$ on $D_xL$ (cf.~Equations~\eqref{eq:homogeneization:1} and~\eqref{eq:homogeneization:2}).
	Further it fits in the following commutative diagram, for $i=1,2$:
	\begin{equation*}
	\begin{tikzcd}
	T_\nu\widetilde{L}\arrow[d, hook, two heads, swap, "\widehat{\pi}_\nu"]\arrow[r, "T_\nu\widetilde\Phi_i"]&T_{\widetilde{\Phi}_i(\nu)}\widetilde{L}_i\arrow[d, hook, two heads, "\widehat{\pi}_{i,\widetilde{\Phi}_i(\nu)}"]\\
	D_xL\arrow[r, swap, "D_x\Phi_i"]&D_{\varphi_i(x)}L_i
	\end{tikzcd}
	\end{equation*}
	Consequently, the symplectic isomorphism $\widehat{\pi}_{\nu}:(T_\nu\widetilde{L},\widetilde{\varpi}_\nu)\to ( D_xL,\langle\nu,\varpi_x\rangle )$ transforms $\ker T_\nu\widetilde{\Phi}_i$ and  $(\ker T_\nu\widetilde{\Phi}_i)^{\perp\widetilde{\varpi}}$ into $\ker D_x\Phi_i$ and  $(\ker D_x\Phi_i)^{\perp\varpi}$ respectively.
	This completes the proof.
\end{proof}

\subsection{The relation between local and global contact duality}
\label{sec:relation_Howe_Weinstein}

The relation existing between the two notions of contact dual pairs (the local one in Definition~\ref{def:CDP} and the global one in Definition~\ref{def:H_CDP}) is described in the next Proposition inspired by the analogous result for symplectic dual pairs (cf.~\cite[Theorem~4]{MOR}). 
\begin{proposition}
	\label{prop:relation_Howe_Weinstein}
	Assume, with reference to diagram~\eqref{eq:CDP}, that the maps $M_1\overset{\varphi_1}{\longleftarrow} M\overset{\varphi_2}{\longrightarrow} M_2$ are surjective submersions.
	Then:
	\begin{enumerate}[label=(\arabic*)]
		\item\label{item:Weinstein->Howe}
		if diagram~\eqref{eq:CDP} is a contact dual pair, with underlying maps having connected fibers, then it is also a Howe contact dual pair,
		\item\label{item:Howe->Weinstein}
		if diagram~\eqref{eq:CDP} is a Howe contact dual pair, with $1+\dim M_1+\dim M_2=\dim M$, then it is also a contact dual pair.
	\end{enumerate}
\end{proposition}

\begin{proof}
	(1)
	In view of Equation~\eqref{eq:rem:def:LW_CDP:2_bis}, and for symmetry reasons, it is enough to prove one of the two inclusions $\calP_1^c\subset\calP_2$ and $\calP_2^c\subset\calP_1$.
	We will focus on proving the first inclusion.
	So fix $\lambda\in\Gamma(L)$ and assume that $\lambda$ is in the centralizer of $\calP_1:=\Phi_1^\ast\Gamma(L_1)$.
	Actually, this assumption can be equivalently rephrased as
	\begin{equation*}
	\Delta_\lambda:=\J^\sharp(j^1\lambda)\in\Gamma(\ker D\Phi_1).
	\end{equation*}
	Further, by means of $\ker D\Phi_1=(\ker D\Phi_2)^{\perp\varpi}$ (cf.~Proposition~\ref{prop:LW_CDP_Atiyah_forms}) and $(\ker D\Phi_2)^{\perp\varpi}=\J^\sharp((\ker D\Phi_2)^\circ)$, the latter becomes
	\begin{equation*}
	j^1\lambda\in\Gamma(\left(\ker D\Phi_2\right)^\circ).
	\end{equation*}
	Hence, because of Equation~\eqref{eq:cor:LW_CDP:fiber_tangent_derivations}, the connectedness hypothesis implies that the section $\lambda$ is constant on the fibers of $\varphi_2:M\to M_2$, so that $\lambda=\Phi_2^\ast\lambda_2$, for some (unique) $\lambda_2\in\Gamma(L_2)$. 
	
	(2)
	If diagram~\eqref{eq:CDP} is a Howe contact dual pair, then, in particular, Condition~\ref{enumitem:def:CDP_2} in Definition~\ref{def:CDP} holds, namely $\{\Phi_1^\ast\lambda_1,\Phi_2^\ast\lambda_2\}_\calH=0$, for all $\lambda_1\in\Gamma(L_1)$ and $\lambda_2\in\Gamma(L_2)$.
	In view of Lemma~\ref{lem:LW_CDP_Atiyah_forms:2}, the latter is completely equivalent to the following condition
	\begin{equation*}
	\left(\ker D\Phi_1\right)^{\perp\varpi}\subset\ker D\Phi_2.
	\end{equation*}
	Now the hypothesis, and the identity $\rank(\ker D\Phi_i)=\rank(\ker T\varphi_i)$, for $i=1,2$, allow to compute
	\begin{equation*}
	\rank((\ker D\Phi_1)^{\perp\varpi})=1+\dim M-\rank(\ker T\varphi_1)=1+\dim M_1=\dim M-\dim M_2=\rank(\ker D\Phi_2),
	\end{equation*}
	and so $(\ker D\Phi_1)^{\perp\varpi}=\ker D\Phi_2$ as we needed to prove, in view of Proposition~\ref{prop:LW_CDP_Atiyah_forms}.
\end{proof}

\begin{remark}
	\label{rem:Howe->Weinstein}
	For future reference it will be helpful to notice that the proof of Proposition~\ref{prop:relation_Howe_Weinstein}~\ref{item:Howe->Weinstein} does not use the full Condition~\ref{enumitem:def:H_CDP:2} from Definition~\ref{def:H_CDP} but only one of the two identities $\calP_1=\calP_2^c$ and $\calP_2=\calP_1^c$.
	Consequently, if $\calH$ is transverse to both $\ker T\varphi_1$ and $\ker T\varphi_2$, the underlying maps $M_1\overset{\varphi_1}{\longleftarrow}M\overset{\varphi_2}{\longrightarrow}M_2$ are surjective submersions with connected fibers, and $1+\dim M_1+\dim M_2=\dim M$, then the following conditions are equivalent:
	\begin{enumerate}[label=(\arabic*)]
		\item\label{rem:Howe->Weinstein:1}
		$\calP_1$ is the centralizer of $\calP_2$ w.r.t.~$\{-,-\}_\calH$, i.e.~$\calP_1=\calP_2^c$,
		\item\label{rem:Howe->Weinstein:2}
		$\calP_2$ is the centralizer of $\calP_1$ w.r.t.~$\{-,-\}_\calH$, i.e.~$\calP_2=\calP_1^c$.
	\end{enumerate}
	Hence, in this case, diagram~\eqref{eq:CDP} is a Howe contact dual pair if and only if $\calH$ is transverse to both $\ker T\varphi_1$ and $\ker T\varphi_2$, and one of the above Conditions~\ref{rem:Howe->Weinstein:1} and~\ref{rem:Howe->Weinstein:2} holds.
\end{remark}

\section{Characteristic Leaf Correspondence}
\label{sec:CharacteristicLeafCorrespondence}

This section discusses one of the main results of this paper, namely the Characteristic Leaf Correspondence Theorem which parallels the analogous result obtained by Weinstein~\cite[Section~8]{We83} for symplectic dual pairs.
It actually consists of three parts.
First, Theorem~\ref{theor:LeafCorrespondenceI} shows that, given a full contact dual pair, if the underlying maps have connected fibers, then there exists a 1-1 correspondence between the characteristic leaves of the two legs of the diagram.
Second, Theorem~\ref{theor:LeafCorrespondenceII} proves that the corresponding characteristic leaves are either both contact or both l.c.s., and exhibits the relation existing between their inherited (transitive Jacobi) structures.
Third, using the Dirac--Jacobi geometric description of contact dual pairs (cf.~Proposition~\ref{prop:CDP_Dirac-Jacobi}), Theorem~\ref{theor:transverse_structure} shows that the transverse structures to corresponding characteristic leaves are anti-isomorphic.
In order to obtain these results, we also provide a very brief review of the local structure of Jacobi manifolds.

\subsection{Characteristic Foliation}
\label{sec:CharacteristicFoliation}

Let $\{-,-\}$ be a Jacobi structure on a line bundle $L\to M$, with corresponding Jacobi bi-DO $\J\in\calD^2L$.
The \emph{characteristic distribution} of the Jacobi manifold $(M,L,\{-,-\})$ is the smooth singular distribution $\calC\subset TM$ generated by the Hamiltonian vector fields, i.e.
\begin{equation*}
\calC:=\text{span}\{\calX_\lambda:\lambda\in\Gamma(L)\}=(\sigma\circ\J^\sharp)(J^1L).
\end{equation*}
In particular, if $\calC=TM$, the Jacobi structure is called \emph{transitive}.
By Equations~\eqref{eq:contact_vector_fields:DO} and \eqref{eq:lcs_Ham_vector_fields:DO} contact and l.c.s.~structures are transitive Jacobi structures (see also the following Propositions~\ref{prop:odd-dim_transitive_Jacobi} and~\ref{prop:even-dim_transitive_Jacobi}).

\begin{proposition}[\textbf{Characteristic Foliation Theorem}~\cite{kirillov}]
	\label{prop:CharacteristicFoliation}
	For any Jacobi manifold $(M,L,\{-,-\})$, the characteristic distribution $\calC$ is integrable à la Stefan--Sussmann.
	For any integral leaf $\calS$ of $\calC$, also called \emph{characteristic leaf} of $(M,L,\{-,-\})$, there exists a unique transitive Jacobi structure on the restricted line bundle $L|_\calS\to\calS$ such that the inclusion $L|_\calS\to L$ is a Jacobi morphism.
\end{proposition}

The dimensional relation $\text{rank}(DL)=1+\dim M$ and the skew-symmetry of $\J\in\calD^2L=\Gamma(\wedge^2(J^1L)^\ast\otimes L)$ imply that: the Jacobi bi-DO $\J$ is non-degenerate if and only if the Jacobi structure $\{-,-\}$ is transitive and $M$ is odd-dimensional.
This, together with Proposition~\ref{prop:non-degenerateJacobi=contact}, leads to the following. 

\begin{proposition}
	\label{prop:odd-dim_transitive_Jacobi}
	There is a canonical 1-1 correspondence between transitive Jacobi structures on $L\to M$, with $\dim M=\text{odd}$, and $L$-valued contact forms on $M$.
\end{proposition}

Let $L\to M$ be a line bundle, with $\dim M=\text{even}$.
As recalled in Section~\ref{sec:lcs_structures}, a l.c.s.~structure $(\nabla,\omega)$ on $L\to M$ is fully encoded by the associated Jacobi structure $\{-,-\}$.
Moreover, the latter is always transitive.
Indeed, for all $\lambda\in\Gamma(L)$, the two associated Hamiltonian vector fields, the one w.r.t.~$(\nabla,\omega)$ and the other one w.r.t.~$\{-,-\}$, coincide.
So, in view of Equation~\eqref{eq:lcs_Ham_vector_fields:DO}, the Hamiltonian vector fields generate $TM$.
Conversely, for any transitive Jacobi structure $\{-,-\}$ on $L\to M$, there exist a $TM$-connection $\nabla$ on $L\to M$ and a non-degenerate $\omega\in\Omega^2(M;L)$ uniquely determined by
\begin{equation}
\label{eq:even-dim_transitive_Jacobi}
\omega(\calX_\lambda,\calX_\mu)=\nabla_{\calX_\mu}\lambda=\{\lambda,\mu\},
\end{equation}
for all $\lambda,\mu\in\Gamma(L)$.
The Jacobi identity for $\{-,-\}$ implies that $(\nabla,\omega)$ is actually a l.c.s.~structure on $L\to M$, i.e.$\nabla$ is flat and $\rmd_\nabla\omega=0$, and clearly its associated Jacobi structure is $\{-,-\}$.
This leads to the following.

\begin{proposition}
	\label{prop:even-dim_transitive_Jacobi}
	There is a canonical 1-1 correspondence between transitive Jacobi structures on $L\to M$, with $\dim M=\text{even}$, and l.c.s.~structures on $L\to M$.
\end{proposition}

\subsection{Characteristic Leaf Correspondence}
\label{sec:CharacteristicLeafCorrespondence:part1}

Let us begin by stating two of the three main theorems of this Section~\ref{sec:CharacteristicLeafCorrespondence}, and explain some examples where these results can be applied.
We leave the proofs of these two theorems for the end of the current Section~\ref{sec:CharacteristicLeafCorrespondence:part1}. 

\begin{theorem}[\textbf{Characteristic Leaf Correspondence I}]
	\label{theor:LeafCorrespondenceI}
	In a full contact dual pair
	\begin{equation}
	\label{eq:theor:LeafCorrespondenceI:CDP}
	\begin{tikzcd}
	(M_1,L_1,\{-,-\}_1)&(M,\calH)\arrow[l, swap, "\Phi_1"]\arrow[r, "\Phi_2"]&(M_2,L_2,\{-,-\}_2),
	\end{tikzcd}
	\end{equation}
	with underlying maps 
	$\!\!\begin{tikzcd}
	M_1&M\arrow[l, swap, "\varphi_1"]\arrow[r, "\varphi_2"]&M_2,\end{tikzcd}\!\!\!$ the singular distribution $\calD:=\ker T\varphi_1+\ker T\varphi_2$ on $M$ is smooth and integrable à la Stefan--Sussmann.
	Further, if $\varphi_1:M\to M_1$ and $\varphi_2:M\to M_2$ have connected fibers, then the relations $\varphi_1^{-1}(\calS_1)=\calS=\varphi_2^{-1}(\calS_2)$ establish 1-1 correspondences between
	\begin{itemize}
		\item the integral leaves $\calS$ of $\calD$, and
		\item the characteristic leaves $\calS_i$ of $M_i$, with $i=1,2$.
	\end{itemize}
	In particular, if $\calS_1$ and $\calS_2$ correspond to each other as above, i.e.~$\calS_2=\varphi_2(\varphi_1^{-1}(\calS_1))$, then they have the same codimension, and either they are both even-dimensional or they are both odd-dimensional.
\end{theorem}

Let $\calS_1$ be a characteristic leaf of $M_1$ and let $\calS_2$ be the corresponding characteristic leaf of $M_2$, as in Theorem~\ref{theor:LeafCorrespondenceI}.
Set $\calS:=\varphi_1^{-1}(\calS_1)=\varphi_2^{-1}(\calS_2)$, $\ell:=L|_\calS\to\calS$ and $\ell_i:=L_i|_{\calS_i}\to\calS_i$, for $i=1,2$.
Then the restriction $\Phi_i|_\calS:\ell\to\ell_i$ is a regular line bundle morphism, covering the surjective submersion $\varphi_i|_\calS:\calS\to\calS_i$, with $i=1,2$.

\begin{theorem}[\textbf{Characteristic Leaf Correspondence II}]
	\label{theor:LeafCorrespondenceII}
	Let $\calS_1$ and $\calS_2$ correspond to each other as in Theorem~\ref{theor:LeafCorrespondenceI}, with $\calS=\varphi_1^{-1}(\calS_1)=\varphi_2^{-1}(\calS_2)$.
	Denote by $I_S:\ell\to L$ the line bundle embedding covering the inclusion $\iota_\calS:\calS\to M$.
	Then there are only the following two possible cases.
	\begin{enumerate}[label=(\arabic*)]
		\item $\calS_i$ is odd-dimensional with inherited $\ell_i$-valued contact form $\theta_i$, for $i=1,2$.
		Then the following relation holds:
		\begin{equation}
		\label{eq:relation_contact}
		I_\calS^\ast\theta=\Phi_1|_\calS^\ast\theta_1+\Phi_2|_\calS^\ast\theta_2\in\Omega^1(\calS;\ell).
		\end{equation}
		\item $\calS_i$ is even-dimensional with inherited l.c.s.~structure $(\ell_i,\nabla^i,\omega_i)$, for $i=1,2$.
		Then there exists a unique representation $\nabla$ of $T\calS$ on $\ell\to\calS$ such that, for all $\lambda_i\in\Gamma(L_i)$ and $i=1,2$,
		\begin{equation}
		\label{eq:connection}
		\left(\nabla_{\calX_{\Phi_i^\ast\lambda_i}|_\calS}\right)\circ\Phi_i|_\calS^\ast=\Phi_i|_\calS^\ast\circ\nabla^i_{\calX_{\lambda_i}},
		\end{equation}
		i.e.~$\nabla$ is the pull-back of $\nabla^i$ along $\Phi_i|_\calS:\ell\to\ell_i$, for $i=1,2$.
		%i.e.~the following holds, for $i=1,2$,
		%\begin{equation}
		%\left(\nabla_{\calX_{\varphi_i^\ast\lambda_i}|_\calS}\right)\circ\iota_\calS^\ast=\iota_\calS^\ast\circ\left(\Delta_{\varphi_i^\ast\lambda_i}\right)
		%\end{equation}
		Further the following relation holds
		\begin{equation}
		\label{eq:relation_lcs}
		\rmd_\nabla(I_\calS^\ast\theta)=\Phi_1|_\calS^\ast\omega_1+\Phi_2|_\calS^\ast\omega_2\in\Omega^2(\calS;\ell),
		\end{equation}
		where $(\Omega^\bullet(\calS;\ell),\rmd_\nabla)$ denotes the de Rham complex of $T\calS$ with values in the representation $\nabla$.
	\end{enumerate}
\end{theorem}

\begin{remark}
	Notice that Equations~\eqref{eq:relation_contact} and~\eqref{eq:relation_lcs} involve the graded module morphisms $I_\calS^\ast\colon\Omega^\bullet(M;L)\to\Omega^\bullet(\calS;\ell)$ and $\Phi_i|_\calS^\ast\colon\Omega^\bullet(\calS_i;\ell_i)\to\Omega^\bullet(\calS;\ell)$, over the graded algebra morphisms $\iota_\calS^\ast\colon\Omega^\bullet(M)\to\Omega^\bullet(\calS)$ and $\varphi_i|_\calS^\ast\colon\Omega^\bullet(\calS_i)\to\Omega^\bullet(\calS)$, induced (as in Remark~\ref{rem:pull-back_VB_valued_forms}) by the regular line bundle morphisms $I_\calS:\ell\to L$ and $\Phi_i|_\calS:\ell\to\ell_i$ respectively, covering $\varphi_i|_\calS:\calS\to\calS_i$ and $\iota_S:\calS\to M$ respectively.
\end{remark}

\begin{remark}
	\label{noncon}
	As for symplectic dual pairs~\cite[Theorem E.15]{blaom2001geometric}, results similar to Theorems~\ref{theor:LeafCorrespondenceI} and~\ref{theor:LeafCorrespondenceII} hold true when, in the full contact dual pair~\eqref{eq:theor:LeafCorrespondenceI:CDP}, only one of the two underlying maps has connected fibers.
	Assuming that $\varphi_2$ has connected fibers, to each characteristic leaf $\calS_1\subset M_1$ one associates the so-called \emph{pseudoleaf} $\calS_2':=\varphi_2(\varphi_1^{-1}(\calS_1))$ of $M_2$, whose connected components are characteristic leaves of $(M_2,L_2,\{-,-\}_2)$.
	Then, generalizing Theorem~\ref{theor:LeafCorrespondenceI}, there exists a 1-1 correspondence between the characteristic leaves $\calS_1$ of $M_1$ and the pseudoleaves $\calS_2'$ of $M_2$.
	Furthermore, also the statement of Theorem~\ref{theor:LeafCorrespondenceII} keeps holding true, up to replacing the characteristic leaves $\calS_2$ with the pseudoleaves $\calS_2'$ of $M_2$.
\end{remark}

\begin{example}\label{ex:contact_groupoids}
	If the contact dual pair comes from a contact groupoid
	$(\calG,\calH)\rightrightarrows\calG_0$, then the characteristic leaf correspondence is the identity.
	Indeed, the foliation $\calC$ is precisely given by the image via the differential of the target $t$ of $\operatorname{span} \{\calX_{T^\ast\lambda}:\lambda\in\Gamma(L_0)\}$, which, by Proposition \ref{prop:LW_CDP:tangent_fibers}, is equal to $\ker Ts$.
	Hence, the leaves of $\calC$ coincide with the orbits of the groupoid i.e.\ the immersed submanifolds $\bigO_x:=t(s^{-1}(x))\subset \calG_0,\ x\in \calG_0$.
	It's now an easy exercise to see that $t(s^{-1}(\bigO_x))=s(t^{-1}(\bigO_x))=\bigO_x$.
\end{example}

\begin{example}[\textbf{The sphere of $\mathfrak{su}(3)$}]
Consider the Lie algebra $\mathfrak{su}(3)=\{A\in M_3(\mathbb{C}) \colon A+A^\ast=0,\ \mathrm{tr}(A)=0\}$
%	\begin{equation*}
%	\mathfrak{su}(3)=\{A\in M_2(\mathbb{C}) \colon A+A^\ast=0,\ \mathrm{tr}(A)=0\}.
%	\end{equation*}
	whose 1-connected Lie group is the compact Lie group $\operatorname{SU}(3)=\{U\in M_3(\bbC) \colon UU^\ast=I,\ \mathrm{det}(U)=1\}$.
%	\begin{equation*}
%	\operatorname{SU}(3)=\{U\in M_3(\bbC) \colon UU^\ast=I,\ \mathrm{det}(U)=1\}.
%	\end{equation*}
	Let $\bbS(\mathfrak{su}(3)^\ast)$ be the sphere on $\mathfrak{su}(3)^\ast$ for the bi-invariant inner product $\langle A,B \rangle=-\operatorname{tr}(AB).$
	Using this inner product we identify $\mathfrak{su}(3)$ with its dual, and define the maps $\bar s,t:\operatorname{SU}(3)\times \bbS(\mathfrak{su}(3)^\ast)\to \bbS(\mathfrak{su}(3)^\ast)$ by setting $t:(U,A)=\Ad^\ast_U A$, $\bar s(U,A):=\bar A$, where the bar denotes the complex conjugation.
	These maps define a (non-strict) full contact dual pair
	\begin{equation}
	\begin{tikzcd}
	\bbS(\mathfrak{su}(3)^\ast)&\operatorname{SU}(3)\times \bbS(\mathfrak{su}(3)^\ast)\arrow[l, swap, "\bar s"]\arrow[r, "t"]&\bbS(\mathfrak{su}(3)^\ast),
	\end{tikzcd}
	\end{equation}
	with the property that the characteristic leaf correspondence, as in Theorem~\ref{theor:LeafCorrespondenceI}, is not the identity map.
	This can be seen as follows: in view of~\cite[Example~2.3]{zambon2006contact}, the maps $s,t:\operatorname{SU}(3)\times \bbS(\mathfrak{su}(3)^\ast)\to \bbS(\mathfrak{su}(3)^\ast)$, where $s(U,A):=A$, are the source and target of the contact groupoid $\operatorname{SU}(3)\times \bbS(\mathfrak{su}(3)^\ast)$ over $\bbS(\mathfrak{su}(3)^\ast)$, hence the induced map on the leaf spaces is the identity (cf.~Remark~\ref{ex:contact_groupoids}).
	On the other hand, $\bar s$ is the composition of $s$ with the Poisson automorphism of $\bbS(\mathfrak{su}(3)^\ast)$ given by conjugation, which, as shown in~\cite{marcut2014liepoisson}, induces a non-trivial map on the leaf space of the Lie--Poisson sphere $\bbS(\mathfrak{su}(3)^\ast)$.
\end{example}

\begin{example}[\textbf{The trivial line bundle case}]
	Let us assume that in the full contact dual pair~\eqref{eq:theor:LeafCorrespondenceI:CDP} all the line bundles, $L=TM/\calH$, $L_1$ and $L_2$, are the trivial ones.
	So diagram~\eqref{eq:theor:LeafCorrespondenceI:CDP} can be rewritten as
	\begin{equation}
	\begin{tikzcd}
	(M_1,\Pi_1,E_1)&(M,\theta)\ar[l, swap, "{(\varphi_1,a_1)}"]\ar[r, "{(\varphi_2,a_2)}"]&(M_2,\Pi_2,E_2),
	\end{tikzcd}
	\end{equation}
	where $\theta\in\Omega^1(M)$ gives the contact structure on $M$ and, for $i=1,2$, Jacobi pair $(\Pi_i,E_i)$ encodes the Jacobi structure on $\bbR_{M_i}$ and $\varphi_i:M\to M_i$ is a Jacobi map with connected fibers and conformal factor $a_i\in C^\infty(M)$.
	Let the characteristic leaves $\calS_1\subset M_1$ and $\calS_2\subset M_2$ correspond to each other as in Theorem~\ref{theor:LeafCorrespondenceI}, with $\calS=\varphi_1^{-1}(\calS_1)=\varphi_2^{-1}(\calS_2)$.
	Then the only two possible cases (cf.~Theorem~\ref{theor:LeafCorrespondenceII}) look as follows.
	\begin{enumerate}[label=(\arabic*)]
		\item $\calS_i$ is odd-dimensional with $\theta_i\in\Omega(\calS_i)$ the inherited coorientable contact structure, for $i=1,2$.
		Then the following relation holds
		\begin{equation}
		\label{eq:relation_contact:coorientable}
		i_\calS^\ast\theta=a_1|_\calS\cdot(\varphi_1|_\calS^\ast\theta_1)+a_2|_\calS\cdot(\varphi_2|_\calS^\ast\theta_2).
		\end{equation}
		\item $\calS_i$ is even-dimensional, with $(\eta_i,\omega_i)$ the inherited l.c.s.~structure on $\bbR_{\calS_i}$, for $i=1,2$.
		Then there is a unique closed $1$-form $\eta\in\Omega^1(\calS)$ such that
		\begin{equation}
		\label{eq:connection:coorientable}
		\eta=-a_i|_\calS^{-1}\rmd a_i|_\calS+\varphi_i|_{\calS}^\ast\eta_i,\ \text{on}\ \ker T\varphi_i,
		\end{equation}
		for $i=1,2$.
		Further the following relation holds
		\begin{equation}
		\label{eq:relation_lcs:coorientable}
		\rmd(\iota_\calS^\ast\theta)-(\iota_\calS^\ast\theta)\wedge\eta=a_1|_\calS\cdot(\varphi_1|_\calS^\ast\omega_1)+a_2|_\calS\cdot(\varphi_2|_\calS^\ast\omega_2).
		\end{equation}
	\end{enumerate}
\end{example}

Coming back to the proofs of the theorems, we remark that, as pointed out by Blaom in~\cite[App.~E]{blaom2001geometric}, the Symplectic Leaf Correspondence Theorem (cf.~Section~\ref{sec:intro}) relies on a technical lemma about integrability and integral foliation of the pull-back of a singular distribution.
The same holds true also for the characteristic leaf correspondence in contact dual pairs.
On this regard recall that, for a singular distribution $\calC\subset TN$ and a smooth map $\varphi:M\to N$, the \emph{pull-back distribution} of $\calC$ along $\varphi$ is the singular distribution on $M$ given by $\varphi^\ast\calC:=(T\varphi)^{-1}\calC\subset TM$.

\begin{lemma}[{\cite[Corollary E.8]{blaom2001geometric}}]
	\label{lem:blaom}
	If $\varphi$ is a submersion and $\calC\subset TN$ is smooth/integrable, then $\varphi^\ast\calC$ is smooth/integrable.
	Further, if $\varphi$ is a surjective submersion with connected fibers and $\calC$ is integrable, then the relation $\calK=\varphi^{-1}(\calS)$ establishes a 1-1 correspondence between the integral leaves $\calS$ of $\calC$ and the integral leaves $\calK$ of $\varphi^\ast\calC$.
\end{lemma}

For our aims it will be helpful to point out also the following result (cf.~\cite[Lemma E.11]{blaom2001geometric} for its analogue in the Poisson setting).

\begin{lemma}
	\label{lem:pullback_distribution}
	Let $\Phi_i:(M,L,\theta)\to(M_i,L_i,\{-,-\}_i)$ be a Jacobi morphism covering $\varphi_i:M\to M_i$.
	If $\varpi\in\Omega_L^2$ denotes the symplectic Atiyah form corresponding to $\theta$ (cf.~Proposition~\ref{prop:contact=symplecticAtiyah}) and $\J_i\in\calD^2L_i$ denotes the Jacobi bi-DO corresponding to $\{-,-\}_i$ (cf.~Proposition~\ref{prop:Jacobi_bi-DOs}), then one obtains, for any $x\in M$,
	\begin{equation}
	\label{eq:integrable_distribution_1}
	(D_x\Phi_i)^{-1}(\im \J_i^\sharp)=\ker D_x\Phi_i+(\ker D_x\Phi_i)^{\perp\varpi}.
	\end{equation}
\end{lemma}

\begin{proof}
	Fix $x\in M$.
	Since $\Phi_i:(M,L,\theta)\to(M_i,L_i,\{-,-\}_i)$ is a Jacobi morphism, in particular we have $(D_x\Phi_i)(\Delta_{\Phi_i^\ast\lambda_i})_x=(\Delta_{\lambda_i})_{\varphi_i(x)}$, for all $\lambda_i\in\Gamma(L_i)$ (cf.~Remark~\ref{rem:Jacobi_morphism}).
	Hence we can easily compute
	\begin{equation*}
	(D_x\Phi_i)^{-1}(\im\J_i^\sharp)=\ker D_x\Phi_i+\{(\Delta_{\Phi_i^\ast\lambda_i})_x:\lambda_i\in\Gamma(L_i)\}=\ker D_x\Phi_i+(\ker D_x\Phi_i)^{\perp\varpi},
	\end{equation*}
	where in the very last step we have made use of Equation~\eqref{eq:LW_CDP_Atiyah_forms:2}.
\end{proof}

\begin{proof}[Proof of Theorem~\ref{theor:LeafCorrespondenceI}]
	Combining together Proposition~\ref{prop:LW_CDP_Atiyah_forms} and Lemma~\ref{lem:pullback_distribution} we obtain, for any $x\in M$, that
	\begin{equation}
		\label{eq:proof:theor:LeafCorrespondenceI:pullback_distribution:1}
		(D_x\Phi_1)^{-1}(\im \J_1^\sharp)=(D_x\Phi_2)^{-1}(\im \J_2^\sharp)=\ker D_x\Phi_1+\ker D_x\Phi_2.
	\end{equation}
	So, applying the symbol map $\sigma:DL\to TM$ to both sides of the latter, we get the following
	\begin{equation}
		\label{eq:proof:theor:LeafCorrespondenceI:pullback_distribution:2}
		(T_x\varphi_1)^{-1}(\im (\sigma\circ \J_1^\sharp))=(T_x\varphi_2)^{-1}(\im (\sigma\circ \J_2^\sharp))=\ker T_x\varphi_1+\ker T_x\varphi_2.
	\end{equation}
	In the last step we have also used the following two identities.
	First, $\sigma(\ker D_x\Phi_i)=\ker T_x\varphi_i$, which is a straightforward consequence of Equation~\eqref{eq:LW_CDP_Atiyah_forms:3}, and second, $\sigma((D_x\Phi_i)^{-1}(\im\J^\sharp_i))=(T_x\varphi_i)^{-1}(\im(\sigma\circ\J_i^\sharp))$ which follows from the fact that $\Phi_i$ is a Jacobi morphism (cf.~Remark~\ref{rem:Jacobi_morphism}).
	In view of Equation~\eqref{eq:proof:theor:LeafCorrespondenceI:pullback_distribution:2}, the singular distribution $\calD:=\ker T\varphi_1+\ker T\varphi_2$ is the pull-back distribution, along $\varphi_i$, of the characteristic distribution $\calC_i:=\im(\sigma\circ\calJ_i^\sharp)$, with $i=1,2$.
	So the stated 1-1 correspondences follow by simply applying Lemma~\ref{lem:blaom} to this setting.
	Finally, a dimension count that uses the dimensional relation in Equation~\eqref{eq:dimensional_relation} proves the last statement of the theorem. 
\end{proof}

\begin{proof}[Proof of Theorem~\ref{theor:LeafCorrespondenceII}]
Since $\calS_i$ is a characteristic leaf of $(M_i,L_i,\{-,-\}_i)$, we get that
\begin{equation}
\label{eq:Dl_i:0}
D\ell_i=\sigma^{-1}(T\calS_i)=(\langle\mathbbm{1}\rangle+\im \J_i^\sharp)|_{\calS_i}.
\end{equation}
Further, since $\rank(D\ell_i)=1+\dim\calS_i$, and $\J_i$ is skew-symmetric (so that $\rank (\im \J^\sharp)$ is even), Equation~\eqref{eq:Dl_i:0} can be rewritten as follows:
\begin{equation}
\label{eq:Dl_i}
D\ell_i=
	\begin{cases}
	(\langle\mathbbm{1}\rangle\oplus\im \J_i^\sharp)|_{\calS_i},&\text{if $\dim\calS_i$=even},\\
	(\im \J_i^\sharp)|_{\calS_i},&\text{if $\dim\calS_i$=odd}.
	\end{cases}
\end{equation}
Using the hypothesis $\calS=\varphi_i^{-1}(\calS_i)$, and the fact that $\sigma\circ(D\Phi_i)=(T\varphi_i)\circ\sigma$, we can also write:
\begin{equation}
\label{eq:Dl}
D\ell=\sigma^{-1}(T\calS)=\sigma^{-1}(T\varphi_i)^{-1}(T\calS_i)=(D\Phi_i)^{-1}\sigma^{-1}(T\calS_i)=(D\Phi_i)^{-1}(D\ell_i).
\end{equation}
Combining together Equations~\eqref{eq:proof:theor:LeafCorrespondenceI:pullback_distribution:1}, \eqref{eq:Dl_i} and~\eqref{eq:Dl}, we get
\begin{equation}
\label{eq:Dl_cases}
D\ell=
\begin{cases}
\left.\left(\langle\mathbbm{1}\rangle\oplus\left(\ker D\Phi_1+\ker D\Phi_2\right)\right)\right|_\calS,&\text{if $\dim\calS_i$=even},\\
\left.\left(\ker D\Phi_1+\ker D\Phi_2\right)\right|_\calS,&\text{if $\dim\calS_i$=odd},
\end{cases}
\end{equation}
where, in the case $\dim\calS_i=\text{even}$, we have also used the fact that $\Phi_i^\ast\circ\mathbbm{1}=\mathbbm{1}\circ\Phi_i^\ast$.

Notice that Equation~\eqref{eq:Dl_cases}, when $\dim\calS_i=\text{even}$, provides a splitting of the short exact sequence
$$0\longrightarrow\langle\mathbbm{1}\rangle\overset{\text{incl}}{\longrightarrow} D\ell\overset{\sigma}{\longrightarrow} T\calS\longrightarrow 0$$
and so, it is also equivalent to the existence of a (unique) $T\calS$-connection $\nabla$ on $\ell\to\calS$ such that
\begin{equation*}
	\label{eq:connection:1}
	\{\nabla_\xi\in D_x\ell:x\in\calS,\ \xi\in T_x\calS\}=\left.\left(\ker D\Phi_1+\ker D\Phi_2\right)\right|_\calS.
\end{equation*}
In view of Equations~\eqref{eq:rem:Jacobi_morphism}, \eqref{eq:cor:LW_CDP:fiber_tangent_derivations} and \eqref{eq:even-dim_transitive_Jacobi}, the definition of $\nabla$ can be rephrased as in Equation~\eqref{eq:connection}, i.e.
\begin{equation}
\label{eq:connection:2}
\left(\nabla_{\calX_{\Phi_i^\ast\lambda_i}|_\calS}\right)\circ\Phi_i|_\calS^\ast\equiv\Phi_i|_\calS^\ast\circ\left(\Delta_{\lambda_i}|_{\calS_i}\right)=\Phi_i|_\calS^\ast\circ\nabla^i_{\calX_{\lambda_i}}:\Gamma(\ell_i)\to\Gamma(\ell),
\end{equation}
for all $i=1,2$, and $\lambda_i\in\Gamma(L_i)$.
Further, as a straightforward consequence of the Jacobi identity for $\{-,-\}$, we also get that the $T\calS$-connection $\nabla$ on $\ell$ is flat.
%Indeed
%\begin{equation*}
%[\nabla_{\calX_{\varphi_i^\ast\lambda_i}|_\calS},\nabla_{\calX_{\varphi_j^\ast\lambda_j}|_\calS}]=[\Delta_{\varphi_i^\ast\lambda_i}|_\calS,\Delta_{\varphi_j^\ast\lambda_j}|_\calS]=\Delta_{\{\varphi_i^\ast\lambda_i,\varphi_j^\ast\lambda_j\}}|_\calS=\nabla_{\calX_{\{\varphi_i^\ast\lambda_i,\varphi_j^\ast\lambda_j\}}|_\calS}=\nabla_{[\calX_{\varphi_i^\ast\lambda_i},\calX_{\varphi_j^\ast\lambda_j}]|_\calS},
%\end{equation*}
%for all $i,j=1,2$, $\lambda_i\in\Gamma(L_i)$, and $\lambda_j\in\Gamma(L_j)$.

Now we are ready to prove the relations~\eqref{eq:relation_contact} and~\eqref{eq:relation_lcs}.
First, since $\Phi_i:(M,L,\{-,-\})\to(M_i,L_i,\{-,-\}_i)$ is a Jacobi morphism, for all $\lambda'_i,\lambda''_i\in\calP_i:=\Phi_i^\ast\Gamma(L_i)$, with $i=1,2$, we can write:
\begin{equation}
\label{eq:proof:theor:LeafCorrespondenceII:aux1}
\varpi(\Delta_{\lambda'_i},\Delta_{\lambda''_i})|_\calS=
\begin{cases}
(\Phi_i|_\calS^\ast\omega_i)(\calX_{\lambda'_i},\calX_{\lambda''_i}),&\text{if $\dim\calS_i=\text{even}$},\\
(\Phi_i|_\calS^\ast\varpi_i)(\Delta_{\lambda'_i},\Delta_{\lambda''_i}),&\text{if $\dim\calS_i=\text{odd}$},
\end{cases}
\end{equation}
where, in accordance with Proposition~\ref{prop:contact=symplecticAtiyah}, $\varpi$ denotes the $L$-valued symplectic Atiyah form corresponding to the contact form $\theta\in\Omega^1(M;L)$ and, if $\dim\calS_i=\text{odd}$, $\varpi_i$ denotes the $\ell_i$-valued symplectic Atiyah form corresponding to the contact form $\theta_i\in\Omega^1(\calS_i;\ell_i)$.
Further, because of Condition~\ref{enumitem:def:CDP_2} in Definition~\ref{def:CDP}, we also get that, in particular,
\begin{equation}
\label{eq:proof:theor:LeafCorrespondenceII:aux2}
\varpi(\Delta_{\lambda'_1},\Delta_{\lambda''_2})|_{\calS}=\varpi(\Delta_{\lambda'_2},\Delta_{\lambda''_1})|_{\calS}=0.
\end{equation}
Fix arbitrarily $x\in M$ and $\lambda'_i,\lambda''_i\in\calP_i:=\Phi_i^\ast\Gamma(L_i)$, with $i=1,2$, and set
\begin{align*}
\delta'_i&:=(\Delta_{\lambda_i'})_x,\ \delta''_i:=(\Delta_{\lambda_i''})_x\in D_xL,\qquad\text{and}\qquad \xi'_i:=(\calX_{\lambda_i'})_x,\ \xi''_i:=(\calX_{\lambda_i''})_x\in T_x\calS.
\end{align*}
Further define $\delta':=\delta'_1+\delta'_2,\ \delta'':=\delta''_1+\delta''_2\in D_xL$ and $\xi':=\xi_1'+\xi_2',\ \xi'':=\xi_1''+\xi_2''\in T_x\calS$.
In the following, we will consider separately the cases $\dim\calS_i=\text{even}$ and $\dim\calS_i=\text{odd}$.

Let us start with the case $\dim\calS_i=\text{even}$.
Using the defining property~\eqref{eq:connection:1} of $\nabla$, we can compute:
\begin{align*}
(d_\nabla(I_\calS^\ast\theta))(\xi',\xi'')\overset{\phantom{\eqref{eq:proof:theor:LeafCorrespondenceII:aux2}}}{=}&(d_\nabla(I_\calS^\ast\theta))(\xi'_1,\xi''_1)+(d_\nabla(I_\calS^\ast\theta))(\xi'_2,\xi''_1)+(d_\nabla(I_\calS^\ast\theta))(\xi'_1,\xi''_2)+(d_\nabla(I_\calS^\ast\theta))(\xi'_2,\xi''_2)\\
\overset{\eqref{eq:proof:theor:LeafCorrespondenceII:aux2}}{=}&\varpi(\delta'_1,\delta''_1)\,+\,{\underbrace{\cancel{\varpi(\delta'_1,\delta''_2)}}_{=0}}\,+\,{\underbrace{\cancel{\varpi(\delta'_2,\delta''_1)}}_{=0}}\,+\,\varpi(\delta'_2,\delta''_2)\\
\overset{\eqref{eq:proof:theor:LeafCorrespondenceII:aux1}}{=}&(\Phi_1|_\calS^\ast\omega_1)(\xi'_1,\xi''_1)+(\Phi_2|_\calS^\ast\omega_2)(\xi'_2,\xi''_2)\\
\overset{\phantom{\eqref{eq:proof:theor:LeafCorrespondenceII:aux2}}}{=}&(\Phi_1|_\calS^\ast\omega_1+\Phi_2|_\calS^\ast\omega_2)(\xi',\xi''),
\end{align*}
where, in the very last step, we have used Equation~\eqref{eq:prop:LW_CDP:tangent_fibers:1} in Proposition~\ref{prop:LW_CDP:tangent_fibers}.
In view of the identity $T_x\calS=\ker T_x\varphi_1+\ker T_x\varphi_2$, and the arbitrariness of $\xi_1',\xi_1''\in\ker T_x\varphi_2$ and $\xi_2',\xi_2''\in\ker T_x\varphi_1$, the latter means exactly that $d_\nabla(I^\ast_\calS\theta)=\Phi_1|_\calS^\ast\omega_1+\Phi_2|_\calS^\ast\omega_2$, as we needed to prove.

Let us continue with the case $\dim\calS_i=\text{odd}$.
We can compute:
\begin{align*}
\varpi(\delta',\delta'')\overset{\eqref{eq:proof:theor:LeafCorrespondenceII:aux2}}{=}&\varpi(\delta'_1,\delta''_1)\,+\,{\underbrace{\cancel{\varpi(\delta'_1,\delta''_2)}}_{=0}}\,+\,{\underbrace{\cancel{\varpi(\delta'_2,\delta''_1)}}_{=0}}\,+\,\varpi(\delta'_2,\delta''_2)\\
\overset{\eqref{eq:proof:theor:LeafCorrespondenceII:aux1}}{=}&(\Phi_1|_\calS^\ast\varpi_1)(\delta'_1,\delta''_1)+(\Phi_2|_\calS^\ast\varpi_2)(\delta'_2,\delta''_2)\\
\overset{\phantom{\eqref{eq:proof:theor:LeafCorrespondenceII:aux2}}}{=}&(\Phi_1|_\calS^\ast\varpi_1+\Phi_2|_\calS^\ast\varpi_2)(\delta',\delta''),
\end{align*}
where, in the very last step, we have used Equation~\eqref{eq:cor:LW_CDP:fiber_tangent_derivations} in Corollary~\ref{cor:LW_CDP:fiber_tangent_derivations}.
In view of the special case of Equation~\eqref{eq:Dl_cases}, for $\dim\calS_i=\text{odd}$, and the arbitrariness of $\delta'_1,\delta''_1\in\ker D_x\Phi_2$ and $\delta'_2,\delta''_2\in\ker D_x\Phi_1$, the latter means exactly that $I_\calS^\ast\varpi=\Phi_1|_\calS^\ast\varpi_1+\Phi_2|_\calS^\ast\varpi_2$.
Since both sides are 2-cocycles in the der-complex $(\Omega^\bullet_\ell,d_D)$, the latter can also be equivalently rephrased as $I_\calS^\ast(\theta\circ\sigma)=\Phi_1|_\calS^\ast(\theta_1\circ\sigma)+\Phi_2|_\calS^\ast(\theta_2\circ\sigma)$, up to applying component-wise the contracting homotopy $\iota_\mathbbm{1}$ of the der-complex (cf.~Appendix~\ref{sec:der-complex}).
Finally, the surjectivity of the symbol maps implies that $I_\calS^\ast\theta=\Phi_1|_\calS^\ast\theta_1+\Phi_2|_\calS^\ast\theta_2$, and this completes the proof.
\end{proof}

\subsection{Transverse Structure to Corresponding Characteristic Leaves}
\label{sec:transverse_structure_corresponding_char_leaves}

The study of the local structure of Jacobi manifolds was initiated by Dazord, Lichnerowicz and Marle~\cite{dazord1991structure}.
In particular, they proved the existence of certain structures transverse to the characteristic leaves of a Jacobi manifold.
More recently, this result has been generalized to Dirac--Jacobi manifolds by Vitagliano~\cite{vitagliano2018djbundles}.
In this section we describe the transverse structures to corresponding characteristic leaves in a contact dual pair; for this purpose, we present below, without proof, a proposition that, following mainly~\cite{vitagliano2018djbundles}, describes the transverse structure to characteristic leaves of a Jacobi manifold.
We follow the notation and operations described in the Appendix \ref{sec:Dirac-Jacobi} about the Omni-Lie algebroid and Dirac--Jacobi structures.

\begin{proposition}
	\label{prop:transverse_structure}
	Let $\calS$ be a characteristic leaf of a Jacobi manifold $(M,L,\J=\{-,-\})$.
	\newline
	1) 
	Let $x\in\calS$ and let $Q\subset M$ be a submanifold transverse to $\calS$ at $x$, with $T_xM=T_x\calS\oplus T_xQ$.
	Denote by $I_Q:L|_Q\to L$ the regular line bundle morphism, covering $i_Q:Q\to M$, given by the inclusion.
	Then ${I_Q}^!\operatorname{Gr}\J$ is a Dirac--Jacobi structure on $L_Q\to Q$.
	Specifically:
	\begin{itemize}
		\item if $\calS$ is a l.c.s.~leaf, there exists a unique Jacobi structure $\J_Q$ on $L_Q\to Q$ such that ${I_Q}^!\operatorname{Gr}\J=\operatorname{Gr}\J_Q$,
		\item if $\calS$ is a contact leaf, then ${I_Q}^!\operatorname{Gr}\J$ is of \emph{homogeneous Poisson type} (cf.~\cite[Definition 2.14]{schnitzer2019normal}).
	\end{itemize}
	2) Let $x,y\in\calS$ and let $Q,P\subset M$ be submanifolds transverse to $\calS$ at $x$ and $y$ respectively, with $T_xM=T_x\calS\oplus T_xQ$ and $T_yM=T_y\calS\oplus T_yP$.
	Then there exists a local line bundle isomorphism $\Phi:L|_Q\to L|_P$, covering $\varphi:Q\to P$, such that $\varphi(x)=y$ and $\bbD\Phi({I_Q}^!\operatorname{Gr}\J)={I_P}^!\operatorname{Gr}\J$.
\end{proposition}

For more details, including a proof, we refer the reader to~\cite[Propositions 6.8 and 6.9]{vitagliano2018djbundles} (see also~\cite{schnitzer2019normal}).

\begin{theorem}[\textbf{Characteristic Leaf Correspondence III}]
	\label{theor:transverse_structure}
	Consider a full contact dual pair
	\begin{equation}
		\begin{tikzcd}
			(M_1,L_1,\J_1=\{-,-\}_1)&(M,L,\theta)\arrow[l, swap, "\Phi_1"]\arrow[r, "\Phi_2"]&(M_2,L_2,\J_2=\{-,-\}_2)
		\end{tikzcd}
	\end{equation}
	whose underlying maps 
	$\!\!\begin{tikzcd}
	M_1&M\arrow[l, swap, "\varphi_1"]\arrow[r, "\varphi_2"]&M_2\end{tikzcd}\!\!\!$ have connected fibers.
	Let $\calS_1\subset M_1$ and $\calS_2\subset M_2$ be corresponding characteristic leaves, as in Theorem~\ref{theor:LeafCorrespondenceI}.
	Then the transverse structures to $\calS_1$ and $\calS_2$ are anti-isomorphic.
	More specifically, for any $x_i\in\calS_i$ and $Q_i\subset M_i$ submanifold transverse to $\calS_i$ at $x_i$, with $T_{x_i}M_i=T_{x_i}\calS_i\oplus T_{x_i}Q_i$, for $i=1,2$, there exists a local line bundle isomorphism $\Psi:L_1|_{Q_1}\to L_2|_{Q_2}$, covering a local diffeomorphism $\psi:Q_1\to Q_2$, such that $\psi(x_1)=x_2$ and
	\begin{equation}
		\bbD\Psi({I_{Q_1}}^!\operatorname{Gr}\J_1)=-({I_{Q_2}}^!\operatorname{Gr}\J_2),
	\end{equation}
	where $I_{Q_i}:L_i|_{Q_i}\to L_i$ denotes the inclusion, for $i=1,2$, and $\bbD\Psi$ is the isomorphism of omni-Lie algebroids $\bbD(L_1|_{Q_1})\to \bbD(L_2|_{Q_2})$ determined by $\Psi$ (see Appendix~\ref{sec:Dirac-Jacobi}).
\end{theorem}

\begin{proof}
	Set $\calS=\varphi_1^{-1}(\calS_1)=\varphi_2^{-1}(\calS_2)$.
	In view of Theorem~\ref{theor:LeafCorrespondenceI}, it is possible to choose $x\in\varphi_1^{-1}(x_1)\cap\varphi_2^{-1}(x_2)\subset\calS$, and additionally $Q\subset M$, a submanifold transverse to $\calS$ at $x$, such that $T_xM=T_x\calS\oplus T_xQ$ and
	\begin{equation}
		\label{eq:proof:theor:transverse_structure:a}
		T_xQ\subset\calH_x:=\ker\theta_x.
	\end{equation}
	Notice that we can assume this last property of $Q$ because of Condition~\ref{enumitem:def:CDP_1} in Definition~\ref{def:CDP}.
	Then, for $i=1,2$, the submanifold $Q_i:=\varphi_i(Q)\subset M$ is transverse to $\calS_i$ at $x_i$ with $T_{x_i}M_i=T_{x_i}\calS_i\oplus T_{x_i}Q_i$.
	Denote by $I_Q:L|_Q\to L$ and $I_{Q_i}:L_i|_{Q_i}\to L_i$, for $i=1,2$, the regular line bundle morphisms given by the inclusions.
	Up to replacing $M$ and $M_i$ with suitable neighbourhoods of $x$ and $x_i$ respectively, one can assume, for $i=1,2$, that $\Phi_i:L\to L_i$ induces a line bundle isomorphism $\Psi_i:L|_Q\to L_i|_{Q_i}$, covering a diffeomorphism $\psi_i:Q\to Q_i$, such that
	\begin{equation}
		\label{eq:proof:theor:transverse_structure:b}
		\Phi_i\circ I_Q=I_{Q_i}\circ\Psi_i.
	\end{equation} 
	Now, denoting by $\varpi\in\Omega_L^2$ the symplectic Atiyah form corresponding to $\theta$ according to Proposition~\ref{prop:contact=symplecticAtiyah}, and making use of Proposition~\ref{prop:CDP_Dirac-Jacobi}, one can easily compute
	\begin{multline*}
		{\Psi_1}^!{I_{Q_1}}^!\operatorname{Gr}\J_1\overset{\eqref{eq:proof:theor:transverse_structure:b}}{=}{I_Q}^!{\Phi_1}^!\operatorname{Gr}\J_1\overset{\eqref{eq:prop:CDP_Dirac-Jacobi}}{=}{I_Q}^!\calR_\varpi{\Phi_2}^!\operatorname{Gr}(-\J_2)=\calR_{I_Q^\ast\varpi}{I_Q}^!{\Phi_2}^!\operatorname{Gr}(-\J_2)\\={I_Q}^!{\Phi_2}^!\operatorname{Gr}(-\J_2)\overset{\eqref{eq:proof:theor:transverse_structure:b}}{=}{\Psi_2}^!{I_{Q_2}}^!\operatorname{Gr}(-\J_2)={\Psi_2}^!(-({I_{Q_2}}^!\operatorname{Gr}\J_2)).
	\end{multline*}
	Above we also used the fact that $I_Q^\ast\varpi=0$ because of Equation~\eqref{eq:proof:theor:transverse_structure:a}.
	The latter implies that $\bbD\Psi({I_{Q_1}}^!\operatorname{Gr}\J_1)=-({I_{Q_2}}^!\operatorname{Gr}\J_2)$, where $\Psi:L_1|_{Q_1}\to L_2|_{Q_2}$ is the line bundle isomorphism defined by $\Psi:=\Psi_2\circ\Psi_1^{-1}$. 
\end{proof}

\section{Contact Dual Pairs from Contact Reduction}
\label{sec:contact_groupoid_action_CDP}

In addition to contact groupoids, another source of examples for contact dual pairs is represented by contact reduction.
Following~\cite{zambon2006contact}, we introduce and study, in the generically non-coorientable case, contact actions of contact groupoids on contact manifolds (see Proposition~\ref{prop:contact_action} and Definition~\ref{def:contact_action}).
As an interesting special case, we also consider, in Proposition~\ref{prop:Lie_group_actions}, a Lie group $G$ acting by contactomorphisms on a contact manifold $(M,\calH)$ such that the $G$-orbits are transverse to $\calH$.
The main result of this section, namely Theorem~\ref{theor:CDP_contact groupoid_action}, proves that contact dual pairs naturally emerge from free, proper, and contact actions of source-connected contact groupoids on contact manifolds.
Finally, as a by-product of the latter, Corollary~\ref{cor:CDP_contact groupoid_action} gives an alternative proof of a result (cf.~\cite[Theorem 4.4]{zambon2006contact}) first proven, by different techniques, in the coorientable case.

\subsection{Contact actions by Lie groupoids}
\label{sec:contact_actions}

Let $\calG\rightrightarrows\calG_0$ be a Lie groupoid, with structure maps $s,t:\calG\to\calG_0$, $m:\calG^{(2)}\to\calG$, $u:\calG_0\to\calG$, $i:\calG\to\calG$, and let $\bbJ:M\to\calG_0$ be a smooth map.
Recall that a \emph{(left) action}~\cite[Definition 3.2]{mikami1988moments} of the Lie groupoid $\calG\rightrightarrows\calG_0$ on the manifold $M$ with \emph{moment map} $\moment:M\to\calG_0$ consists of a map
\begin{equation}
\label{eq:action}
\Phi:\calG{}_s{\times}_\moment M\to M,\ (g,x)\mapsto\Phi_g(x)=g\cdot x,
\end{equation}
such that the following action axioms are satisfied, for all $(h,g)\in\calG^{(2)}$ and $(g,x)\in\calG{}_s{\times}_\moment M$:
\begin{equation}
\label{eq:action_axioms}
\moment(g\cdot x)=t(g),\quad (hg)\cdot x=h\cdot(g\cdot x),\quad \moment(x)\cdot x=x.
\end{equation}
Additionally, the action $\Phi$ is called \emph{free} if, for any $x\in M$, the stabilizer group of $x$ is trivial, and it is called \emph{proper} if the map $\calG{}_s{\times}_\moment M\to M\times M$, $(g,x)\mapsto(g\cdot x,x)$, is proper.

\begin{remark}
	\label{rem:left_multiplication_groupoid}
	The action of the Lie groupoid $\calG\rightrightarrows\calG_0$ on itself by left translations is an instance of (left) action with moment map $t:\calG\to\calG_0$.
\end{remark}

%Let $\Phi$ be a (right) action of a Lie groupoid $\calG\rightrightarrows\calG_0$ on a manifold $M$ along a moment map $\moment:M\to\calG_0$.
%Let us recall that:
%\begin{itemize}
%	\item the groupoid action $\Phi$ is \emph{free} if, for any $x\in M$, the stabilizer group of $x$ is trivial,
%	\item the groupoid action $\Phi$ is \emph{proper} if the map $M{}_\moment{\times}_t\calG\to M\times M$, $(x,g)\mapsto(x,\Phi(x,g))$, is proper.
%\end{itemize}

Assume now that $M$ and $\calG$ are equipped with a contact distribution $\calH_M$ and a multiplicative contact distribution $\calH_\calG$ respectively.
Set $L_M:=(TM)/\calH_M$, $L:=(T\calG)/\calH_\calG$ and $L_0:=L|_{\calG_0}$.
Denote by $\theta_M\in\Omega^1(M;L_M)$ the corresponding contact form on $M$ (cf.~Section~\ref{sec:contact_structures}) and by $\theta_{\calG}\in\Omega^1(\calG;t^\ast L_0)$ the corresponding multiplicative contact form on $\calG$ (cf.~Section~\ref{sec:contact_groupoids}).

\begin{proposition}
	\label{prop:contact_action}
	For the action $\Phi$ in Equation~\eqref{eq:action}, the following two conditions are equivalent.
	\begin{enumerate}[label=(\arabic*)]
		\item
		\label{enumitem:prop:contact_action:distributions}
		For all $(\xi,u)\in T(\calG{}_s{\times}_\moment M)$, with $u\in\calH_M$, one gets that $(T\Phi)(\xi,u)\in\calH_M$ iff $\xi\in\calH_\calG$.
		\item
		\label{enumitem:prop:contact_action:forms}
		There exists a \emph{fat moment map} $\hat{\moment}:L_M\to L_0$, i.e.~a regular line bundle morphism, covering $\moment:M\to\mathcal{G}_0$,
		%(so that, as a consequence, $L_M$ identifies with $\moment^\ast L_{\mathcal{G}_0}$ via $\hat{\moment}$),
		such that, for all $(g,x)\in\calG{}_s{\times}_\moment M$,
		\begin{equation}
		\label{eq:contact_action}
		\theta_M\circ T_{(g,x)}\Phi =(\theta_{\calG}\circ T_{(g,x)}\text{pr}_{\mathcal{G}}) + g\cdot(\theta_M\circ T_{(g,x)}\text{pr}_M),
		\end{equation}
		where $\text{pr}_\calG:\calG{}_s{\times}_\moment M\to\calG$ and $\text{pr}_M:\calG{}_s{\times}_\moment M\to M$ are the standard projections.
		In a more compact way, the latter can also be rewritten as follows:
		\begin{equation}
		\label{eq:contact_action_bis}
		(\Phi^\ast\theta_M)_{(g,x)}=(\text{pr}_{\calG}^\ast\theta_{\calG})_{(g,x)} + g\cdot(\text{pr}_M^\ast\theta_M)_{(g,x)}.
		\end{equation}
	\end{enumerate}
\end{proposition}

\begin{remark}
	\label{rem:def:contact_action}
	Notice that Equation~\eqref{eq:contact_action} makes sense because the fat moment map $\hat\moment:L_M\to L_0$ identifies $L_M$ with $\moment^\ast L_0$.
	Indeed, in Equation~\eqref{eq:contact_action}, not only the first summand of its RHS takes values in $L_{0,t(g)}$, but also its LHS and the second summand of its RHS take values in $L_{0,\moment(g\cdot x)}=L_{0,t(g)}=g\cdot L_{0,\moment(x)}$.
	Above we have also used the representation of $\calG$ on $L_0$ induced by the multiplicative contact structure (cf.~Section~\ref{sec:contact_groupoids}). 
\end{remark}

\begin{remark}
	\label{rem:lifted_action}
	For later use in Section~\ref{subsec:global_contact_reduction}, we point out that, with the same assumptions as above, for any choice of a fat moment map $\hat{\moment}$ as in Proposition~\ref{prop:contact_action}~\ref{enumitem:prop:contact_action:forms}, there exists a unique \emph{fat action} $\hat{\Phi}:L{}_S{\times}_{\hat{\moment}}L_M\to L_M$.
	By fat action we mean an action $\hat{\Phi}$ of the line bundle groupoid $L\rightrightarrows L_0$, as introduced in Remark~\ref{rem:multiplicative_forms}, on $L_M$ with moment map $\hat{\moment}$ which is additionally a regular line bundle morphism from $L{}_S{\times}_{\hat{\moment}}L_M\to\calG{}_s{\times}_\moment M$ to $L_M\to M$ covering $\Phi:\calG{}_s{\times}_\moment M\to M$.
	Indeed such $\hat{\Phi}$ is uniquely determined as follows
	\begin{equation}
	\label{eq:rem:lifted_action}
	\hat{\moment}_{g\cdot x}(\hat{\Phi}(\lambda_g,\lambda_x))=g\cdot(\hat{\moment}_x(\lambda_x)),
	\end{equation}
	for all $(g,x)\in\calG{}_s{\times}_\moment M$, $\lambda_g\in L_g$ and $\lambda_x\in L_{M,x}$ such that $S(\lambda_g)=\hat{\moment}_x(\lambda_x)$.
\end{remark}

Proposition~\ref{prop:contact_action}, whose proof is presented below, justifies the following definition of contact action.

\begin{definition}
	\label{def:contact_action}
	An action $\Phi$, as in Equation~\eqref{eq:action}, of the contact groupoid $\calG\rightrightarrows\calG_0$ on the contact manifold $M$ with moment map $\moment:M\to\calG_0$ is called \emph{contact} if it satisfies the equivalent conditions~\ref{enumitem:prop:contact_action:distributions} and~\ref{enumitem:prop:contact_action:forms} in Proposition~\ref{prop:contact_action}.
\end{definition}

\begin{remark}
	\label{rem:left_multiplication_contact_groupoid}
	The action of the contact Lie groupoid $\calG\rightrightarrows\calG_0$ on itself by left translations is an example of contact (left) action with moment map $t:\calG\to\calG_0$, and fat moment map $T:L\to L_0$ (cf.~Section~\ref{sec:contact_groupoids}).
\end{remark}

One way to produce contact actions of Lie groupoids is as follows. Let $(\Gamma,\omega_\Gamma)\rightrightarrows\Gamma_0$ be a homogeneous symplectic groupoid (cf.~Definition~\ref{def:homogeneous_symplectic_groupoid}) and let $(\calG,\calH)\rightrightarrows\calG_0$ be the corresponding contact groupoid obtained by dehomogeneization (cf.~Corollary~\ref{cor:homogeneous_symplectic_groupoids}), so that, in particular, $\calG=\Gamma/\bbR^\times$ and $\calG_0=\Gamma_0/\bbR^\times$.
Let $(P,\omega_P)$ be a homogeneous symplectic manifold (cf.~Definition~\ref{def:homogeneous_symplectic/Poisson}) and let $(M,\calH_M)$ be the corresponding contact manifold obtained by dehomogeneization (cf.~Proposition~\ref{prop:equivalence_categories:2}), so that, in particular, $M=P/\bbR^\times$.
Consider additionally a (left) action of the Lie groupoid $\Gamma\rightrightarrows\Gamma_0$ on the manifold $P$, with moment map $\moment:P\to\Gamma_0$,
\begin{equation}
\label{eq:homogeneous_symplectic_action}
\Phi:\Gamma{}_s{\times}_\moment P\to P,\ (\gamma,p)\mapsto\Phi_\gamma(p)=\gamma\cdot p.
\end{equation}
This action is called \emph{homogeneous} if both $\moment:P\to\Gamma_0$ and $\Phi:\Gamma{}_s{\times}_\moment P\to P$ are $\bbR^\times$-equivariant, where the fibered product $\Gamma{}_s{\times}_\moment P$ is equipped with the diagonal $\bbR^\times$-action.
Further, the action $\Phi$ is called \emph{symplectic} if $\Phi^\ast\omega_P=\text{pr}_\Gamma^\ast\omega_\Gamma+\text{pr}_P^\ast\omega_P$, where $\text{pr}_\Gamma:\Gamma{}_s{\times}_\moment P\to\Gamma$ and $\text{pr}_P:\Gamma{}_s{\times}_\moment P\to P$ are the standard projections.

\begin{proposition}
	\label{prop:homogeneous_symplectic_action}
	Let the action $\Phi$, as in Equation~\eqref{eq:homogeneous_symplectic_action}, be homogeneous and symplectic.
	Then it gives rise, by dehomogeneization, to a contact action $\tilde{\Phi}$ of the contact groupoid $(\calG,\calH_\calG)\rightrightarrows\calG_0$ on the contact manifold $(M,\calH_M)$, with moment map $\tilde{\moment}:M\to \calG_0$,
	\begin{equation}
	\label{eq:prop:homogeneous_symplectic_action}
	\tilde\Phi:\calG{}_s{\times}_{\tilde\moment} M\to M,\ (g,x)\mapsto\tilde\Phi_g(x)=g\cdot x,
	\end{equation}
	which are defined by $\tilde{\moment}([p])=[\moment(p)]$ and $[\gamma]\cdot[p]=[\gamma\cdot p]$, for all $(\gamma,p)\in\Gamma{}_s{\times}_\moment P$.
\end{proposition}

Since the proof of Proposition~\ref{prop:homogeneous_symplectic_action} works along the same lines as the afore-mentioned correspondence between contact groupoids and homogeneous symplectic groupoids (see~\cite[Theorem 5.8]{bruce2017remarks}), we omit it.

\begin{proof}[{Proof of Proposition~\ref{prop:contact_action}}]
	\ref{enumitem:prop:contact_action:distributions} $\Rightarrow$ \ref{enumitem:prop:contact_action:forms}.
	The trasversality property of contact groupoids (see Equation~\eqref{eq:CDP_from_contact_grpd:1st_condition}), together with Condition~\ref{enumitem:prop:contact_action:distributions}, guarantees that there exists a unique regular line bundle morphism $\hat\moment:L_M\to L_0$, covering $\moment:M\to\calG_0$, such that
	\begin{equation}
		\label{eq:proof:prop:contact_actions:forms:0}
	(\hat\moment_x\circ\theta_{M,x}\circ T_{(\moment x,x)}\Phi)(\xi_{\moment x},0_x)=\theta_{\calG,\moment x}(\xi_{\moment x}),
	\end{equation}
	for all $x\in M$ and $\xi_{\bbJ(x)}\in \ker T_{\bbJ(x)}s$.
	Additionally, such $\hat\moment$ satisfies also the following two conditions
	\begin{gather}
	\label{eq:proof:prop:contact_actions:forms:1}
	(\hat\moment_{g\cdot x}\circ\theta_{M,g\cdot x}\circ T_{(g,x)}\Phi)(\xi_g,0_x)=\theta_{\calG,g}(\xi_g),\\
	\label{eq:proof:prop:contact_actions:forms:2}
	(\hat\moment_{g\cdot x}\circ\theta_{M,g\cdot x}\circ T_{(g,x)}\Phi)((T_{sg}\Sigma)(T_x\moment)u_x,u_x)=g\cdot(\hat{\bbJ}_x\circ\theta_{M,x}(u_x)),
	\end{gather}
	for all $(g,x)\in\calG{}_s{\times}_\moment M$, $\xi_g\in \ker T_gs$ and $u_x\in T_xM$, where $\Sigma$ is any local Legendrian bisection of $(\calG,\calH_\calG)\rightrightarrows\calG_0$ with $\Sigma(sg)=g$.
	Indeed, setting $\xi_g=(T_{(tg,g)}m)(\xi_{tg},0_g)$, with $\xi_{tg}\in \ker T_{tg}s$, one can easily compute:
	\begin{multline*}
		(\hat\moment_{g\cdot x}\circ\theta_{M,g\cdot x}\circ T_{(g,x)}\Phi)(\xi_g,0_x)=(\hat\moment_{g\cdot x}\circ\theta_{M,g\cdot x}\circ T_{(g,x)}\Phi)((T_{(tg,g)}m)(\xi_{tg},0_g),0_x)\\
		\overset{\eqref{eq:action_axioms}}{=}(\hat\moment_{g\cdot x}\circ\theta_{M,g\cdot x}\circ T_{(tg,g\cdot x)}\Phi)(\xi_{tg},0_{g\cdot x})\overset{\eqref{eq:proof:prop:contact_actions:forms:0}}{=}\theta_{\calG,tg}(\xi_{tg})=\theta_{\calG,g}(\xi_g),
	\end{multline*}
	and this proves Equation~\eqref{eq:proof:prop:contact_actions:forms:1}.
	Further, Condition~\ref{enumitem:prop:contact_action:distributions} allows to assume w.l.o.g.~that $u_x=(T_{(\moment x,x)}\Phi)(\xi_{\moment x},0_x)$, for some arbitrary $\xi_{\moment x}\in 
	\ker T_{\moment x}s$.
	Then, setting $\xi_g:=(T_{g,sg}m)((T_{sg}\Sigma)(T_{sg}t)\xi_{\moment x},\xi_{\moment x})\in \ker T_gs$, one can compute:
	\begin{multline*}
		(\hat\moment_{g\cdot x}\circ\theta_{M,g\cdot x}\circ T_{(g,x)}\Phi)((T_{sg}\Sigma)(T_x\moment)u_x,u_x)=(\hat\moment_{g\cdot x}\circ\theta_{M,g\cdot x}\circ T_{(g,x)}\Phi)((T_{sg}\Sigma)(T_{\moment x}t)\xi_{\moment x},(T_{(\moment x,x)}\Phi)(\xi_{\moment x},0_x))\\
		=(\hat\moment_{g\cdot x}\circ\theta_{M,g\cdot x}\circ T_{(g,x)}\Phi)(\xi_g,0_x)\overset{\eqref{eq:proof:prop:contact_actions:forms:1}}{=}\theta_{\calG,g}(\xi_g)=\theta_{\calG,g}((T_{g,sg}m)((T_{sg}\Sigma)(T_{sg}t)\xi_{\moment x},\xi_{\moment x}))\\
		\overset{\eqref{eq:mult_contact_form:CrainicSalazar}}{=}g\cdot(\theta_{\calG,\moment x}(\xi_{\moment x}))\overset{\eqref{eq:proof:prop:contact_actions:forms:0}}{=}g\cdot((\hat\moment_x\circ\theta_{M,x}\circ T_{(\moment x,x)}\Phi)(\xi_{\moment x},0_x))=g\cdot((\hat\moment_x\circ\theta_{M,x})u_x),
	\end{multline*}
	and this proves Equation~\eqref{eq:proof:prop:contact_actions:forms:2}.
	Finally, because of the direct sum decomposition
	\begin{equation*}
	T_{(g,x)}\calG{}_s{\times}_\moment M=(\ker T_g s\times\{0_x\})\oplus\{((T_{sg}\Sigma)(T_x\moment)u_x,u_x):u_x\in T_xM\},\quad\text{for all}\ (g,x)\in\calG{}_s{\times}_\bbJ M,
	\end{equation*}
	Equation~\eqref{eq:contact_action} follows directly from Equations~\eqref{eq:proof:prop:contact_actions:forms:1} and~\eqref{eq:proof:prop:contact_actions:forms:2}.
	This completes the proof of Condition~\ref{enumitem:prop:contact_action:forms}. 
		
	\ref{enumitem:prop:contact_action:forms} $\Rightarrow$ \ref{enumitem:prop:contact_action:distributions}.
	Fix $(g,x)\in\calG{}_s{\times}_\moment M$ and $(\xi,u)\in T_{(g,x)}(\calG{}_s{\times}_\moment M)$.
	Assume further that $u\in\calH_M$.
	Then Condition~\ref{enumitem:prop:contact_action:forms} guarantees that
	\begin{equation*}
	(T_{(g,x)}\Phi)(\xi,u)\in\calH_M\Leftrightarrow 0=(\Phi^\ast\theta_M)(\xi,u)\overset{\eqref{eq:contact_action_bis}}{=}(\operatorname{pr}_\calG^\ast\theta_\calG)(\xi,u)+g\cdot\cancel{(\operatorname{pr}^\ast\theta_M)(\xi,u)}=\theta_\calG\xi\Leftrightarrow\xi\in\calH_G,
	\end{equation*}
	and this proves Condition~\ref{enumitem:prop:contact_action:distributions}.
\end{proof}

\subsection{Contact actions by Lie groups}
\label{sec:Lie_group_actions}

This section shows that, under a certain transversality condition, any action by contactomorphisms of a Lie group $G$ on a contact manifold is fully encoded into a contact action (see Definition~\ref{def:contact_action}) of the contact groupoid $\bbP(T^\ast G)\rightrightarrows\bbP(\frakg^\ast)$ of Example~\ref{projective2} on the same contact manifold.

Let $G$ be a Lie group, with Lie algebra $\frakg$, and let $M$ be a manifold.
Consider a (left) action of $G$ on $M$
\begin{equation}
\label{eq:Lie_group_action}
\phi:G\times M\to M,\ (g,x)\mapsto\phi_g(x)=g\cdot x,
\end{equation}
with induced Lie algebra action $\zeta:\frakg\to\frakX(M),\ X\mapsto\zeta(X)$.
Assume, additionally, that $M$ is equipped with a contact structure, equivalently given by the contact distribution $\calH$ and the contact form $\theta\in\Omega^1(M;L)$.

\begin{proposition}
	\label{prop:Lie_group_actions}
	Let the Lie group action $\phi$, as in Equation~\eqref{eq:Lie_group_action}, satisfy the following two conditions
	\begin{enumerate}[label=(\arabic*)]
		\item\label{enumitem:prop:Lie_group_actions:contactomorphisms}
		the Lie group $G$ acts by contactomorphisms on $(M,\calH)$, i.e.~$(T\phi_g)\calH=\calH$, for all $g\in G$,
		\item\label{enumitem:prop:Lie_group_actions:trasversality}
		the $G$-orbits are transverse to the contact distribution $\calH$, i.e.~$TM=\calH+\operatorname{span}\{\zeta(X):X\in\frakg\}$.
	\end{enumerate}
	Then it gives rise to a contact action $\tilde\Phi:\bbP(T^\ast G){}_s{\times}_{\tilde\moment}M\to M$ of the contact groupoid $\bbP(T^\ast G)\rightrightarrows\bbP(\frakg^\ast)$ on the contact manifold $(M,\calH)$, with moment map $\tilde\moment:M\to\bbP(\frakg^\ast)$,
	which are defined as follows:
	\begin{equation}
	\label{eq:prop:Lie_group_actions:moment_map_and_action}
	\tilde\moment(x)=[\eta_x\circ\theta_x\circ\zeta],\qquad[\alpha_g]\cdot x=g\cdot x,
	\end{equation}
	for all $x\in M$, $g\in G$,  $\eta_x\in L_x^\ast$ and $\alpha_g\in T_g^\ast G$, such that $\tilde\moment(x)=s[\alpha_g]\equiv[\alpha_g\circ T_eR_g]$.
\end{proposition}

\begin{proof}
	As for any Lie group action, one can construct the cotangent lift of $\phi$ to a Lie group action of $G$ on the cotangent bundle $T^\ast M$
	\begin{equation*}
	G\times T^\ast M,\ (g,\nu_x)\mapsto (T^\ast_x\phi_g)\nu_x:=\nu_x\circ(T_x\phi_g)^{-1}.
	\end{equation*}
	Further, as it is well-known, this action is Hamiltonian w.r.t.~the canonical symplectic form $\omega_M:=\rmd\lambda_M$, where $\lambda_M\in\Omega^1(T^\ast M)$ is the Liouville $1$-form, and it has the following coadjoint equivariant moment map, which is fiberwise $\bbR$-linear,
	\begin{equation*}
	\moment:T^\ast M\to\frakg^\ast,\ \nu_x\mapsto\moment(\nu_x):=\nu_x\circ\zeta.
	\end{equation*}
	Now one can construct the following action of $T^\ast G\rightrightarrows\frakg^\ast$ on $T^\ast M$, with moment map $\moment:T^\ast M\to\frakg^\ast$,
	\begin{equation*}
	\Phi:T^\ast G{}_s{\times}_\moment T^\ast M\to T^\ast M,\ (\alpha_g,\nu_x)\mapsto (T^\ast_x\phi_g)\nu_x,
	\end{equation*}
	and, as it is easy to see, this action is homogeneous and symplectic.
	
	Using the line bundle isomorphism $TM/\calH\to L$ induced by the contact form $\theta\in\Omega^1(M;L)$, one also identifies $L^\ast$ with $\calH^0\subset T^\ast M$, the annihilator of $\calH$.
	Condition~\ref{enumitem:prop:Lie_group_actions:contactomorphisms} means exactly that $\Phi$ preserves $L^\ast$.
	Therefore, by restriction, $\Phi$ induces an action, still denoted by $\Phi$, of $T^\ast G\rightrightarrows\frakg^\ast$ on $L^\ast$, with moment map obtained by restriction of $\moment$ to $L^\ast$, and still denoted by $\bbJ$.
	Unwrapping the identifications, one gets that the moment map is explicitly given by
	\begin{equation}
	\label{eq:proof:prop:Lie_group_actions:moment_map}
	\moment:L^\ast\to\frakg^\ast,\ \eta_x\mapsto\moment(\eta_x)=\eta_x\circ\theta_x\circ\zeta,
	\end{equation}
	and the action $\Phi:T^\ast G{}_s{\times}_\moment L^\ast\to L^\ast$ is completely determined as follows
	\begin{equation}
	\label{eq:proof:prop:Lie_group_actions:action}
	\langle\Phi(\alpha_g,\eta_x),\theta_{g\cdot x}(u_{g\cdot x})\rangle=\langle\eta_x,\theta_x((T_x\phi_g)^{-1}u_{g\cdot x})\rangle,
	\end{equation}
	for all $\alpha_g\in T^\ast_gG$, $\eta_x\in L^\ast_x$, such that $\moment(\eta_x)=s(\alpha_g):=\alpha_g\circ T_e\rmR_g$, and $u_{g\cdot x}\in T_{g\cdot x}M$.
	Additionally, Condition~\ref{enumitem:prop:Lie_group_actions:trasversality} means exactly that, for any $x\in M$, the $\bbR$-linear map $\bbJ_x:L^\ast_x\to\frakg^\ast$ is injective.
	Therefore, by restriction, $\Phi$ induces a homogeneous symplectic action, still denoted by $\Phi$, of the homogeneous symplectic groupoid $T^\ast G\setminus{\bf 0}_G\rightrightarrows\frakg^\ast\setminus\{0\}$ on the homogeneous symplectic manifold $L^\ast\setminus{\bf 0}_M$, with $\bbR^\times$-equivariant moment map $\moment:L^\ast\setminus{\bf 0}_M\to\frakg^\ast\setminus\{0\}$.
	
	Finally, applying to the current setting the dehomogeneization procedure described in Proposition~\ref{prop:homogeneous_symplectic_action}, one obtains a contact action $\tilde\Phi$ of the contact groupoid $\bbP(T^\ast G)\rightrightarrows\bbP(\frakg^\ast)$ on the contact manifold $M$ with moment map $\tilde\moment:M\to\bbP(\frakg^\ast)$.
	It is now easy to see that, in view of Equations~\eqref{eq:proof:prop:Lie_group_actions:moment_map} and~\eqref{eq:proof:prop:Lie_group_actions:action}, the expressions of $\tilde\moment$ and $\tilde\Phi$ agree with the ones given in Equation~\eqref{eq:prop:Lie_group_actions:moment_map_and_action}
\end{proof}

\subsection{Global contact reduction}
\label{subsec:global_contact_reduction}

This section shows that contact dual pairs naturally emerge from reduction of contact manifolds with symmetries.
We begin by stating the main results and leave their proofs, split in a series of smaller lemmas, for the end of the current Section~\ref{subsec:global_contact_reduction}.

The next proposition extends, to the non-coorientable case, a result first obtained in~\cite[Section~4.2]{zambon2006contact}.
\begin{proposition}
	\label{prop:quotient_Jacobi_structure}
	Let the contact action $\Phi$ be free and proper.
	Then there exists a unique Jacobi structure on the quotient line bundle $L_M/\calG\to M/\calG$, denoted by $\{-,-\}_{M/\calG}$, such that the quotient map $\hat{q}:L_M\to L_M/\calG$, regular line bundle morphism covering the quotient map $q:M\to M/\calG$, is a Jacobi morphism from $(M,L_M,\{-,-\}_M)$ to $(M/\calG,L_M/\calG,\{-,-\}_{M/\calG})$.
\end{proposition}

\begin{theorem}
	\label{theor:CDP_contact groupoid_action}
	Let $\Phi$ be a free, proper and contact action of a source-connected contact groupoid $(\calG,\calH_\calG)\rightrightarrows\calG_0$ on a contact manifold $(M,\calH_M)$, with moment map $\moment:M\to\calG_0$.
	Then the quotient bundle map $\hat{q}:L_M\to L_M/\calG$, covering $q:M\to M/\calG,$ and the fat moment map $\hat{\moment}:L_M\to L_0$, covering $\moment:M\to\calG_0$, form the full contact dual pair
	\begin{equation}
	\label{eq:theor:CDP_contact_groupoid_action}
	\begin{tikzcd}
	(M/\calG,L_M/\calG,\{-,-\}_{M/\calG})&\ar[l, swap, "\hat{q}"](M,L_M,\{-,-\}_M)\ar[r, "\hat{\moment}"]&(\calG^\moment_0,L_{\calG^\moment_0},\{-,-\}_{\calG^\moment_0}),
	\end{tikzcd}
	\end{equation}
	where $\calG^\moment_0\subset\calG_0$ denotes the open $\calG$-invariant image of $M$ by $\moment$. 
\end{theorem}

\begin{remark}[\textbf{Contact actions by Lie group}]
	\label{rem:CDP_Lie group_action}
	Consider a free and proper action of a connected Lie group $G$ on a contact manifold $(M,\calH)$ s.~t.~$G$ acts by contactomorphisms and the $G$-orbits are transverse to $\calH$.
	As shown in Proposition~\ref{prop:Lie_group_actions}, this gives rise to a contact action $\tilde\Phi:\bbP(T^\ast G){}_s{\times}_\moment M\to M$ of the source-connected contact groupoid $\bbP(T^\ast G)\rightrightarrows\bbP(\frakg^\ast)$ on the contact manifold $(M,\calH)$, with a certain moment map $\moment:M\to\bbP(\frakg^\ast)$.
	Fix an associated fat moment map $\hat{\moment}:L_M\to \bigO_{\bbP(\frakg^\ast)}(1)$, covering $\tilde\moment:M\to\bbP(\frakg^\ast)$, as in Proposition~\ref{prop:contact_action}, where $\bigO_{\bbP(\frakg^\ast)}(1)$ denotes
	the dual of the tautological line bundle over $\bbP(\frakg^*)$.
	% ($\tau_{[\mu]}=\langle\mu\rangle,\ \mu\in\frakg^*\setminus\{0\}$). 
	Together with the quotient bundle map $\hat{q}:L_M\to L_M/G$, covering $q:M\to M/G,$ they
	form the full contact dual pair
	\begin{equation}
	\label{eq:theor:CDP_Lie_group_action}
	\begin{tikzcd}
	(M/G,L/G,\{-,-\}_{M/G})&\ar[l, swap, "\hat{q}"](M,L_M,\{-,-\}_M)\ar[r, "\hat{\moment}"]&(\bbP(\frakg^\ast)^{\tilde\moment},\bigO_{\bbP(\frakg^\ast)}(1),\{-,-\}_{\bbP(\frakg^\ast)}),
	\end{tikzcd}
	\end{equation}
	where $\bbP(\frakg^\ast)^{\tilde\moment}\subset\bbP(\frakg^\ast)$ denotes the open $G$-invariant image of $M$ by $\tilde\moment$. 
\end{remark}

The following is a straightforward consequence of Theorems~\ref{theor:LeafCorrespondenceI} and~\ref{theor:CDP_contact groupoid_action} and Example~\ref{ex:contact_groupoids}.
So we obtain an alternative proof of a result which was first proven, by different techniques, in the coorientable case (cf.~\cite[Theorem 4.4]{zambon2006contact}).

\begin{corollary}[{\cite[Theorem 4.4]{zambon2006contact}}]
	\label{cor:CDP_contact groupoid_action}
	Let $\Phi$ be a free, proper and contact action of a source-connected contact groupoid $\calG\rightrightarrows\calG_0$ on a contact manifold $(M,\calH_M)$, with moment map $\moment:M\to\calG_0$, as in Theorem \ref{theor:CDP_contact groupoid_action}.
	Then, the characteristic leaves of $(M/\calG,L_M/\calG,\{-,-\}_{M/\calG})$ are given by
	\begin{equation*}
	\moment^{-1}(\bigO_x)/\calG,\quad\text{for}\ x\in \calG^\bbJ_0:=\bbJ(M),
	\end{equation*} 
	where $\bigO_x:=t(s^{-1}(x))\subset \calG_0$ is the coadjoint orbit of $\calG$ through $x$.
	Moreover, $ \moment^{-1}(\bigO_x)/\calG$ is either a contact leaf if the dimension of $\bigO_x$ is odd, or is a l.c.s.~leaf if the dimension of $\bigO_x$ is even.
\end{corollary}

Coming back to the proofs of Proposition~\ref{prop:quotient_Jacobi_structure} and Theorem~\ref{theor:CDP_contact groupoid_action}, we split them into a series of Lemmas having their own interest.
The next lemma extends to the non-coorientable case a result first obtained in~\cite[Remark~3.2.i) and Lemma 3.3]{zambon2006contact}.

\begin{lemma}
	\label{lem:Jacobi_moment_map}
	For any contact action $\Phi$, as in Definition~\ref{def:contact_action}, the fat moment map $\hat{\moment}:L_M\to L_0$, covering the moment map $\moment:M\to\mathcal{G}_0$, is a $\calG$-equivariant Jacobi morphism.
	In particular, for all $\lambda\in\Gamma(L_0)$ and $(g,x)\in \calG{}_s{\times}_\moment M$, we have:
	\begin{equation}
	\label{eq:lem:Jacobi_moment_map}
	(\Delta_{\hat{\moment}^\ast\lambda})_{g\cdot x}=(\Delta_{t^\ast\lambda})_g\cdot 0_x\quad\text{and}\quad (\calX_{\hat\moment^\ast\lambda})_{g\cdot x}=(\calX_{t^\ast\lambda})_g\cdot 0_x,
	\end{equation}
	where the dot $\cdot$
	denotes, on the LHS, the induced action of the derivation Lie groupoid $DL\rightrightarrows DL_0$ (cf.~\cite[Section~4.2]{ETV2016}) on $DL_M$, with moment map $D\hat\moment:DL_M\to DL_0$, and, on the RHS, the induced action of the tangent Lie groupoid $T\calG\rightrightarrows T\calG_0$ on $TM$, with moment map $T\moment:TM\to T\calG_0$.
\end{lemma}

\begin{proof}
	We claim that it is sufficient to prove (the RHS of) Equation~\eqref{eq:lem:Jacobi_moment_map}.
	Indeed, the latter, which is equivalent to $\calX_{\hat{\moment}^\ast \lambda}=\Phi_\ast(\calX_{t^\ast \lambda},0),$
	for all $\lambda\in\Gamma(L_0)$, allows us to compute
	\begin{multline*}
	(\{\hat{\moment}^\ast\lambda_1,\hat{\moment}^\ast\lambda_2\}_M)_{g\cdot x}=(\theta_M)_{g\cdot x}\left(\Phi_\ast([\calX_{t^\ast\lambda_1},\calX_{t^\ast\lambda_2}],0)\right)\overset{\eqref{eq:contact_action}}{=}(\theta_\calG)_g[\calX_{t^\ast\lambda_1},\calX_{t^\ast \lambda_2}]+g\cdot((\theta_M)_x0)\\
	=\left(\{t^\ast \lambda_1,t^\ast\lambda_2\}_\calG\right)_g	=(t^\ast\{\lambda_1,\lambda_2\}_0)_g= (\hat{\moment}^\ast\{\lambda_1,\lambda_2\}_0)_{g\cdot x},
	\end{multline*}
	for all $\lambda_1,\lambda_2\in\Gamma(L_0)$, and $(g,x)\in \calG{}_s{\times}_\moment M$.
	This actually means that, as needed, $\hat{\moment}$ is a Jacobi morphism from $(M,L_M,\{-,-\}_M)$ to $(\calG_0,L_0,\{-,-\}_0)$. 
	
	Fix now an arbitrary section $\lambda\in\Gamma(L_0)$.
	The contact vector field $\calX_{t^\ast\lambda}$ generates a one-parameter group $\{\psi_\epsilon\}_{\epsilon\in\bbR}$ of local contactomorphisms of $(\calG,\calH_G)$.
	Similarly, the Hamiltonian DO $\Delta_{t^\ast\lambda}:={\rm J}_\calG^\sharp(j^1(t^\ast\lambda))$ generates a one-pa\-ram\-e\-ter group $\{\Psi_\epsilon\}_{\epsilon\in\bbR}$ of local Jacobi automorphisms of $(\calG,L,\{-,-\}_\calG\simeq{\rm J}_\calG)$, covering $\{\psi_\epsilon\}_{\epsilon\in\bbR}$.
	In view of Proposition~\ref{prop:LW_CDP:tangent_fibers} and Theorem~\ref{theor:Jacobi}, the $\psi_\epsilon$'s preserve the $s$-fibers of $\calG$, and so we can define a one-parameter group $\{w_\epsilon\}_{\epsilon\in\bbR}$ of local diffeomorphisms of $M$ by setting $w_\epsilon:=\Phi\circ(\psi_\epsilon\circ\moment,\id_M)$, for all $\epsilon\in\bbR$.
	Similarly, we can consider the one-parameter group $\{W_\epsilon\}_{\epsilon\in\bbR}$ of local line bundle automorphisms of $L_M\to M$, covering $\{w_\epsilon\}_{\epsilon\in\bbR}$, which is well-defined by setting $W_\epsilon:=\hat{\Phi}\circ(\Psi_\epsilon\circ\hat{\moment},\id_{L_M})$, for all $\epsilon\in\bbR$.
	In order to complete the proof of Equation~\eqref{eq:lem:Jacobi_moment_map}, it is enough to show that $\{W_\epsilon\}_{\epsilon\in\bbR}$ consists of Jacobi automorphisms of $(M,L_M,\{-,-\}_M\simeq{\rm J}_M)$ and it is generated by $\Delta_{\hat{\moment}^\ast\lambda}:={\rm J}_M^\sharp(j^1(\hat{\moment}^\ast\lambda))$.
	
	Fix $x\in M$, $\epsilon\in\bbR$, and set $\bar{x}:=\moment(x)$.
	Then we can compute:
	\begin{multline}
	(W_\epsilon^\ast\theta_M)_x=(W_\epsilon)_{x}^{-1}((\theta_M)_{w_\epsilon(x)}\circ T_{(\psi_\epsilon(\bar{x}),x)}\Phi\circ(T_{\bar{x}}\psi_\epsilon\circ T_x\moment,\id_{T_xM})\\
	\overset{\eqref{eq:contact_action}}{=}(W_\epsilon)_{x}^{-1}\left((\theta_\calG)_{\psi_\epsilon(\bar{x})}\circ T_{\bar{x}}\psi_\epsilon\circ T_x\moment+(\psi_\epsilon(\bar{x}))\cdot(\theta_M)_x\right)
	=(\theta_M)_x.
	\label{eq:proof:lem:Jacobi_moment_map:Jacobi_automorphism}
	\end{multline}
	In the last two steps above we have used the following two facts: 1) $\{\Psi_\epsilon\}_{\epsilon\in\bbR}$ consists, by construction, of Jacobi automorphisms of $(\calG,L_\calG,\{-,-\}_\calG)$, i.e.~$\Psi_\epsilon^\ast\theta_\calG=\theta_\calG$, and 2) $\calG_0$ is a Legendrian submanifold in $\calG$, so that $\hat{\moment}^\ast\theta_\calG=0$.
	Equation~\eqref{eq:proof:lem:Jacobi_moment_map:Jacobi_automorphism} means exactly	that $\{W_\epsilon\}_{\epsilon\in\bbR}$ consists of Jacobi automorphisms of $(M,L_M,\{-,-\}_M)$.
	Further we can compute:	
	\begin{multline*}
	\theta_M\left(\frac{\rmd}{\rmd \epsilon}w_\epsilon(x)\right)=\theta_M\left(T\Phi((\calX_{t^\ast \lambda})_{\psi_\epsilon(\bar{x})},0_x)\right) \overset{\eqref{eq:contact_action}}{=}{\underbrace{\theta_\calG((\calX_{t^\ast \lambda})_{\psi_\epsilon(\bar{x})})}_{\lambda_{t(\psi_\epsilon(\bar{x}))}}}\;+\;\psi_\epsilon(\bar{x})\cdot(\theta_M(0_x))\overset{\eqref{eq:action_axioms}}{=}\lambda_{\moment(w_\epsilon(x))}=(\hat{\moment}^\ast\lambda)_{w_\epsilon(x)}.
	\end{multline*}
	The latter means that the one-parameter group $\{w_\epsilon\}_{\epsilon\in\bbR}$ of contactomorphisms of $(M,\calH_M)$ is generated by the contact vector field $\calX_{\hat{\moment}^\ast\lambda}$, or equivalently the one-parameter group of Jacobi automorphisms $\{W_\epsilon\}_{\epsilon\in\bbR}$ of $(M,L_M,\{-,-\}_M)$ is generated by the Hamiltonian differential operator $\Delta_{\hat{\moment}^\ast\lambda}$.
	This completes the proof.
\end{proof}

%\begin{remark}
%	Fix an arbitrary section $u\in\Gamma(L_0)$.
%	The contact vector field $\calX_{s^\ast u}$ (resp.~$\calX_{\hat{\moment}^\ast u}$) generates a one-parameter group $\{\psi_t\}_{t\in\bbR}$ (resp.~$\{\varphi_t\}_{t\in\bbR}$) of local contactomorphisms of $(\calG,\calH_G)$ (resp.~$(M,\calH_M)$).
%	Similarly, the Hamiltonian DO $\Delta_{s^\ast u}$ (resp.~$\Delta_{\hat{\moment}^\ast u}$)  generates a one-pa\-ram\-e\-ter group $\{\hat{\psi}_t\}_{t\in\bbR}$ (resp.~$\{\hat{\varphi}_t\}_{t\in\bbR}$) of local Jacobi automorphisms of $(\calG,L_\calG,\{-,-\}_\calG)$ (resp.~$(M,L_M,\{-,-\}_M$), covering $\{\psi_t\}_{t\in\bbR}$ (resp.~$\{\varphi_t\}_{t\in\bbR}$).
%	Then Equation~\eqref{eq:lem:Jacobi_moment_map} from Lemma~\ref{lem:Jacobi_moment_map} can be equivalently rewritten as follows:
%	\begin{equation}
%	\hat{\psi}_t=\hat{\Phi}\circ(\id_{L_M},\hat{\varphi}_t) \quad\text{and}\quad \psi_t=\Phi\circ(\id_M,\varphi_t).
%	\end{equation}
%\end{remark}

\begin{remark}
	Let us recall that any action $\Phi$, as in Equation~\eqref{eq:action}, of the Lie groupoid $\calG\rightrightarrows\calG_0$ on $M$, with moment map $\moment:M\to\calG_0$, induces an \emph{infinitesimal action}~\cite{kosmann2002differential} of the associated Lie algebroid $A\Rightarrow\calG_0$ on $M$, with moment map $\moment:M\to\calG_0$, denoted by
	\begin{equation}
	\zeta:\Gamma(A)\to\frakX(M),\ a\mapsto\zeta(a),
	\end{equation}
	such that $\zeta(a)_{g\cdot x}:=(a^r)_g\cdot 0_x$, for all $a\in\Gamma(A)$ and $(g,x)\in\calG{}_s{\times}_\moment M$, where the dot $\cdot$ denotes the induced action of $T\calG\rightrightarrows T\calG_0$ on $TM$, with moment map $T\moment:TM\to T\calG_0$.
	Then the action $\Phi$ is called \emph{locally free} at $x\in M$ if the $\bbR$-linear map $\zeta_x:A_{\moment(x)}\to T_xM$, $a_{\moment(x)}\mapsto\zeta(a)_x$, is injective.
	Further recall that, as for any contact groupoid (cf.~\cite[Theorem 1]{crainic2015jacobi}), there exists a canonical isomorphism from the Lie algebroid $A\Rightarrow\calG_0$ of $\calG\rightrightarrows\calG_0$ to the Lie algebroid $J^1L_0\Rightarrow\calG_0$ of the Jacobi manifold $(\calG_0,L_0,\{-,-\}_{\calG_0})$ identifying $\calX_{t^\ast\lambda}$ with $j^1\lambda$, for all $\lambda\in\Gamma(L_0)$.
	Understanding this latter identification, the RHS of Equation~\eqref{eq:lem:Jacobi_moment_map} means that
	\begin{equation}
	\label{eq:infinitesimal_action_at_a_point}
	\zeta(j^1\lambda)=\calX_{\hat{\moment}^\ast\lambda},\quad\text{for all}\ \lambda\in\Gamma(L_0).
	\end{equation}
\end{remark}

The next lemma extends to the non-coorientable case a result first obtained in~\cite[Lemma 3.5]{zambon2006contact}.

\begin{lemma}
	\label{lem:locally_free}
	The contact action $\Phi$ is locally free at $x\in M$ if and only if the following two conditions hold:
	\begin{enumerate}[label=(\arabic*)]
		\item\label{enumitem:lem:locally_free:submersion}
		$\moment:M\to\calG_0$ is a submersion at $x$, and
		\item\label{enumitem:lem:locally_free:transverse}
		$\ker T_x\moment$ is transverse to $\calH_{M,x}$.
	\end{enumerate}
\end{lemma}

\begin{proof}
	Fix $x\in M$, and set $\bar{x}:=\moment(x)\in\calG_0$.
	Equation~\eqref{eq:infinitesimal_action_at_a_point} allows us to factorize $\zeta_x:A_{\bar{x}}\simeq J^1_{\bar{x}}L_0\to T_xM$, $a_{\bar{x}}\mapsto(\zeta_a)_x$, as follows
	\begin{equation}
	\begin{tikzcd}
	A_{\bar{x}}\simeq J^1_{\bar{x}}L_0\ar[rr, "\zeta_x"]\ar[rd, swap, "(D_x\hat{\moment})^\ast"]&&T_xM\\
	&J^1_xL_M\ar[ur, two heads, swap, "\sigma\circ{\calJ}^\sharp"]&
	\end{tikzcd}
	\end{equation}
	where $(D_x\hat{\moment})^\ast$ denotes the $\bbR$-linear map $J^1_{\bar{x}}L_0\simeq((DL_0)^\ast\otimes L_0)_{\bar{x}}\to J^1_xL_M\simeq((DL_M)^\ast\otimes L_M)_x$, $j^1_{\bar{x}}\lambda\mapsto j^1_x(\hat{\moment}^\ast\lambda)$.
	Moreover, we denote by $\sigma:DL_M\to TM$ the symbol map and by ${\calJ}\in\calD^2L_M\equiv\Gamma(\wedge^2(J^1L_M)^\ast\otimes L_M)$ the non-degenerate Jacobi bi-DO on $L_M\to M$ corresponding to the contact structure on $M$ (cf.~Proposition~\ref{prop:non-degenerateJacobi=contact}).	
	Hence, since $\ker\sigma=\langle\mathbbm{1}\rangle$ and, by construction, $\calJ^\sharp(\sigma^\ast\theta)=\mathbbm{1}$,	we get that the contact groupoid action $\Phi$ is locally free at $x$ if and only if the following conditions hold:
	\begin{enumerate}[label=(\arabic*${}^\prime$)]
		\item\label{enumitem:proof:lem:locally_free:submersion}
		$(D_x\hat{\moment})^\ast:J^1_{\bar{x}}L_0\to J^1_xL_M$ is injective, and
		\item\label{enumitem:proof:lem:locally_free:transverse}
		the intersection of $\langle(\sigma^\ast\theta_M)_x\rangle$ and $\im[(D_x\hat{\moment})^\ast]$ is trivial, i.e.~$\langle(\sigma^\ast\theta_M)_x\rangle\cap\im[(D_x\hat{\moment})^\ast]=\{0_x\}\subset J^1_xL_M$.
	\end{enumerate}
	In order to conclude the proof, it will be sufficient to prove that conditions~\ref{enumitem:proof:lem:locally_free:submersion} and~\ref{enumitem:proof:lem:locally_free:transverse} are equivalent to conditions~\ref{enumitem:lem:locally_free:submersion} and~\ref{enumitem:lem:locally_free:transverse} respectively.
	The equivalence of conditions~\ref{enumitem:proof:lem:locally_free:submersion} and~\ref{enumitem:lem:locally_free:submersion} can be easily checked, e.g.~working with local frames and local coordinates, so we will concentrate on the remaining equivalence.
	
	The natural $L_M$-valued pairing between $J^1L_M$ and $DL_M$ allows us to write% Equation~\eqref{eq:proof:lem:locally_free:transverse} as follows:
	\begin{equation*}
	\label{eq:proof:lem:locally_free:transverse:bis}
	%D_xL_M\equiv\{0_x\}^\circ=
	\{\langle(\sigma^\ast\theta_M)_x\rangle\cap\im((D_x\hat{\moment})^\ast)\}^\circ=\langle(\sigma^\ast\theta_M)_x\rangle^\circ+\{\im[(D_x\hat{\moment})^\ast]\}^\circ=\sigma^{-1}(\ker\theta_M)_x+\ker D_x\hat{\moment}=\sigma^{-1}((\calH_M)_x+\ker T_x\moment).
	\end{equation*}
	Since $\sigma:DL_M\to TM$ is surjective, the latter shows that $\langle(\sigma^\ast\theta_M)_x\rangle$ and $\im[(D_x\hat{\moment})^\ast]$ intersect trivially in $J^1_xL_M$ if and only if $(\calH_M)_x$ and $\ker T_x\moment$ are transverse in $T_xM$.
	This proves the remaining equivalence of conditions~\ref{enumitem:proof:lem:locally_free:transverse} and~\ref{enumitem:lem:locally_free:transverse}, and so concludes the proof.
\end{proof}

The next lemma extends to the non-coorientable case a result first obtained in~\cite[Lemma~4.5]{zambon2006contact}.
%\end{remark}

\begin{lemma}
	\label{lem:Legendrian_bisections}
	Let $\Sigma$ be a local Legendrian bisection of $\calG\rightrightarrows\calG_0$, with $\underline{\smash{r_\Sigma}}:=t\circ\Sigma$ the induced local diffeomorphism of $\calG_0$.
	Consider the local line bundle automorphism $\hat r_\Sigma$ of $L_M\to M$, covering the local diffeomorphism $r_\Sigma$ of $M$, which are defined by $r_\Sigma:=\Phi\circ(\Sigma\circ\moment,\id_M)$ and $\hat r_\Sigma:=\hat{\Phi}\circ(\Sigma\circ\moment\circ\pi,\id_{L_M})$.
	Then $\hat r_\Sigma$ is a local Jacobi automorphism of $(M,L_M,\{-,-\}_M)$, or, equivalently, $r_\Sigma$ is a contactomorphism of $(M,\calH_M)$, i.e.~locally
	\begin{equation*}
	\hat r_\Sigma^\ast\theta_M=\theta_M.
	\end{equation*}
\end{lemma}

\begin{proof}
	Fix an arbitrary $x\in M$ belonging to the domain of $r_\Sigma$, and set $\bar{x}:=\moment(x)$.
	Then we can compute:
	\begin{multline*}
	(\hat{r}_\Sigma^\ast\theta_M)_x=(\hat{r}_\Sigma)_x^{-1}(\theta_M\circ T_xr_{\Sigma})=(\hat{r}_\Sigma)_x^{-1}(\theta_M\circ T\Phi\circ(T_{\bar{x}}\Sigma\circ T_x\moment,\id_{T_xM}))\\
	=(\hat{r}_\Sigma)_x^{-1}(\cancel{(\theta_\calG)_{\Sigma\bar{x}}\circ T_{\bar{x}}\Sigma\circ T_x\moment}\,+\,(\Sigma(\bar{x}))\cdot(\theta_M)_x)=(\theta_M)_x,
	\end{multline*}
	where, just in the last step, we used the fact that $\Sigma$ is a Legendrian embedding in $\calG$.
\end{proof}

\begin{lemma}
	Through any point of the contact groupoid $\calG\rightrightarrows\calG_0$ there exists a local Legendrian bisection.
\end{lemma}

	\begin{proof}
	Fix arbitrarily $g\in\calG$.
	We will prove that there exists a local Legendrian bisection of $\calG\rightrightarrows\calG_0$ passing through $g$.
	Let us choose a Lagrangian subspace $V$ of the symplectic linear space $((\calH_\calG)_g,(\rmc_{\calH_\calG})_g)$ which is transverse to both $(\calH_\calG\cap\ker Ts)_g$ and  $(\calH_\calG\cap\ker Tt)_g$.
	By Darboux theorem for contact manifolds, there exist local coordinates $(x^i,u,y_i)$ on $\calG$, centered at $g$, such that $\calH_\calG=\ker\left(\rmd u-y_i\rmd x^i\right)$ locally around $g$, and $V=\langle\left.\frac{\partial}{\partial y_1}\right|_{g},\ldots,\left.\frac{\partial}{\partial y_n}\right|_{g}\rangle$.
	Therefore $u=x^1=\ldots=x^n=0$ defines a Legendrian submanifold of $\calG$ which is transverse to both the $s$-fibers and the $t$-fibers locally around $g$.
\end{proof}

\begin{proof}[Proof of Proposition~\ref{prop:quotient_Jacobi_structure}]
	First let us recall that, as $C^\infty(M/\calG)\simeq q^\ast(C^\infty(M/\calG))$ coincides with the space $C^\infty(M)^\calG\subset C^\infty(M)$ of $\calG$-invariant functions on $M$, similarly $\Gamma(L_M/\calG)\simeq \hat{q}^\ast(\Gamma(L_M/\calG))$ coincides with the space $\Gamma(L_M)^\calG\subset\Gamma(L_M)$ of $\calG$-invariant sections of $L_M\to M$.
	
	In the following, we will show that $\Gamma(L_M)^\calG$ is a Lie subalgebra of $(\Gamma(L_M),\{-,-\}_M)$.
	As a consequence, for any $\lambda\in\Gamma(L_M)^\calG$, the associated contact vector field $\calX_\lambda$ will be $q$-projectable.
	Hence, the Lie bracket $\{-,-\}_{M/\calG}$ induced by $\{-,-\}_M$ on $\Gamma(L_M/\calG)$ will be automatically a bi-DO, i.e.~it is actually a Jacobi structure on $L_M/\calG\to M/\calG$.
	
	It is easy to see that, for any $\lambda\in\Gamma(L_M)$, the following conditions are equivalent:
	\begin{enumerate}[label=(\arabic*)]
		\item $\lambda$ is $\calG$-invariant, i.e.~$\lambda\in\hat{q}^\ast[\Gamma(L_M/\calG)]\simeq\Gamma(L_M)^\calG\subset\Gamma(L_M)$,
		\item $\hat{r}_\Sigma^\ast\lambda=\lambda$, for all local Legendrian bisections $\Sigma$ of $\calG\rightrightarrows\calG_0$.
	\end{enumerate}
	Further, Lemma~\ref{lem:Legendrian_bisections} guarantees that, for all sections $\lambda_1,\lambda_2\in\Gamma(L_M)$, and all local Legendrian bisections $\Sigma$, if $\hat{r}_\Sigma^\ast\lambda_1=\lambda_1$ and $\hat{r}_\Sigma^\ast\lambda_2=\lambda_2$, then we also have
	%\begin{equation*}
	$\hat{r}_\Sigma^\ast\{\lambda_1,\lambda_2\}_M=\{\hat{r}_\Sigma^\ast \lambda_1,\hat{r}_\Sigma^\ast\lambda_2\}_M=\{\lambda_1,\lambda_2\}_M$.
	%\end{equation*}
	So $\Gamma(L_M)^\calG$ is closed w.r.t.~$\{-,-\}_M$, and this concludes the proof.
\end{proof}

\begin{lemma}
	\label{lem:centralizers}
	Let the contact action $\Phi$ be free and proper.
	Assume further that $\calG$ is source-connected.
	Then the subspaces $\hat{q}^\ast(\Gamma(L_M/\calG)),\hat{\moment}^\ast(\Gamma(L_0))\subset\Gamma(L_M)$ centralize each other w.r.t.~$\{-,-\}_M$, i.e.
	\begin{equation}
	\label{eq:lem:centralizers}
	\hat{q}^\ast(\Gamma(L_M/\calG))=[\hat{\moment}^\ast(\Gamma(L_0))]^c\quad\text{and}\quad \hat{\moment}^\ast(\Gamma(L_0))=[\hat{q}^\ast(\Gamma(L_M/\calG))]^c.
	\end{equation}
\end{lemma}

\begin{proof}
	In view of the dimensional relation $1+\dim(M/\calG)+\dim\calG_0=\dim M$, it is enough to prove only one component of Equation~\eqref{eq:lem:centralizers} (cf.~Remark~\ref{rem:Howe->Weinstein}).
	Specifically, we will restrict to prove its RHS, namely that $\hat{q}^\ast(\Gamma(L_M/\calG))$ is the centralizer of $\hat{\moment}^\ast(\Gamma(L_0))$ in $(\Gamma(L_M),\{-,-\}_M)$.	
	As recalled above, the sections of $L/\calG\to M/\calG$ are identified by pull-back via $\hat{q}$ to the $\calG$-invariant sections of $L_M\to M$.
	So we are going to prove that a section $\lambda\in\Gamma(L_M)$ is $\calG$-invariant if and only if $\lambda$ belongs to the centralizer of $\hat{\moment}^\ast(\Gamma(L_0))$.

	In view of Proposition~\ref{prop:LW_CDP:tangent_fibers} and Theorem~\ref{theor:Jacobi}, we have that $\ker Ts=\text{span}\{\calX_{t^\ast\lambda'}: \lambda'\in\Gamma(L_0)\}$, and additionally we have assumed that $\calG$ is source-connected.
	Hence Lemma~\ref{lem:Jacobi_moment_map}, and specifically Equation~\eqref{eq:lem:Jacobi_moment_map}, implies that, for any $\lambda\in\Gamma(L_M)$, the following two conditions are equivalent:
	\begin{enumerate}[label=\alph*)]
		\item\label{enumitem:proof:lem:centralizers:1}
		$\lambda$ is $\calG$-invariant, i.e.~$\lambda_{g\cdot x}=g\cdot\lambda_x$, for all $(g,x)\in \calG{}_s{\times}_\moment M$,
		\item\label{enumitem:proof:lem:centralizers:2}
		$\lambda$ is invariant under the one-parameter group $\psi_t$ of local line bundle automorphisms of $L_M\to M$ generated by $\Delta_{\hat{\moment}^\ast\mu}$, for all $\mu\in\Gamma(L_0)$.
		\setcounter{proof:lem:centralizers}{\value{enumi}}
	\end{enumerate}
	As a consequence of Proposition~\ref{prop:DO_infinitesimal_LB_automorphisms}, Condition~\ref{enumitem:proof:lem:centralizers:2} can be equivalently rewritten as
	\begin{enumerate}[label=\alph*)]
		\setcounter{enumi}{\value{proof:lem:centralizers}}
		\item $\lambda$ commutes with $\hat{\moment}^\ast\mu$ in $(\Gamma(L_M),\{-,-\}_M)$, for all $\mu\in\Gamma(L_0)$.
	\end{enumerate}
	The latter means exactly that $\lambda$ belongs to the centralizer of $\hat{\moment}^\ast(\Gamma(L_0))$ in $(\Gamma(L_M),\{-,-\}_M)$, and so this concludes the proof.
\end{proof}

\begin{proof}[Proof of Theorem~\ref{theor:CDP_contact groupoid_action}]
	In view of Proposition~\ref{prop:relation_Howe_Weinstein}~\ref{item:Howe->Weinstein} and the dimensional relation $1+\dim(M/\calG)+\dim\calG_0=\dim M$, it is enough to prove that diagram~\eqref{eq:theor:CDP_contact_groupoid_action} forms a full Howe contact dual pair.
	
	The quotient bundle map $\hat{q}:L_M\to L_M/\calG$ and the moment bundle map $\hat{\moment}:L_M\to L_0$ are Jacobi morphisms in view of Proposition~\ref{prop:quotient_Jacobi_structure} and Lemma~\ref{lem:Jacobi_moment_map} respectively.
	Since the groupoid action is free and proper,the quotient map $q:M\to M/\calG$ is a surjective submersion and the moment map $\moment:M\to\calG_0$ is a submersion (see Condition~\ref{enumitem:lem:locally_free:submersion} in Lemma~\ref{lem:locally_free}).
	
	We prove now that $\ker T\moment$ and $\ker Tq$ are both transverse to $\calH_M$.
	Since the contact groupoid action is free (and a fortiori locally free), the fibers of the moment map $\moment:M\to\calG_0$ are transverse to $\calH_M$ (see condition~\ref{enumitem:lem:locally_free:transverse} in Lemma~\ref{lem:locally_free}).
	%\st{As a consequence of Lemma~\ref{prop:quotient_Jacobi_structure}, and in particular Equation~\eqref{eq:lem:quotient_Jacobi_structure}, we get that, for any section $u\in\Gamma(L_M/\calG_0)$, the contact vector field $\calX_{\hat{q}^\ast u}$ is tangent to the $\calG$-orbits,}
	As a consequence of Lemma~\ref{lem:Jacobi_moment_map}, and in particular (the RHS of) Equation~\eqref{eq:lem:Jacobi_moment_map}, we get that, for any section $\lambda\in\Gamma(L_0)$, the contact vector field $\calX_{\hat{\moment}^\ast\lambda}$ is tangent to the $\calG$-orbits,
	and so the fibers of the quotient map $q:M\to M/\calG$ are transverse to $\calH_M$.
	
	Finally, Lemma~\ref{lem:centralizers} assures that $\hat{q}^\ast(\Gamma(L_M/\calG))$ and $\hat{\moment}^\ast(\Gamma(L_0))$ centralize each other in $\Gamma(L_M)$ w.r.t.~$\{-,-\}_M$.
	So diagram~\eqref{eq:theor:CDP_contact_groupoid_action} is a full Howe contact dual pair, and this concludes the proof.
\end{proof}

\appendix

\section{Differential Operators and Atiyah Forms of a Line Bundle}
\label{sec:differential_operators}

This Appendix sets the notation and collects some basic facts about jets and differential operators (DOs) of a line bundle $L\to M$.
In particular, it provides a quick review of what represents the conceptual background of the line bundle approach to Jacobi (and Dirac--Jacobi) geometry.
Therefore, it describes the gauge algebroid (and the omni-Lie algebroid) of $L$, with the associated graded Lie algebra of multi-DOs on $L$ and the der-complex of $L$-valued Atiyah forms.
Our main references for this material are~\cite{LOTV,tortorella2017phdthesis,vitagliano2018djbundles}.

\subsection{The Gauge Algebroid}
\label{sec:gauge_algebroid}

Given a line bundle $L\to M$, we recall that a \emph{first order linear differential operator (DO)} from $L$ to $L$ can be seen as a $\bbR$-linear map $\square:\Gamma(L)\to\Gamma(L)$ such that, for some (necessarily unique) $X\in\frakX(M)$, the following Leibniz rule holds
\begin{equation*}
\square(f\lambda)=X(f)\lambda+f\square\lambda,
\end{equation*}
for all $f\in C^\infty(M)$ and $\lambda\in\Gamma(L)$.
The vector field $X$ is called the \emph{symbol} of $\square$ and is denoted by $\sigma_\square$, or also $\sigma(\square)$.
In this paper, all the DOs are assumed to be first order linear.
As stated in the next proposition, the DOs from $L$ to $L$ can also be viewed as infinitesimal line bundle automorphisms of $L\to M$.

\begin{proposition}[{\cite[Lemma 2.2]{ETV2016}}]
	\label{prop:DO_infinitesimal_LB_automorphisms}
	For any line bundle $L\to M$, the following relation
	\begin{equation}
	\square\lambda=\left.\frac{\rmd}{\rmd\epsilon}\right|_{\epsilon=0}\Phi_\epsilon^\ast\lambda,\quad\text{for all}\ \lambda\in\Gamma(L),
	\end{equation}
	establishes a 1-1 correspondence between DOs $\square$ from $L$ to $L$ and one-parameter groups of local line bundle automorphisms $\{\Phi_\epsilon\}_{\epsilon\in\bbR}$ of $L\to M$.
\end{proposition}

The space of DOs from $L$ to $L$ is denoted by $\calD L$ and it can be identified with the space of sections of the vector bundle $DL:=(J^1L)^\ast\otimes L$, i.e.~the space of vector bundle morphisms, over $\id_M$, from the first jet bundle $J^1L$ to $L$.
In particular, for any $x\in M$, the corresponding fiber of $DL$, denoted by $D_xL$, consists of $\bbR$-linear maps $\delta:\Gamma(L)\to L_x$ s.t.~for some (necessarily unique) $\xi\in T_xM$ the following Leibniz rule holds
\begin{equation}
\delta(f\lambda)=\xi(f)\lambda_x+f(x)\delta\lambda,
\end{equation}
for all $f\in C^\infty(M)$ and $\lambda\in\Gamma(L)$.
The vector bundle $DL\to M$ is an illustration of a Jacobi algebroid (cf., e.g., \cite[Definition 1.3]{tortorella2017phdthesis}) and it is called the \emph{gauge algebroid (or Atiyah algebroid) of $L$}.
Indeed, it is naturally equipped with a Lie algebroid structure $([-,-],\rho)$ and a representation $\nabla$ on the line bundle $L\to M$.
Specifically, we have:
\begin{equation}
\label{eq:gauge_algebroid:structure_maps}
[\square,\Delta]:=\square\circ\Delta-\Delta\circ\square,\quad \rho(\square)=\sigma_\square,\quad \nabla_\square=\square,
\end{equation}
for all $\square,\Delta\in\calD L$.
A regular line bundle morphism $\Phi:L_1\to L_2$ (cf.~Remark~\ref{rem:regularVBmorphisms}), over $\varphi:M_1\to M_2$, determines a vector bundle morphism $D\Phi:DL_1\to DL_2$, over $\varphi:M_1\to M_2$, by setting
\begin{equation}
(D\Phi)\delta:=\Phi_x\circ\delta\circ\Phi^\ast\in D_{\varphi(x)}L_2,
\end{equation}
for all $x\in M_1$ and $\delta\in D_xL_1$.
Additionally, $D\Phi:DL_1\to DL_2$ is a Jacobi algebroid morphism, i.e.~it is a Lie algebroid morphism and it is compatible with the representations (of $DL_i$ on $L_i$, for $i=1,2$), in the sense that $(D\Phi)\nabla_\delta=\nabla_{(D\Phi)\delta}$, for all $\delta\in DL_1$.
Further, the above constructions define a functor from the category of line bundles (with regular line bundle morphisms) to the category of Jacobi algebroids.

For any line bundle $L\to M$, out of the structures~\eqref{eq:gauge_algebroid:structure_maps} existing on $DL$, one can also construct the \emph{Schouten--Jacobi bracket} $\ldsb-,-\rdsb$ (see, e.g., \cite[Section~1.3.1]{tortorella2017phdthesis} for its explicit definition).
In particular, $\ldsb-,-\rdsb$ is a graded Lie bracket on the shifted graded space $(\calD^\bullet L)[1]$, where $\calD^\bullet L:=\Gamma(\wedge^\bullet(J^1L)^\ast\otimes L)$ is the graded space of skew-symmetric multi-DOs from $L$ to $L$.
Indeed, for any $k\in\bbN$, $\calD^kL$ consists of the  skew-symmetric multi-linear maps $\square:\Gamma(L)^{\times k}\to\Gamma(L)$ which are a DO in each entry separately.
In particular, $\calD^0L=\Gamma(L)$ and $\calD^1L=\calD L$.
In this paper, all the multi-DOs are assumed to be skew-symmetric.

\subsection{The Der-Complex}
\label{sec:der-complex}

For any line bundle $L\to M$, the associated \emph{der-complex} $(\Omega_L^\bullet,\rmd_D)$ is the Lie algebroid de Rham complex of $DL$ with coefficients in its representation $L$ (cf.~\cite{rubtsov1980dercomplex}).
Hence $\Omega_L^\bullet=\Gamma(\wedge^\bullet(DL)^\ast\otimes L)$ is 
the space of the so called \emph{$L$-valued Atiyah forms},
with, in particular, $\Omega_L^0=\Gamma(L)$ and $\Omega_L^1=\Gamma(J^1L)$, and $\rmd_D$, the so-called \emph{der-differential}, is given by the Chevalley--Eilenberg formula
\begin{equation*}
(d_D\eta)(\Delta_0,\ldots,\Delta_k)=\sum_i(-)^i\Delta_i(\eta(\Delta_0,\ldots,\widehat{\Delta_i},\ldots,\Delta_k))+\sum_{i<j}(-)^{i+j}\eta([\Delta_i,\Delta_j],\Delta_0,\ldots,\widehat{\Delta_i},\ldots,\widehat{\Delta_j},\ldots,\Delta_k)
\end{equation*}
for all $k\in\bbN$, $\eta\in\Omega_L^k$ and $\Delta_0,\ldots,\Delta_k\in\calD L$, where the ``hat'' denotes omission.
In particular, $d_D:\Omega^0_L\to\Omega_L^1$ coincides with the first jet prolongation $j^1:\Gamma(L)\to\Gamma(J^1L),\ \lambda\mapsto j^1\lambda$.
A regular line bundle morphism $\Phi:L_1\to L_2$, over $\varphi:M_1\to M_2$, determines a cochain morphism $\Phi^\ast:(\Omega^\bullet_{L_2},\rmd_D)\to(\Omega^\bullet_{L_1},\rmd_D)$ by setting
\begin{equation}
(\Phi^\ast\eta)_x(\delta_1,\ldots,\delta_k):=\Phi_x^{-1}(\eta_{\varphi(x)}((D\Phi)\delta_1,\ldots,(D\Phi)\delta_k)),
\end{equation}
for all $k\in\bbN$, $\eta\in\Omega_{L_2}^k$, $x\in M_1$, and $\delta_1,\ldots,\delta_k\in D_xL_1$.
As on any Lie algebroid de Rham complex, there exists on the der-complex $(\Omega_L^\bullet,\rmd_D)$ a well-defined \emph{Cartan calculus}.
In addition to the der-differential $\rmd_D:\Omega^\bullet_L\to\Omega^{\bullet+1}_L$, the structural operations of this calculus are given, for any $\square\in\calD L$, by the \emph{contraction} $\iota_\square:\Omega^\bullet_L\to\Omega^{\bullet-1}_L$ and the \emph{Lie derivative} $\calL_\square:\Omega^\bullet_L\to\Omega_L^\bullet$.
The structural relations of such a calculus are:
\begin{equation*}
[d_D,\iota_\square]=\calL_\square,\quad[\calL_\square,\calL_\Delta]=\calL_{[\square,\Delta]},\quad[\calL_\square,\iota_\Delta]=\iota_{[\square,\Delta]},\quad[\iota_\square,\iota_\Delta]=0,
\end{equation*}
for all $\square,\Delta\in\Gamma(DL)$, where, on the LHS, $[-,-]$ denotes the graded commutator.
Consequently, the der-complex $(\Omega_L^\bullet,\rmd_D)$ is acyclic with a contracting homotopy given by $\iota_{\mathbbm{1}}$, i.e.~$\calL_{\mathbbm{1}}=[\rmd_D,\iota_{\mathbbm{1}}]$ is the identity map on $\Omega_L^\bullet$.
Here by $\mathbbm{1}\in\calD L$ we denote the DO which acts like the identity map, i.e.~$\mathbbm{1}\lambda=\lambda$, for all $\lambda\in\Gamma(L)$.

\subsection{The Omni-Lie Algebroid}
\label{sec:Dirac-Jacobi}

Following~\cite{vitagliano2018djbundles}, we recall that, for any line bundle $L\to M$, the associated \emph{omni-Lie algebroid} is the vector bundle $\bbD L:=DL\oplus J^1L\to M$, with the standard projections denoted by $\operatorname{pr}_D:\bbD L\to DL$ and $\operatorname{pr}_J:\bbD L\to J^1 L$, which is further equipped with:
\begin{itemize}
	\item $\ldab-,-\rdab:\bbD L\underset{M}{\times}\bbD L\to L$, the $L$-valued non-degenerate symmetric product on $\bbD L$ defined by
	\begin{equation*}
	\ldab(\square,\alpha),(\Delta,\beta)\rdab=\iota_\square\beta+\iota_\Delta\alpha,\qquad\text{for all}\ \square,\Delta\in\Gamma(D L),\ \alpha,\beta\in\Gamma(J^1L),
	\end{equation*}
	%for all $\square,\Delta\in\calD L$ and $\alpha,\beta\in\Omega_L^1$,
	\item $\ldsb-,-\rdsb:\Gamma(\bbD L)\times\Gamma(\bbD L)\to\Gamma(L)$, the Dorfman-like bracket on $\Gamma(\bbD L)$ defined by
	\begin{equation*}
	\ldsb(\square,\alpha),(\Delta,\beta)\rdsb=([\square,\Delta],\calL_\square\beta-\iota_\Delta\rmd_D\alpha),\qquad\text{for all}\ \square,\Delta\in\Gamma(D L),\ \alpha,\beta\in\Gamma(J^1L).
	\end{equation*}
	%for all $\square,\Delta\in\calD L$ and $\alpha,\beta\in\Omega_L^1$.
\end{itemize}
An isomorphism of omni-Lie algebroids $\bbD L_1 \to \bbD L_2$ is a pair $(\Psi,\Phi)$ of vector bundle isomorphisms $\Psi:\bbD L_1\to\bbD L_2$ and $\Phi:L_1\to L_2$, covering the same diffeomorphism $\varphi:M_1\to M_2$, such that
\begin{equation*}
\operatorname{pr}_D\circ\Psi=D\Phi\circ\operatorname{pr}_D,\quad \Psi^\ast\ldsb w_1,w_2\rdsb=\ldsb\Psi^\ast w_1,\Psi^\ast w_2\rdsb,\quad\Phi^\ast\ldab w_1,w_2\rdab=\ldab\Psi^\ast w_1,\Psi^\ast w_2\rdab,
\end{equation*}
for all $w_1,w_2\in\Gamma(\bbD L)$.
Let us just recall that, in particular, each line bundle isomorphism $\Phi:L_1\to L_2$, covering $\varphi:M_1\to M_2$, determines the isomorphism $(\bbD\Phi,\Phi)$ of omni-Lie algebroids $\bbD L_1\to\bbD L_2$, where
	\begin{equation*}
	\bbD\Phi:\bbD L_1\to\bbD L_2,\ (\delta,\alpha)\mapsto((D\Phi)\delta,\Phi\circ\alpha\circ(D\Phi)^{-1}).
	\end{equation*}
	Each closed $L$-valued Atiyah $2$-form $\varpi$, i.e.~$\varpi\in\Omega^2_L$ with $\rmd_D\varpi=0$, determines the automorphism $(\calR_\varpi,\id_L)$ of the omni-Lie algebroid $\bbD L$, where $\calR_\varpi$ is called the \emph{gauge transformation} by $\varpi$ and it is defined as follows
	\begin{equation*}
	\calR_\varpi:\bbD L\to\bbD L,\ (\delta,\alpha)\mapsto(\delta,\alpha+\varpi^\flat(\delta)).
	\end{equation*}
Actually these two kinds of automorphisms generate the entire automorphism group of the omni-Lie algebroid $\bbD L$ (cf.~\cite[Lemma 2.6]{schnitzer2019normal}).

A \emph{Dirac--Jacobi structure} on $L\to M$ consists of a vector subbundle $\calL\subset\bbD L$ which is both \emph{Lagrangian} w.r.t.~$\ldab-,-\rdab$, i.e.~$\calL=\calL^{\perp}$, and \emph{involutive} w.r.t.~$\ldsb-,-\rdsb$, i.e.~$\ldsb\Gamma(\calL),\Gamma(\calL)\rdsb\subset\Gamma(\calL)$.
Notice that, if $\calL\subset\bbD L$ is a Dirac--Jacobi structure, then also its \emph{opposite} $-\calL:=\{(\delta,-\alpha):(\delta,\alpha)\in\calL\}\subset\bbD L$ is a Dirac--Jacobi structure.

For the reader's convenience we conclude this section recalling that a regular line bundle morphism $\Phi:L\to L'$, covering a $\varphi:M\to M'$, gives rise to the operation of \emph{pull-back} and \emph{push-forward}.
For any linear family $\calL\subset\bbD L$, the \emph{push-forward of $\calL$ along $\Phi$} is the linear family $\Phi_!\calL\subset\varphi^\ast(\bbD L')$ such that
\begin{equation*}
	(\Phi_!\calL)_x=\{((D_x\Phi)\delta,\alpha')\in(\bbD L')_{\varphi(x)}:(\delta,\Phi_x^{-1}\circ\alpha'\circ D_x\Phi)\in\calL_x\},\quad\text{for any}\ x\in M.
\end{equation*}
For any linear family $\calL'\subset\bbD L'$, the \emph{pull-back of $\calL'$ along $\Phi$} is the linear family $\Phi^!\calL'\subset\bbD L$ such that
\begin{equation*}
	(\Phi^!\calL')_x=\{(\delta,\Phi_x^{-1}\circ\alpha'\circ D_x\Phi)\in(\bbD L)_x:((D_x\Phi)\delta,\alpha')\in\calL'_{\varphi(x)}\},\quad\text{for any}\ x\in M.
\end{equation*}

\section{Symplectization/Poissonization of contact/Jacobi structures}
\label{sec:homogeneous_symplectic/Poisson}

Closely following~\cite{bruce2017remarks}, this Appendix summarizes an alternative, but equivalent, approach to contact and Jacobi geometry.
This approach, inspired by the ``symplectization/Poissonization trick'', plays a crucial role in Sections~\ref{sec:Homogeneization_CDP} and~\ref{sec:Lie_group_actions}.
Specifically, in this appendix, we first introduce, in Definition~\ref{def:homogeneous_symplectic/Poisson}, the category of homogeneous symplectic (resp.~Poisson) manifolds and then we describe, in Proposition~\ref{prop:equivalence_categories:2}, the equivalence of categories existing between the latter and the category of contact (resp.~Jacobi) manifolds.

Let $P$ be a principal $\bbR^\times$-bundle, with action $\rmh:\bbR^\times\times P\to P$, $(t,p)\mapsto \rmh_t(p)$, quotient $M:=P/\bbR^\times$ and bundle map $\pi:P\to M$.
One can introduce the following spaces:
\begin{itemize}
	\item the $C^\infty(M)$-submodule $C^\infty_\text{lin}(P)\subset C^\infty(P)$ of those functions $f$ on $P$ which are degree $1$ homogeneous, i.e.~$\rmh_t^\ast f=tf$, for all $t\in\bbR^\times$,
	\item $\frakX^\bullet_\text{hom}(P)\subset\frakX^\bullet(P)$, the $C^\infty(M)$-submodule and Lie subalgebra formed by the \emph{homogeneous multivector fields}, i.e., for any $k\in\bbN$, by those $X\in\frakX^k(P)$ such that $\rmh_t^\ast X=t^{1-k}X$, for all $t\in\bbR^\times$,
	\item $\Omega^\bullet_\text{hom}(P)\subset\Omega^\bullet(P)$, the $C^\infty(M)$-submodule and cochain subcomplex formed by the \emph{homogeneous differential forms}, i.e.~by those $\omega\in\Omega^\bullet(P)$ such that $\rmh_t^\ast\omega=t\omega$, for all $t\in\bbR^\times$.
\end{itemize}
In particular, a Poisson/symplectic structure on $P$ is called \emph{homogeneous} if it is given by a homogeneous Poisson bivector/symplectic form.
So, one can introduce the homogeneous Poisson/symplectic category.

\begin{definition}
	\label{def:homogeneous_symplectic/Poisson}
	A \emph{homogeneous symplectic} (resp.~\emph{Poisson}) \emph{manifold} is a principal $\bbR^\times$-bundle equipped with a homogeneous symplectic (resp.~Poisson) structure.
	A \emph{homogeneous symplectomorphism} (resp.~\emph{Poisson map}) is a $\bbR^\times$-equivariant symplectomorphism (resp.~Poisson map).
\end{definition}

For our aims it is convenient to recall the equivalence of categories existing between: (1) the category of line bundles, with regular line bundle morphisms, and (2) the category of principal $\bbR^\times$-bundles, with $\bbR^\times$-equivariant maps.

In one direction (1) $\Rightarrow$ (2) there exists the \emph{homogenization functor}~\cite{vitagliano2017holomorphic}.
For any line bundle $L\to M$, the associated principal $\bbR^\times$-bundle $\widetilde{L}\overset{\pi}{\to} M$ is given by the \emph{slit dual} of $L$, i.e.~$\widetilde{L}:=L^\ast\setminus M$, where we canonically identify $M$ with the image of the zero section of $L^\ast\to M$.
For any regular line bundle morphism $\Phi:L_1\to L_2$, the associated $\bbR^\times$-equivariant map $\widetilde{\Phi}:\widetilde{L}_1\to \widetilde{L}_2$ is given by the fiberwise adjoint of $\Phi$, i.e.~$\widetilde{\Phi}(\nu_x):=\nu_x\circ\Phi_x^{-1}$, for all $x\in M_1$ and $\nu_x\in\widetilde{L}_{1,x}$.

In the other direction (2) $\Rightarrow$ (1) there exists the \emph{de\-ho\-mog\-e\-ni\-za\-tion functor}~\cite{vitagliano2017holomorphic}.
For any principal $\bbR^\times$-bundle $P\overset{\pi}{\to}M$, the associated line bundle $\widetilde{P}\to M$ is given by the dual of the tautological bundle, i.e.~$\widetilde{P}:=\bigO_{P/\bbR^\times}(1)\to M$.
Equivalently, $\widetilde{P}$ is given by the associated bundle for the action $\bbR^\times\times\bbR\to\bbR$, $(t,s)\mapsto t^{-1}s$, i.e.~$\widetilde{P}=(P\times\bbR)/\bbR^\times$.
For any $\bbR^\times$-equivariant map $\Psi:P_1\to P_2$, the associated regular line bundle morphism $\widetilde{\Psi}:\widetilde{P}_1\to\widetilde{P}_2$, is given by $\widetilde{\Psi}[(p,s)]=[(\Psi(p),s)]$, for all $(p,s)\in P_1\times\bbR$.

For any line bundle $L\to M$ and principal $\bbR^\times$-bundle $P$, there exist a canonical line bundle isomorphism $L\simeq\bigO_{\widetilde{L}/\bbR^\times}(1)$ and a canonical $\bbR^\times$-equivariant diffeomorphism $P\simeq\widetilde{\bigO_{P/\bbR^\times}(1)}$.
The latter defines natural transformations between the identity functors and the compositions of homogenization and dehomogenization functors.
This leads to the following.

\begin{proposition}
	\label{prop:equivalence_categories:1}
	The homogenization and dehomogenization functors establish an equivalence of categories between the category of line bundles and the category of principal $\bbR^\times$-bundles.
\end{proposition}

Let $L\to M$ be a line bundle.
If $\pi$ denotes the bundle map $\widetilde{L}:=L^\ast\setminus M\to M$, then the pull-back line bundle $\pi^\ast L\to\widetilde{L}$ is trivial.
Indeed, there exists a regular line bundle morphism $\widehat{\pi}:\bbR_{\widetilde{L}}\to L$, covering $\pi:\widetilde{L}\to M$, which is defined by $\nu(\widehat{\pi}(\nu,s))=s$, for all $\nu\in\widetilde{L}$ and $s\in\bbR$.
Equivalently, $\widehat{\pi}$ is determined by the fact that, by pull-back, it induces the $C^\infty(M)$-module isomorphism $\Gamma(L)\to C^\infty_{\text{hom}}(\widetilde{L})$, $\lambda\mapsto\widetilde{\lambda}:=\widehat{\pi}^\ast\lambda$, such that, for all $x\in M$, $\nu_x\in\widetilde{L}_x$, and $\lambda\in\Gamma(L)$,
\begin{equation}
\label{eq:homogeneization:1}
\widetilde{\lambda}(\nu_x)=\nu_x(\lambda_x).
\end{equation}
Further, there exist regular vector bundle morphisms, both denoted by $\widehat{\pi}$ and covering $\pi:\widetilde{L}\to M$,
\begin{equation}
\begin{tikzcd}
\wedge^\bullet T\widetilde{L}\arrow[d, swap, "\widehat{\pi}"]\arrow[r]&\widetilde{L}\arrow[d, "\pi"]\\
D^\bullet L:=\wedge^\bullet(J^1L)^\ast\otimes L\arrow[r]& M
\end{tikzcd}
\qquad\qquad
\begin{tikzcd}
\wedge^\bullet T^\ast\widetilde{L}\arrow[d, swap, "\widehat{\pi}"]\arrow[r]&\widetilde{L}\arrow[d, "\pi"]\\
\wedge^\bullet(DL)^\ast\otimes L\arrow[r]& M
\end{tikzcd}
\end{equation}
which are determined by the following properties.
The one on the LHS induces by pull-back the $C^\infty(M)$-module and Lie algebra isomorphism $(\calD^\bullet L)[1]\to\frakX_{\text{hom}}^\bullet(\widetilde{L})[1]$, $\square\mapsto\widetilde{\square}:=\widehat{\pi}^\ast\square$.
The one on the RHS induces by pull-back the $C^\infty(M)$-module and cochain isomorphism $\Omega^\bullet_L\to\Omega^\bullet_{\text{hom}}(\widetilde{L})$, $\varpi\mapsto\widetilde{\varpi}:=\widehat{\pi}^\ast\varpi$, s.~t.
\begin{equation}
\label{eq:homogeneization:2}
\widetilde{\square(j^1\lambda_1,\ldots,j^1\lambda_k)}=\widetilde{\square}(\rmd\widetilde{\lambda}_1,\ldots,\rmd\widetilde{\lambda}_k),\qquad\qquad\widetilde{\varpi(\delta_1,\ldots,\delta_k)}=\widetilde{\varpi}(\widetilde{\delta}_1,\ldots,\widetilde{\delta}_k),
\end{equation}
for all $k\in\bbN$, $\square\in\calD^kL$, $\varpi\in\Omega_L^k$, and $\lambda_1,\ldots,\lambda_k\in\Gamma(L)$, $\delta_1,\ldots,\delta_k\in\calD L$.

\begin{definition}
	\label{def:Poissonization}
	The \emph{Poissonization} of a Jacobi manifold $(M,L,\{-,-\})$ is the homogeneous Poisson manifold $(\widetilde{L}=L^*\setminus M,\{-,-\}_{\widetilde{L}})$ such that, if $\J\in\calD^2L$ is the Jacobi bi-DO corresponding to $\{-,-\}$ (cf.~Proposition~\ref{prop:Jacobi_bi-DOs}), then $\widetilde{\J}\in\frakX_{\text{hom}}^2(\widetilde{L})$ is the homogeneous Poisson bivector field corresponding to $\{-,-\}_{\widetilde{L}}$.
\end{definition}

%\begin{remark}
%\label{rem:Poissonization}
Unravelling Definition~\ref{def:Poissonization}, $\{-,-\}_{\widetilde{L}}$ is equivalently characterized as the unique (homogeneous) Poisson structure on $\widetilde{L}$ such that $\widehat{\pi}:(\widetilde{L},\bbR_{\widetilde{L}},\{-,-\}_{\widetilde{L}})\to(M,L,\{-,-\})$ is a Jacobi morphism, i.e.~$$\{\widetilde{\lambda},\widetilde{\mu}\}_{\widetilde{L}}=\widetilde{\{\lambda,\mu\}},\ \text{for all}\ \lambda,\mu\in\Gamma(L).$$
Further, for any regular line bundle morphism $\Phi:L_1\to L_2$, it turns out that $\Phi:(M_1,L_1,\{-,-\}_1)\to(M_2,L_2,\{-,-\}_2)$ is a Jacobi morphism if and only if $\widetilde{\Phi}:\widetilde{L}_1\to\widetilde{L}_2$ is a (homogeneous) Poisson map.
%\end{remark}

\begin{definition}
	\label{def:symplectization}
	The \emph{symplectization} of a contact manifold $(M,L,\theta)$ is the homogeneous symplectic manifold $(\widetilde{L},\widetilde{\varpi})$, where $\varpi\in\Omega_L^2$ is the symplectic Atiyah form corresponding to $\theta$ (cf.~Proposition~\ref{prop:contact=symplecticAtiyah}).
\end{definition}

%\begin{remark}
%\label{rem:symplectization}
Unravelling Definition~\ref{def:symplectization}, symplectization $\widetilde{\varpi}$, seen as a non-degenerate Poisson structure, coincides with the Poissonization of $\varpi$, seen as a non-degenerate Jacobi structure.
Further, since canonically $L\simeq TM/\calH$, with $\calH=\ker\theta$, one identifies $\widetilde{L}$ with symplectic submanifold $\calH^\circ\setminus M$ of $(T^\ast M,\omega_M)$, where $\calH^\circ$ denotes the annihilator of $\calH$ in $T^\ast M$.
Under this identification, it is easy to see, e.g.~in local coordinates, that the symplectic form induced by $\omega_M$ on $\calH^\circ\setminus M$ coincides with $\widetilde{\varpi}$. 
%\end{remark}
In conclusion, we obtain the following.

\begin{proposition}
	\label{prop:equivalence_categories:2}
	The homogenization and dehomogenization functors define an equivalence of categories between the Jacobi/contact category and the homogeneous Poisson/symplectic category. 
\end{proposition}

For future reference, we point out that Proposition~\ref{prop:equivalence_categories:2} specializes to the setting of Lie groupoids.
\begin{definition}
	\label{def:homogeneous_symplectic_groupoid}
	A \emph{homogeneous symplectic groupoid}~\cite{bruce2017remarks} is a symplectic groupoid $(\calG,\omega)\rightrightarrows\calG_0$ equipped with a principal $\bbR^\times$-bundle structure s.~t.~$\bbR^\times$ acts on $\calG$ by groupoid morphisms and $\omega$ is homogeneous.
\end{definition}
Then the following corollary extends to the non-coorientable case, the analogous result first obtained for coorientable contact groupoids (see~\cite[Proposition 2.4]{crainic}).

\begin{corollary}[{\cite[Theorem 5.8]{bruce2017remarks}}]
	\label{cor:homogeneous_symplectic_groupoids}
	%Up to natural transformations, 
	Homogenization and dehomogenization establish a 1-1 correspondence between contact groupoids and homogeneous symplectic groupoids.
\end{corollary}

\small
\addtocontents{toc}{\SkipTocEntry}
\section*{Acknowledgements}

We are grateful to 
David Iglesias Ponte,
Bozidar Jovanovi\'c,
Juan Carlos Marrero,
Jonas Schnitzer,
Daniele Sepe,
Izu Vaisman, 
Luca Vitagliano and
Marco Zambon
for useful discussions and helpful suggestions.
A.~M.~B.~and C.~V.~authors were partially supported by  CNCS UEFISCDI, 
project number PN-II-ID-PCE-2011-3-0921.
A.~G.~T.~is supported by an FWO postdoctoral fellowship and acknowledges the partial support received by the FWO research project G083118N (Belgium) during the preparation of this paper, further he is  member of GNSAGA (INdAM).
This study was financed in part by the Coordenação de Aperfeiçoamento de Pessoal de Nìvel Superior - Brasil (CAPES) - Finance code 001

%TODO Insert the most relevant references on contact and Jacobi geometry

\nocite{cabrera2018local,SaSe,schnitzer2019generalizedcontactbundles}

\end{document}